\documentclass[reqno]{amsart}
\usepackage{amsmath}
\usepackage{amsfonts,latexsym,enumerate}
\usepackage{amscd}
\usepackage{float,amsmath,amssymb,mathrsfs,bm,multirow,graphics}
\usepackage[dvips]{graphicx}
\usepackage{overpic}
\usepackage{amsaddr}
\usepackage[numbers,sort&compress]{natbib}
\usepackage{tikz}
\usepackage{epstopdf}
\usepackage{subfigure}
\usepackage{enumerate}
\usepackage{pgfplots}
\usepackage[a4paper, top=2cm, bottom=2cm, left=2.5cm, right=2.5cm]{geometry}
\usetikzlibrary{decorations.markings}

\newcommand{\R}{{\mathbb{R}}}

\newcommand{\CC}{{\Bbb C}}

\newcommand{\Z}{{\Bbb Z}}

\newcommand{\dbar}{\bar{\partial}}

\newtheorem{theorem}{Theorem}[section]
\newtheorem{proposition}[theorem]{Proposition}
\newtheorem{lemma}[theorem]{Lemma}

\newtheorem{assumption}[theorem]{Assumption}
\newtheorem{remark}[theorem]{Remark}

\newtheorem{RHproblem}[theorem]{RH problem}
\newtheorem{RHdbarproblem}[theorem]{RH $\dbar$-problem}
\newtheorem{dbarproblem}[theorem]{$\dbar$-problem}
\numberwithin{equation}{section}

\begin{document}
	\title[Painlevé asymptotics of Boussinesq equation]
	{Long-time asymptotics of the good Boussinesq equation and its modified version: Painlev\'{e} region}

	\author{Deng-Shan Wang$^{\dagger}$ and Xiaodong Zhu$^{\dagger,\ddagger}$}
	\address{ $^{\dagger}$Laboratory of Mathematics and Complex Systems (Ministry of Education),  \\ School of Mathematical Sciences,
		Beijing Normal University,Beijing 100875, China.}
        
        \address{$^{\ddagger}$ SISSA, via Bonomea 265, 34136 Trieste, Italy, INFN Sezione di Trieste.} 
	\email{wangds1980@163.com, xdzbnu@mail.bnu.edu.cn}

	
	\date{\today}
	
	
	
\begin{abstract} Under suitable solitonless assumptions on the associated scattering
data, we rigorously establish the long-time asymptotics of the
modified Boussinesq equation in the Painlev\'e region and its
adjacent transition regions for initial data in a weighted Sobolev
space. Using the $\dbar$-steepest descent method, we reduce the
associated $3\times3$ RH problem to local model problems and
derive uniform error estimates under these weaker regularity
assumptions. The leading terms in the Painlev\'e region and the
first transition region are characterized by the Clarkson--McLeod
solution of the Painlev\'e IV equation. We further investigate the Painlev\'e asymptotics
of the good Boussinesq equation through the Miura transformation.
Numerical simulations support the asymptotic formulas well.

\end{abstract}

\maketitle
	
\setcounter{tocdepth}{1}
\tableofcontents

\section{Introduction}
	
In 1834, J. S. Russell \cite{Russell-1844} first observed a remarkable phenomenon of water waves that maintained their shapes and velocities during propagation. However, due to the limitations of mathematical theories and scientific understanding at that time, people were unable to provide a satisfactory theoretical explanation for this kind of water waves. It was not until 1872 that J. V. Boussinesq \cite{Boussinesq-1872}, when studying the shallow water waves, proposed a model to describe a two-dimensional, irrotational, and inviscid fluid flow in a uniform rectangular channel with a flat bottom. His pioneering work \cite{Boussinesq-1872} provided the first comprehensive explanation for the traveling wave solutions discovered by J. S. Russell, described wave propagation in two directions, and finally offered a theoretical framework for understanding solitary waves. In a first order approximation, Boussinesq’s model reduces to a scalar, fourth order partial differential equation, which in its original form is
\begin{equation*}
 \frac{\partial^2 h}{\partial t^2}=gH\frac{\partial^2 h}{\partial x^2} + gH\frac{\partial^2 }{\partial x^2}\bigg(\frac{3}{2}\frac{h^2}{H}+\frac{H^2}{3}\frac{\partial^2 h}{\partial x^2}\bigg),
\end{equation*}
where $h=h(x,t)$ represents the free surface of the perturbation wave, $H$ is the average depth of the water, and $g$ denotes the acceleration due to gravity. A scaling transformation is applied to the independent variables and the unknown function $h$ as
$$
t\to \sqrt{\frac{H}{g}}t,\quad x \to Hx,\quad h \to -\frac{4}{9}Hu(x,t),
$$
which leads to the dimensionless Boussinesq equation
\begin{align}\label{badboussinesq}
  u_{tt} - u_{xx} +\frac{2}{3} (u^2)_{xx} - \frac{1}{3}u_{xxxx} = 0.
\end{align}
This equation is now known as the bad Boussinesq equation which is completely integrable and has a self-adjoint Lax operator\cite{Charlier-Lenells-Boussinesq}. It is linearly unstable and ill-posed owing to the exponential growth of its Fourier components with frequencies large enough. The most recent results on the blow-up of the bad Boussinesq equation can be found in Charlier \cite{Charlier-Blowup}. Over the past years, significant efforts have been made to study the bad Boussinesq equation (\ref{badboussinesq}). Zakharov \cite{Zakharov-1973} first provided the Lax representation of equation (\ref{badboussinesq}), confirmed the existence of an infinite number of involutive integrals of the motion and reinterpreted the numerical experiments of Fermi, Pasta, and Ulam from a novel perspective. Subsequently, Deift, Tomei and Trubowitz \cite{Deift-1982} developed the inverse scattering technique to explore the  Schwartz class solutions of equation (\ref{badboussinesq}). In the same year,  Caudrey \cite{Caudrey-1982} employed a special case of spectral transformation to derive the $N$-soliton solution of equation (\ref{badboussinesq}). Most recently, Charlier and Lenells constructed the Riemann-Hilbert problem to investigate the long-time asymptotics \cite{Charlier-Lenells-1,Charlier-Lenells-2} and soliton resolution conjecture \cite{Charlier-Lenells-3,Charlier-Lenells-4} of equation (\ref{badboussinesq}) based on the Deift-Zhou nonlinear steepest descent method \cite{DZ1993}. These studies have deepened our understanding of the bad Boussinesq equation. 
\par
Reversing the sign of the fourth derivative term in the bad Boussinesq equation (\ref{badboussinesq}) gives the good Boussinesq equation
\begin{align}\label{good-boussinesq}
  u_{tt} - u_{xx} + \frac{2}{3}(u^2)_{xx} + \frac{1}{3} u_{xxxx} = 0,
\end{align} 
which can be written in two-component form by taking $u_t=v_x$:
\begin{align}\label{good-boussinesq-0}
  v_{t} - u_{x} +\frac{2}{3} (u^2)_{x} + \frac{1}{3}u_{xxx} = 0.
\end{align} 
The good Boussinesq equation is also physically significant. For example, it can describe the oscillations in nonlinear strings \cite{Falkovich-1983} and can also be derived from the lattice dynamics for a chain of atoms in elastic crystals \cite{Maugin-1999}. Taking the simple shift $u\to 2 u+{3}/{4}$ and then rescaling the dependent and independent variables, the term $u_{xx}$ in equation (\ref{good-boussinesq}) can be removed and another form of good Boussinesq equation is obtained \cite{Charlier-Lenells-2021,Wang-APDE}
\begin{equation}\label{good-boussinesq-1}
	u_{tt}+\frac{4}{3}(u^2)_{xx}+\frac{1}{3}u_{xxxx}=0.
\end{equation}
For $u_t = w_x$, (\ref{good-boussinesq-1}) can also be rewritten as
\begin{equation}\label{good-boussinesq-2} 
\begin{cases}
    \begin{aligned}
        &u_t=w_x,\\
        &w_t  + \frac{4}{3}(u^2)_x + \frac{1}{3}u_{xxx}= 0.
    \end{aligned}
\end{cases}
\end{equation}
{The well-posedness theory of the Cauchy problem for the good Boussinesq equation \eqref{good-boussinesq} and related Boussinesq-type equations has been studied under various assumptions on the initial data. Here, we denote the initial data by $u_0=u(x,0)$ and $u_1=u_t(x,0)$. Bona and Sachs \cite{Bona-1988} established the local well-posedness for a class of Boussinesq-type equations with initial data $(u_0,u_1)\in H^s(\mathbb{R})\times H^{s-2}(\mathbb{R})$ for $s>\frac52$. Tsutsumi and Matahashi \cite{Tsutsumi-Matahashi} obtained local and global well-posedness results under weaker assumptions $(u_0,u_1)\in H^1(\mathbb{R})\times H^{-1}(\mathbb{R})$.  Subsequently, Farah \cite{Farah-2009} proved local well-posedness for initial data $(u_0,u_1)\in H^s(\mathbb{R})\times H^{s-1}(\mathbb{R})$ with $s>-\frac14$. These results provide an important foundation for the study of the Cauchy problem and motivate further investigations of the long-time behavior of solutions to the good Boussinesq equation.} From the viewpoint of integrable systems, Kaup \cite{Kaup-1980} considered the inverse scattering problem for the
cubic eigenvalue problems, in which the Boussinesq equation is a typical example. McKean \cite{McKean-1981,McKean-1981-Physica-D} studied the good Boussinesq equation (\ref{good-boussinesq-1}) on the circle to explain how Krichever's construction encompasses all solutions in the space $C_{1}^{\infty}$, provided that curves of infinite genus are considered. In recent years, Charlier and Lenells \cite{Charlier-Lenells-2021} constructed the Riemann-Hilbert problem of equation (\ref{good-boussinesq-1}) by conducting a spectral analysis on its associated Lax pair. Subsequently, Charlier, Lenells and one of the present authors \cite{Wang-APDE}, and the present authors \cite{WangJMP} investigated the long-time asymptotics of the solution to the equation (\ref{good-boussinesq-1}) and equation (\ref{good-boussinesq-0}) in the dispersive wave region, respectively. Nevertheless, these studies still fall short of fully solving the open problem concerning the long-time asymptotics of the good Boussinesq equation with a non-self-adjoint third-order Lax operator, as proposed by Deift \cite{Deift-2008}, because the long-time behavior of the solution near the positive $t$-axis remains open now \cite{Wang-APDE,WangJMP}. The present work aims to resolve this open problem and provides a clear understanding of the middle regions where $x/t$ is near zero in the upper $(x,t)$-plane.   
\par
Similar to the relationship between the KdV equation and the modified KdV (mKdV) equation by Miura transformation \cite{Miura-1968}, the good Boussinesq equations (\ref{good-boussinesq-0}) and (\ref{good-boussinesq-2}) are related with the modified Boussinesq equation \cite{Fordy-1981,Wang-Zhu-Zhu-2025}  
    \begin{equation}
        \begin{cases}
            \begin{aligned}
			p_t & = 2(p q)_x + q_{xx}, \\
			q_t & = -\frac{1}{3} p_{xx} + \frac{2}{3} p p_x - 2 q q_x,
		\end{aligned}
        \end{cases}
        \label{mbequation}
    \end{equation}
via the Miura transformations    
    \begin{equation*}
    \begin{cases}
        \begin{aligned}
		& u =\frac{3}{4} -(2p_x +3q^2+ p^2), \\
		& v = -2(q_{xx} +3 p q_x +q p_x +2 q (p^2-q^2)),
	\end{aligned}
    \end{cases}
    \end{equation*}
and 
\begin{equation}
    \begin{cases}
        \begin{aligned}
		& u = -p_x - \frac{1}{2} p^2 - \frac{3}{2} q^2, \\
		& w = -q_{xx} - 3 p q_x - q p_x - 2 q p^2 + 2 q^3,
	\end{aligned}
    \end{cases}
    \label{Miuratranf}
    \end{equation}
respectively.
\par
By deriving a new similarity reduction, Clarkson \cite{Clarkson-1989} successfully established the connection between the solution of the modified Boussinesq equation (\ref{mbequation}) and the solution $P_{\rm IV}(y)$ of the Painlev\'e IV equation
\begin{equation}       
\frac{\mathrm{d}^2 P_{\rm IV}}{\mathrm{d} y^2} = \frac{1}{2 P_{\rm IV}} \left( \frac{\mathrm{d} P_{\rm IV}}{\mathrm{d} y} \right)^2 + \frac{3}{2} P_{\rm IV}^3 + 4 y P_{\rm IV}^2 + \left( 2 y^2 - 4 \alpha + \beta \right) P_{\rm IV} - \frac{\beta^2}{2 P_{\rm IV}},  
\label{PIV}  
\end{equation} 
where $\alpha$ and $\beta$ are constants.
Subsequently, Clarkson and Kruskal \cite{Clarkson-Kruskal-1989} further reduced the Boussinesq equation into the Painlev\'e IV equation by developing a direct method. So it is believed that the modified Boussinesq equation and the good Boussinesq equation have a very close connection with the Painlev\'e IV equation (\ref{PIV}). In detail, consider the region near the upper half of the $t$-axis in the $(x,t)$-plane; see Figure~\ref{Category}. Let
\(
y = -\frac{\sqrt{3}\, x}{2\sqrt{t}}.
\)
As $t \to \infty$, the solution to the modified Boussinesq equation exhibits the following asymptotic behavior:
\[
\begin{cases}
\begin{aligned}
& p(x,t) \sim \frac{\sqrt{3}}{4\sqrt{t}} \cdot \frac{P_{\rm IV}'(y) + \frac{2}{3}}{P_{\rm IV}(y)}, \\
& q(x,t) \sim \frac{\sqrt{3}}{4\sqrt{t}} \left( P_{\rm IV}(y) + \frac{2}{3} y \right),
\end{aligned}
\end{cases}
\]
where $\alpha = -\frac{1}{6}$ and $\beta = -\frac{2}{3}$. Moreover, applying the Miura transformation given in (\ref{Miuratranf}), one obtains the following asymptotic behavior for the solution of the good Boussinesq equation:
\[
u(x,t) \sim -\frac{\left(3 P'_{\rm IV}(y) + 2\right)^2 - 9 P_{\rm IV}^4(y) - 36 y P_{\rm IV}^3(y) - 20 y^2 P_{\rm IV}^2(y)}{32 t\, P_{\rm IV}^2(y)}.
\]

The asymptotic analysis of Painlev\'e IV equation subject to the boundary condition $P_{\rm IV}(y) \to 0$ as $y \to +\infty$ was pioneered by Clarkson and McLeod \cite{Clarkson-Mcleod,ItsJPA}. More recently, Xia, Xu and Zhao \cite{Xia-Xu-Zhao-2023} studied the Clarkson-McLeod solutions of Painlev\'e IV equation (\ref{PIV}) by using Deift-Zhou nonlinear steepest descent method \cite{DZ1993}. 
\par
Inspired by the observations that the long-time asymptotic behaviors of the mKdV equation \cite{DZ1993} and KdV equation \cite{Deift-Zhou-1994} in the region $|x| \leq M t^{1/3}~(M > 0)$ are connected to the solution of the Painlev\'e $\rm II$ equation, this work applies the $\dbar$-method to the $3\times3$ Riemann-Hilbert problem derived in ~\cite{Charlier-Lenells-2021, WangJMP} to formulate the long-time asymptotic behaviors of the modified Boussinesq equation and the good Boussinesq equation, respectively, {with particular emphasis on the Painlev\'e region.} Because the reflection coefficients of the good Boussinesq equation behave $|r_j(k)|\geq 1~(j=1,2)$ for small $k$, it is necessary to employ the Miura transformation to explore the long-time asymptotics of the good Boussinesq equation near the positive $t$-axis.   
\par
{In this work, we study the Cauchy problems for the good Boussinesq equation~\eqref{good-boussinesq-2} and the modified Boussinesq equation~\eqref{mbequation} with initial data in a weighted Sobolev space. Specifically, for the modified Boussinesq equation, we assume that $p(x,0)=p_0(x)$ and $q(x,0)=q_0(x)$ belong to the weighted Sobolev space $W^{2,1}_1(\mathbb R)$. The initial data $u(x,0)=u_0(x)$ and $w(x,0)=w_0(x)$ for the good Boussinesq equation are related to $(p_0(x),q_0(x))$ through the Miura transformation~\eqref{Miuratranf}, subject to the regularity and compatibility conditions specified below. Under the stated spectral assumptions, the direct scattering estimates provide the regularity and decay of the reflection coefficients needed for the $\dbar$-steepest descent analysis. We formulate a mixed RH $\dbar$-problem and separate it into a pure RH problem, which determines the leading-order asymptotics, and a pure $\dbar$ problem, which controls the contribution of the nonanalytic extensions. This approach allows us to derive long-time asymptotic formulas of integrable equations beyond the solutions of Schwartz class \cite{Wang-Zhu-Math-Ann,Wang-Zhu-SIAM}.
\par
The use of weighted Sobolev spaces is motivated by developments in scattering theory and the $\dbar$-steepest descent method. Bijection results for scattering maps have been established for nonlinear Schr\"odinger equations \cite{Zhou-1998,Liu-Perry-Sulem-2015} and for matrix-valued spectral problems, including the $3\times3$ AKNS system \cite{Liu-2019}. The $\dbar$-steepest descent method has also been applied to the focusing nonlinear Schr\"odinger equation \cite{Borghese-Jenkins-McLaughlin-2018}, nonlocal nonlinear Schr\"odinger equations \cite{Li-Yang-Fan-2025}, the massive Thirring model \cite{He-Liu-Qu-2026}, and the good Boussinesq equation \cite{Yang-2026}. In the present work, we use this framework to investigate the Painlev\'e asymptotics and the transition to dispersive waves. However, a vanishing lemma \cite{Zhou-SIAM} has not yet been established for the RH problem associated with the version of the Boussinesq equation considered here. This leaves a gap in establishing the bijectivity of the scattering map and identifying the optimal function spaces for the scattering theory. We leave these questions for future investigation.}
\par
The paper is organized as follows. Section~\ref{Main} presents the main results. Section~\ref{sec:direct} develops the direct scattering theory for the modified Boussinesq equation and establishes the regularity properties of the reflection coefficients $r_1(k)$ and $r_2(k)$. Section~\ref{sed:P4} analyzes the Painlev\'e region by using the $\dbar$-steepest descent method. Section~\ref{Transition1} examines the adjacent transition region, whose asymptotics is also governed by Painlev\'e IV. Section~\ref{sec:trans2} investigates the transition from Painlev\'e IV asymptotics to dispersive waves. Finally, Appendix~\ref{Appendix} formulates the Painlev\'e IV model RH problem used in the asymptotic analysis.

\subsection{\bf{Notations}}
    \begin{itemize}
    \item \(C\) and \(c\) represent the positive constants that may vary during calculations.  
    \item \([M]_i\) denotes the \(i\)-th column of the matrix \(M\), and \([M]_{i,j}\) represents the \((i,j)\)-th element of \(M\).  
    \item \(\omega = e^{\frac{2\pi i}{3}}\), and \(\Re \Phi\) denotes the real part of the function \(\Phi\).  
    \item \(D_{\epsilon}(k_0)\) denotes the open disk in the complex plane centered at \(k_0\) with radius \(\epsilon\), and \(\partial D_{\epsilon}(k_0)\) represents its boundary.  
    \item \(\mathcal{S}(\mathbb{R})\) denotes the Schwartz space of rapidly decreasing functions on \(\mathbb{R}\).
    \item \(r^{*}(k)\) denotes the Schwarz reflection of \(r(k)\).  
    \item Let $J$ be some diagonal matrix, then \(e^{\hat{J}}\) denotes an operator on the matrix $V$, defined by \( e^{\hat{J}}V = e^{J} V e^{-J} \). 
    \item The function $\log_0(k)$ denotes the logarithm with a branch cut along $\mathbb{R}_+$, and $\ln(x)$ represents the natural logarithm of real value. Specifically, for a nonzero complex number $k = |k| e^{i\arg(k)}$, define $\log_0(k)$ as  
$
\log_0(k) := \ln(|k|) + i\arg(k),
$  
where $0 \leq \arg(k) < 2\pi$.
    \item For two matrices $A$ and $B$, $[A,B]$ means that $[A,B]=AB-BA$.
    \item \(W^{2,1}_1(\mathbb R)\) denotes the weighted Sobolev space
\[
W^{2,1}_1(\mathbb R)
:=
\left\{
f\in W^{2,1}(\mathbb R):
\sum_{\ell=0}^{2}
\int_{\mathbb R}(1+|x|)
|\partial_x^\ell f(x)|\,\mathrm dx<\infty
\right\},
\]
equipped with the norm
\(
\|f\|_{W^{2,1}_1(\mathbb R)}
:=
\sum_{\ell=0}^{2}
\int_{\mathbb R}(1+|x|)
|\partial_x^\ell f(x)|\,\mathrm dx.
\)
\end{itemize}

\section{Main Results}\label{Main}

Before presenting the main results of this work, recall the RH problems corresponding to the good
Boussinesq equation (\ref{good-boussinesq-1}) and the modified Boussinesq equation (\ref{mbequation}). For convenience, adopt the conventional notation that the solution of the RH problem associated with the good Boussinesq equation is denoted by $M(x,t,k)$, while the solution associated with the modified Boussinesq equation is denoted by $m(x,t,k)$. Both RH problems share the same jump contour $\Sigma := \mathbb{R} \cup \omega \mathbb{R} \cup \omega^2 \mathbb{R}$, where $\omega = e^{\frac{2\pi i}{3}}$, as shown in Figure \ref{Sigma}.
    \begin{figure}[htp]
		\centering
		\begin{tikzpicture}[>=latex]
			\draw[<->,very thick] (-2,0)node[below]{${\Sigma_4}$} to (2,0)node[above]{${\Sigma_1}$};
			\draw[very thick] (-4,0) to (4,0)node[right]{$\R$};
			\draw[<->,very thick] (-1,-1.732)node[right=2mm]{${\Sigma_5}$} to (1,1.732)node[left=2mm]{${\Sigma_2}$};
			\draw[very thick] (-2,-1.732*2) to (2,1.732*2);
			\draw[<->,very thick] (-1,1.732)node[left]{${\Sigma_3}$} to (1,-1.732)node[right]{${\Sigma_6}$};
			\draw[very thick] (-2,1.732*2) to (2,-1.732*2);
			\node at (-.4,-.3){$0$};
			\node at (3,1.8){$D_{1}$};
			\node at (0,2.8){$D_{2}$};
			\node at (-3,1.8){$D_{3}$};
			\node at (-3,-1.8){$D_{4}$};
			\node at (0,-2.8){$D_{5}$};
			\node at (3,-1.8){$D_{6}$};
		\end{tikzpicture}
		\caption{The jump contour $\Sigma=\cup_{i=1}^{6}\Sigma_i$ and the six open sets $D_{j}$ for $j=1,2,\cdots,6$.}
		\label{Sigma}
	\end{figure}
    
   Let $\Sigma_j$ and $D_j~(j=1, 2,\cdots,6)$ be the oriented contours and open sets displayed in Figure \ref{Sigma}, respectively. Given two functions called reflection coefficients, i.e., $r_1(k):~(0,\infty)\to \mathbb{C}$ and $r_2(k):~(-\infty,0)\to \mathbb{C}$, define the jump matrices $v(x,t,k)$ on the contours $\Sigma_j~(j=1,2,\cdots,6)$ as follows:
   
    	\begin{equation}\label{jumps0}
		\begin{aligned}
			& v_1=\small{\begin{pmatrix}
				1 & -r_1(k) e^{- \vartheta_{21}(x,t; k)} & 0 \\
				r_1^*(k) e^{ \vartheta_{21}(x,t; k)} & 1-\left|r_1(k)\right|^2 & 0 \\
				0 & 0 & 1
			\end{pmatrix}}, \quad
                v_2=\small{\begin{pmatrix}
				1 & 0 & 0 \\
				0 & 1-\left|r_2(\omega k)\right|^2 & -r_2^*(\omega k) e^{-\vartheta_{32}(x,t; k)} \\
				0 & r_2(\omega k) e^{\vartheta_{32}(x,t; k)} & 1
			\end{pmatrix}}, \\
			&v_3=\small{\begin{pmatrix}
				1-\left|r_1\left(\omega^2 k\right)\right|^2 & 0 & r_1^*\left(\omega^2 k\right) e^{-\vartheta_{31}(x,t; k)} \\
				0 & 1 & 0 \\
				-r_1\left(\omega^2 k\right) e^{\vartheta_{31}(x,t; k)} & 0 & 1
			\end{pmatrix}},\\
                &v_4=\small{\begin{pmatrix}
				1-\left|r_2(k)\right|^2 & -r_2^*(k) e^{-\vartheta_{21}(x,t; k)} & 0 \\
				r_2(k) e^{\vartheta_{21}(x,t; k)} & 1 & 0 \\
				0 & 0 & 1
			\end{pmatrix}}, \quad
			v_5=\small{\begin{pmatrix}
				1 & 0 & 0 \\
				0 & 1 & -r_1(\omega k) e^{-\vartheta_{32}(x,t; k)} \\
				0 & r_1^*(\omega k) e^{\vartheta_{32}(x,t; k)} & 1-\left|r_1(\omega k)\right|^2
			\end{pmatrix}}, \\
               &v_6=\small{\begin{pmatrix}
				1 & 0& r_2\left(\omega^2 k\right) e^{-\vartheta_{31}(x,t; k)} \\
				0 & 1 & 0 \\
				-r_2^*\left(\omega^2 k\right) e^{\vartheta_{31}(x,t; k)} & 0 & 1-\left|r_2\left(\omega^2 k\right)\right|^2
			\end{pmatrix}},
		\end{aligned}
	\end{equation}
    where $\vartheta_{ij}=(\omega^i-\omega^j)kx+(\omega^{2i}-\omega^{2j})k^2t$ for $1\le j<i\le 3$.
    Now introduce the RH problem for $m(x,t,k)$ of the modified Boussinesq equation ($\ref{mbequation}$).

  \begin{RHproblem}\label{RHPm}
      Let $r_1(k)$ and $r_2(k)$ be certain functions defined on $(0,+\infty)$ and $(-\infty,0)$, respectively. Find a 3$\times$3 matrix-valued function $m(x,t,k)$ that satisfies the following properties:
      \begin{enumerate}
          \item[(i)] The function $m(x,t,k)$ is analytic for $k\in\mathbb{C}\setminus\Sigma$.
          \item[(ii)] As $k$ approaches the jump contour $\Sigma$ from left and right, the limits exist and satisfy the jump condition $m_+(x,t,k)=m_-(x,t,k)v(x,t,k)$, where $v(x,t,k)$ is defined in (\ref{jumps0}).
          \item[(iii)] As $k\to\infty$, the expansion of $m(x,t,k)$ admits that
          $$
          m(x,t,k)=I+\mathcal{O}\left(\frac{1}{k}\right).
          $$
          \item[(iv)] The function $m(x, t,k)$ follows the $\mathbb{Z}_2$ and $\mathbb{Z}_3$ symmetries
		\begin{equation*}
			\mathcal{A }m(x, t; \omega k) \mathcal{A}^{-1}=m(x, t,k)=\mathcal{B}\overline{ m(x, t; \bar{k})} \mathcal{B}, \quad k \in \mathbb{C} \backslash \Sigma,
            \end{equation*}
            where 
            $$
            \mathcal{A}:=\begin{pmatrix}
                0&0&1\\
                1&0&0\\
                0&1&0
            \end{pmatrix},\quad \mathcal{B}:=\begin{pmatrix}
                0&1&0\\
                1&0&0\\
                0&0&1
            \end{pmatrix}.
            $$
          \item[(v)] As $k\to0$ and $k\in\bar{D}_1$, there exist $\{\mathbf{m}^{(l)}(x,t)\}_{l=0}^{\infty}$ such that
          $$
          m(x,t,k)=\sum_{l=0}^N\mathbf{m}^{(l)}(x,t)k^l+\mathcal{O}(k^{N+1}).
          $$
      \end{enumerate}
  \end{RHproblem}

Throughout this paper, we assume that
$p_0,q_0\in W^{2,1}_1(\mathbb R)$.
In addition, we impose the following spectral assumptions.

\begin{assumption}\label{assumption}
The spectral problem associated with the modified Boussinesq
equation \eqref{mbequation} has no discrete spectrum.
Moreover, the scattering coefficients appearing in the
denominators in \eqref{def:r1r2} do not vanish on their
respective closed real half-lines.

Following \cite{WangJMP}, we further assume that
\(
|r_1(k)|<1
\)
for \(k\in\mathbb R_+\), and
\(
|r_2(k)|<1
\)
for \(k\in\mathbb R_-\).
Thus, RH problem~\ref{RHPm} contains no pole conditions,
and its solution $m(x,t,k)$ satisfies the analyticity
condition in item~(i).
\end{assumption}

Suppose $p_0(x),q_0(x)\in W^{2,1}_1(\mathbb R)$, which together with
Assumption~\ref{assumption}, yields the regularity of the
reflection coefficients required for the subsequent
$\dbar$-steepest descent analysis.

\begin{proposition}\label{prop:r1r2-H1}
Suppose that $p_0(x), q_0(x)\in W^{2,1}_1(\mathbb R)$
and that Assumption~\ref{assumption} holds, then the associated reflection coefficients satisfy
\[
r_1(k)\in H^1(\mathbb R_+),
\qquad
r_2(k)\in H^1(\mathbb R_-).
\]
\end{proposition}

\begin{proof}
The proof is established in Section \ref{sec:direct}.
\end{proof}

    \begin{proposition}
    Suppose the reflection coefficients $r_1(k)$ and $r_2(k)$ satisfy the Assumption \ref{assumption}, and assume the RH problem \ref{RHPm} has a solution for $(x,t)\in\R\times[0,\infty)$, then if the initial data \(p(x,0)=p_0(x)\) and \(q(x,0)=q_0(x)\) of the modified Boussinesq equation (\ref{mbequation}) satisfy Assumption \ref{assumption}, the solution $p(x, t)$ and $q(x, t)$ can be reconstructed from the RH problem \ref{RHPm} in the form 
        \begin{equation}\label{recovermb}
        \left\{
        \begin{array}{l}
        p(x, t) = \frac{3}{2} \mathop{{\rm lim}}\limits_{k \to \infty}k\left( [m(x, t,k)]_{31} + [ m(x, t,k)]_{32}\right), \\[5pt]
        q(x, t) = \frac{3}{2 \omega(\omega-1)} \mathop{{\rm lim}}\limits_{k \to \infty}k\left([ m(x, t,k)]_{31} -  [m(x, t,k)]_{32}\right),
        \end{array}
        \right.
        \end{equation}
        where \([m(x, t,k)]_{ij}\) denotes the \((i,j)\)-th entry in the matrix \(m(x, t;k)\).    
    \end{proposition}
    
    \subsection{Statement of the main theorem} The long-time asymptotics of the good Boussinesq equation
\eqref{good-boussinesq-1} and the modified Boussinesq equation
\eqref{mbequation} in the dispersive wave region have been established
in \cite{Wang-APDE,WangJMP}. In this work, we rigorously analyze
the Painlev\'e asymptotics of the modified Boussinesq equation
and the associated transition regions. We then use the Miura
transformation to investigate the corresponding asymptotic
behavior of the good Boussinesq equation in the Painlev\'e
region. We now state our main results.
    
\begin{theorem}\label{mainthm}
Let $(p(x,t),q(x,t))$ be a solution of the modified Boussinesq equation
\eqref{mbequation}, defined for $t\geq0$, with initial data
$p_0(x), q_0(x)\in W^{2,1}_1(\mathbb R)$ and
Assumption~\ref{assumption} holds.
Let $c_1,c_2,c_3>0$ be fixed constants, and define
\(
k_0=\frac{x}{2t},
\
\tau=tk_0^2,
\
y=-\frac{\sqrt3x}{2\sqrt t}.
\)
Let $P_{\mathrm{IV}}(y)$ denote the Clarkson--McLeod
solution of the Painlev\'e IV equation \eqref{PIV}, with parameters
$\alpha=-1/6$ and $\beta=-2/3$, determined by the
associated scattering data.
Then the following asymptotic formulas hold as
$t\to+\infty$, see Figure~\ref{Category}.

\begin{itemize}
\item
In the Painlev\'e region $0\leq x<c_1\sqrt t$,
\begin{equation}\label{mbresult}
\begin{aligned}
p(x,t)
&=
\frac{\sqrt3}{4\sqrt t}
\frac{P_{\mathrm{IV}}'(y)+\frac23}{P_{\mathrm{IV}}(y)}
+\mathcal O(t^{-3/4}),
\\
q(x,t)
&=
\frac{\sqrt3}{4\sqrt t}
\left(P_{\mathrm{IV}}(y)+\frac23y\right)
+\mathcal O(t^{-3/4}).
\end{aligned}
\end{equation}

\item
In the first transition region
$c_1\sqrt t\leq x< c_2t^{3/5}$,
the leading-order terms are still governed by the
$P_{\mathrm{IV}}$ function:
\begin{equation*}
\begin{aligned}
p(x,t)
&=
\frac{\sqrt3}{4\sqrt t}
\frac{P_{\mathrm{IV}}'(y)+\frac23}{P_{\mathrm{IV}}(y)}
+
\mathcal O\left(
t^{-3/4}+k_0^{1/2}t^{-1/2}+k_0^{3/2}
\right),
\\
q(x,t)
&=
\frac{\sqrt3}{4\sqrt t}
\left(P_{\mathrm{IV}}(y)+\frac23y\right)
+
\mathcal O\left(
t^{-3/4}+k_0^{1/2}t^{-1/2}+k_0^{3/2}
\right).
\end{aligned}
\end{equation*}

\item
In the second transition region
$c_2t^{3/5}\leq x<c_3t$,
the long-time behaviors of the solution are described by 
\begin{equation}\label{mbresult-3}
\scalebox{0.95}{$
\begin{aligned}
    p(x,t) &= \frac{3^{3/4} \sqrt{\nu_1}}{\sqrt{2 t}} \cos \left(\frac{5 \pi}{12} + \arg r_1(k_0) + \arg \Gamma(i \nu_1) + \sqrt{3} k_0^2 t - \nu_1 \ln \left(6 \sqrt{3} k_0^2 t\right) - \theta_1\right) + \mathcal{O}\left(t^{-3/4}+\frac1{tk_0}\right), \\
    q(x,t) &= -\frac{3^{1/4} \sqrt{\nu_1}}{\sqrt{2 t}} \sin \left(\frac{5 \pi}{12} + \arg r_1(k_0) + \arg \Gamma(i \nu_1) + \sqrt{3} k_0^2 t - \nu_1 \ln \left(6 \sqrt{3} k_0^2 t\right) - \theta_1\right) + \mathcal{O}\left(t^{-3/4}+\frac1{tk_0}\right).
\end{aligned}
$}
\end{equation}
\end{itemize}
All error estimates are uniform in their respective
regions.
\end{theorem}

\begin{proof}
The estimates in the Painlev\'e region are established
in Section~\ref{sed:P4}.
The first and second transition regions are analyzed
in Sections~\ref{Transition1} and \ref{sec:trans2},
respectively.
\end{proof}

    \begin{theorem}\label{mainthm-2}
        Let $u(x,t)$ be the solution of the good Boussinesq equation (\ref{good-boussinesq-1}) or (\ref{good-boussinesq-2}) with initial data $u(x,0)=u_0(x), w(x,0)=w_0(x)$ gotten by the Miura transformation (\ref{Miuratranf}) and the initial data $p_0(x), q_0(x)\in W^{2,1}_1(\mathbb R)$ and Assumption \ref{assumption} holds,
        then for $\frac{|x|}{\sqrt{t}}<c_1$ with $c_1>0$, as $t\to+\infty$, the long-time asymptotic behavior of the equation (\ref{good-boussinesq-1}) in the Painlev\'e region can be described by the Clarkson-McLeod solution $P_{\rm IV}(y)$ with $y= -\frac{\sqrt{3} x}{2\sqrt{t}}$ of the Painlev\'e IV equation (\ref{PIV}) with parameters  $\alpha = -\frac{1}{6}$ and $\beta = -\frac{2}{3}$ as follows:
        \begin{equation}\label{gbresult}
            u(x,t) \sim \small -\frac{\left(3 P'_{\rm IV}(y) + 2\right)^2 - 9 P_{\rm IV}^4(y) - 36 y P_{\rm IV}^3(y) - 20 y^2 P_{\rm IV}^2(y)}{32 t P_{\rm IV}^2(y)}.  
            \end{equation}
    \end{theorem}
    \begin{proof}
        The asymptotic formula in (\ref{gbresult}) follows directly from (\ref{mbresult}) and the Miura transformation (\ref{Miuratranf}).
    \end{proof}

\begin{figure}[h!]
    \centering
    \scalebox{0.9}{%
    \begin{minipage}{11cm}
        \centering
       \setlength{\unitlength}{1pt}%
\includegraphics[width=11cm]{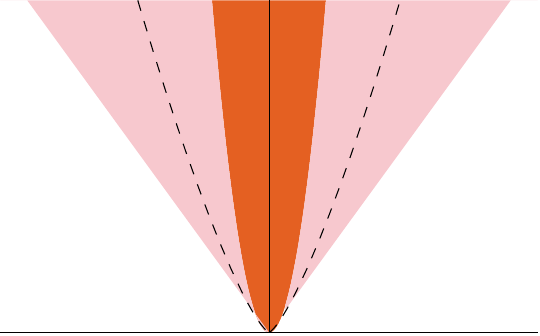}%
\put(-13.18,1.77){\vector(1,0){12}}%
\put(-156.97,182.34){\vector(0,1){12}}%
        \put(-175,160){\fontsize{10}{7}\selectfont \emph{Painlev\'e}}
        \put(-173,150){\fontsize{10}{7}\selectfont \emph{Region}}
        \put(-117,160){\fontsize{10}{7}\selectfont \emph{Transition}}
        \put(-85,150){\fontsize{8}{7}\selectfont \emph{Region $\mathrm{II}$}}
        \put(-127,150){\fontsize{8}{7}\selectfont \emph{Region $\mathrm{I}$}}
        \put(-242,160){\fontsize{10}{7}\selectfont \emph{Transition}}
        \put(-218,150){\fontsize{8}{7}\selectfont \emph{Region $\mathrm{I}$}}
        \put(-258,150){\fontsize{8}{7}\selectfont \emph{Region $\mathrm{II}$}}
        \put(-292,60){\fontsize{10}{7}\selectfont \emph{Dispersive Wave}}
        \put(-290,50){\fontsize{10}{7}\selectfont \emph{Region}}
        \put(-80,60){\fontsize{10}{7}\selectfont \emph{Dispersive Wave}}
        \put(-80,50){\fontsize{10}{7}\selectfont \emph{Region}}
        \put(-5,-10){\fontsize{10}{7}\selectfont $x$}
        \put(-159,-10){\fontsize{10}{7}\selectfont $0$}
        \put(-159,200){\fontsize{10}{7}\selectfont $t$}
        \put(-130,200){\fontsize{7}{7}\selectfont $c_1t^{1/2}$}
        \put(-200,200){\fontsize{7}{7}\selectfont $-c_1t^{1/2}$}
        \put(-90,200){\fontsize{7}{7}\selectfont $c_2t^{3/5}$}
        \put(-20,200){\fontsize{7}{7}\selectfont $c_3t$}
        \put(-310,200){\fontsize{7}{7}\selectfont $-c_3t$}
        \put(-245,200){\fontsize{7}{7}\selectfont $-c_2t^{3/5}$}
        \par
    \end{minipage}%
    }
    \caption{The asymptotic regions in the upper half
    $(x,t)$-plane for the modified Boussinesq equation
    \eqref{mbequation}.}
    \label{Category}
\end{figure}
    \subsection{Numerical results} 
    To verify the validity of the Theorem \ref{mainthm} and Theorem \ref{mainthm-2}, it is necessary to compare the leading-order terms in Painlev\'e asymptotics with the direct numerical simulations. Firstly,
    the numerical results are presented below by taking the initial data of the modified Boussinesq equation (\ref{mbequation}) of the form:  
\begin{equation}\label{initialmb}
\begin{cases}
\begin{aligned}
    & p(x,0) =p_0(x)= -\frac{1}{10} \exp\left(-\frac{x^2}{20}\right),\\
    & q(x,0)=q_0(x) = \frac{1}{10} \exp\left(-\frac{x^2}{20}\right).
\end{aligned}
\end{cases}
\end{equation}  
Under this initial condition, Figure~\ref{mBvsP4100} compares the asymptotic formulas based on the Painlevé $\rm IV$ equation in (\ref{mbresult})-(\ref{mbresult-3}) with the results of direct numerical simulations for \( t = 100 \) and \( t = 300 \). Specifically, the solid green and blue lines represent the direct numerical results for \( q(x,t) \) and \( p(x,t) \), respectively, while the dashed red lines correspond to the simulation results of the leading-order terms in (\ref{mbresult})-(\ref{mbresult-3}). As time increases, it is observed that the long-time asymptotic formula closely aligns with the results obtained by direct numerical simulations, which demonstrates the correctness of the theoretical
results in Theorem \ref{mainthm}.

\begin{figure}[h!]
    \centering
    \includegraphics[width=7.9cm]{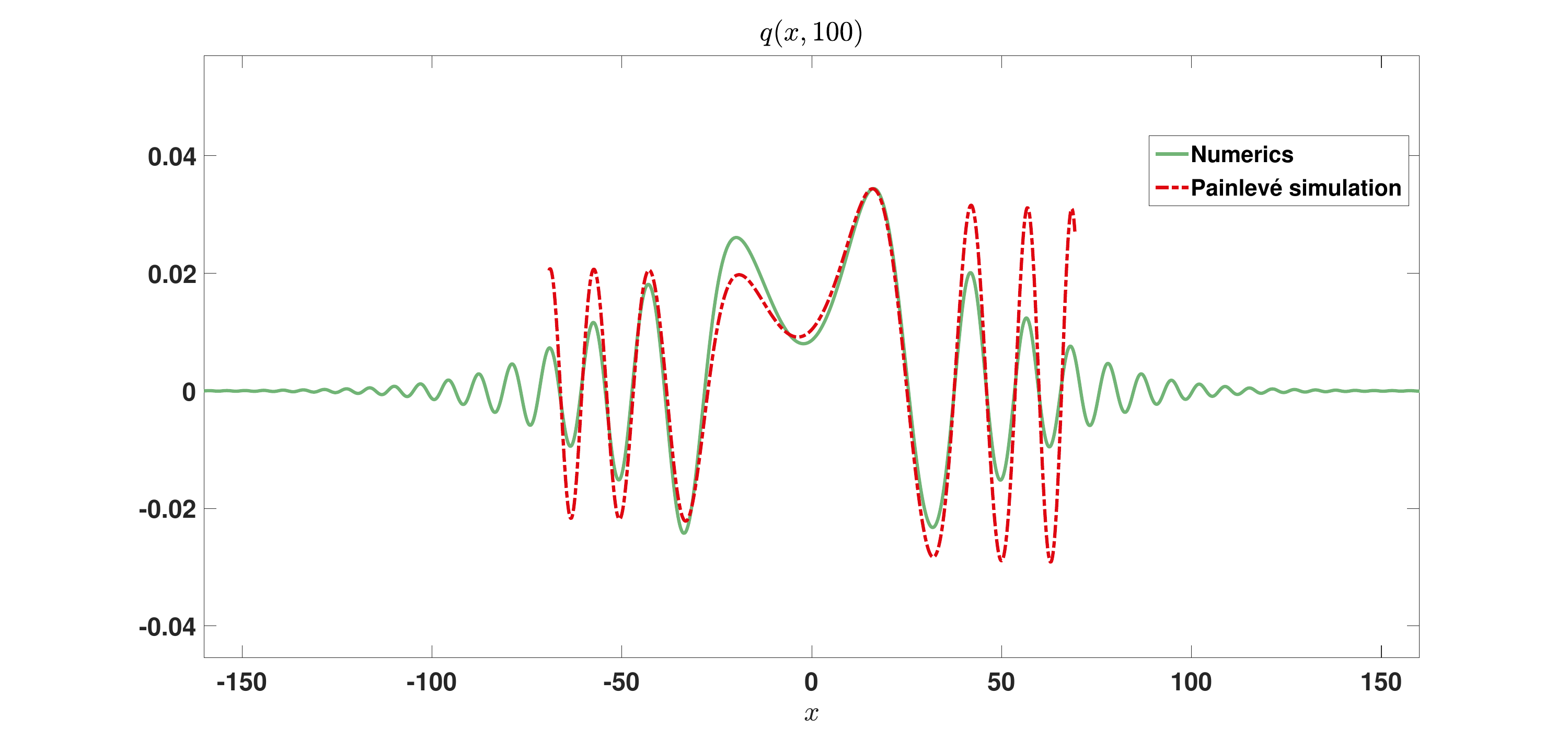}
    \includegraphics[width=7.9cm]{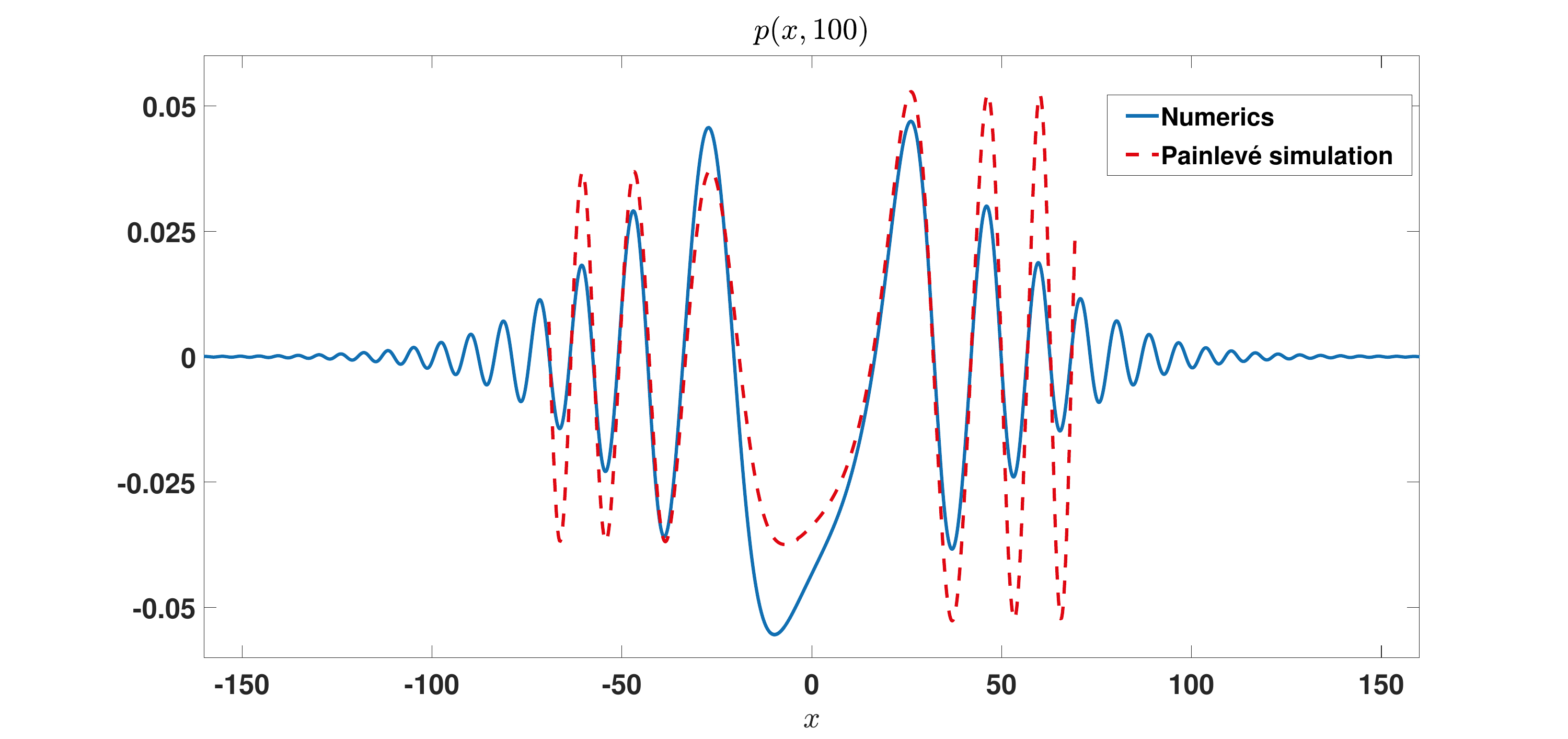}
    \includegraphics[width=7.9cm]{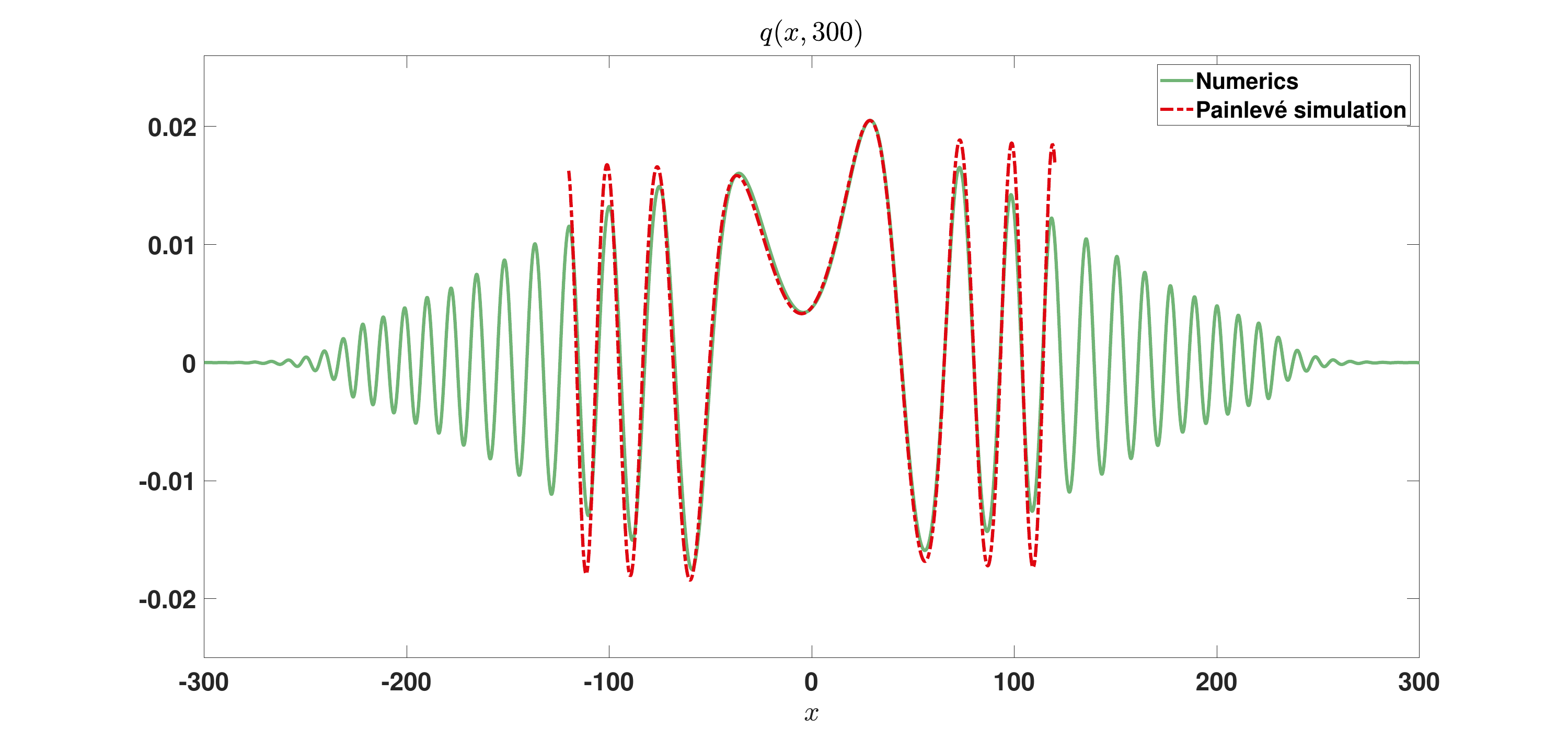}
    \includegraphics[width=7.9cm]{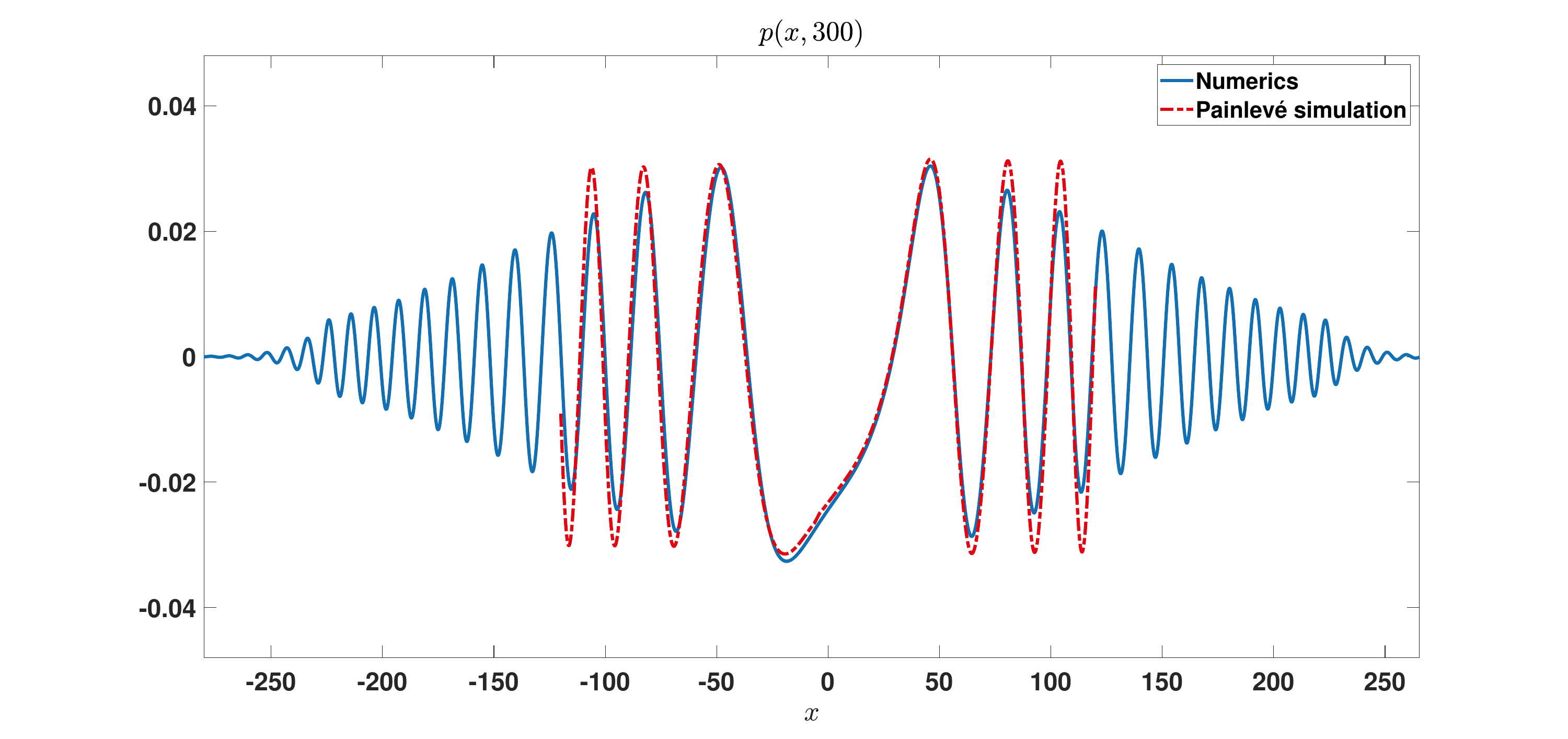}
    \caption{{\protect\small Comparisons of the asymptotic solutions based on the Painlevé $\rm IV$ equation in~(\ref{mbresult})-(\ref{mbresult-3}) with the full numerical simulations of the modified Boussinesq equation~(\ref{mbequation}) under the initial condition~(\ref{initialmb}) at \( t = 100 \) and \( t = 300 \). The solid green and blue lines represent the results of direct numerical simulations on \( q(x,t) \) and \( p(x,t) \), respectively, while the dashed red lines represent the asymptotic solutions.}}
    \label{mBvsP4100}
\end{figure}
Secondly, Figure~\ref{gBvsP4} illustrates the comparisons between the asymptotic formulas in (\ref{gbresult}) and the numerical simulations of the good Boussinesq equation~(\ref{good-boussinesq-1}) at \( t = 100 \) and \( t = 300 \), respectively, based on the following initial data:
\begin{equation}\label{initialgb}
  \begin{cases}
\begin{aligned}
    u(x,0) &= -\frac{1}{50} \exp\left(-\frac{x^2}{10}\right) - \frac{1}{100}x \exp\left(-\frac{x^2}{20}\right), \\[1ex]
    w(x,0) &= -\frac{1}{250}x \exp\left(-\frac{x^2}{10}\right) + \left(\frac{1}{100} - \frac{1}{1000}x^2\right) \exp\left(-\frac{x^2}{20}\right),
\end{aligned}
\end{cases}  
\end{equation}
which is obtained by combining the Miura transformation~(\ref{Miuratranf}) with the initial condition in~(\ref{initialmb}). In this figure, the solid yellow line represents the full numerical simulations of the solution \( u(x,t) \) of the good Boussinesq equation~(\ref{good-boussinesq-1}), while the dashed red lines represent the asymptotic formula in~(\ref{gbresult}). It is remarkably evident that the long-time asymptotic formulas match closely with the results from direct numerical simulations. This strong agreement demonstrates the accuracy of the asymptotic solution in (\ref{gbresult}) in Theorem \ref{mainthm-2}.

\begin{figure}[H]
    \centering
    \includegraphics[width=7.9cm]{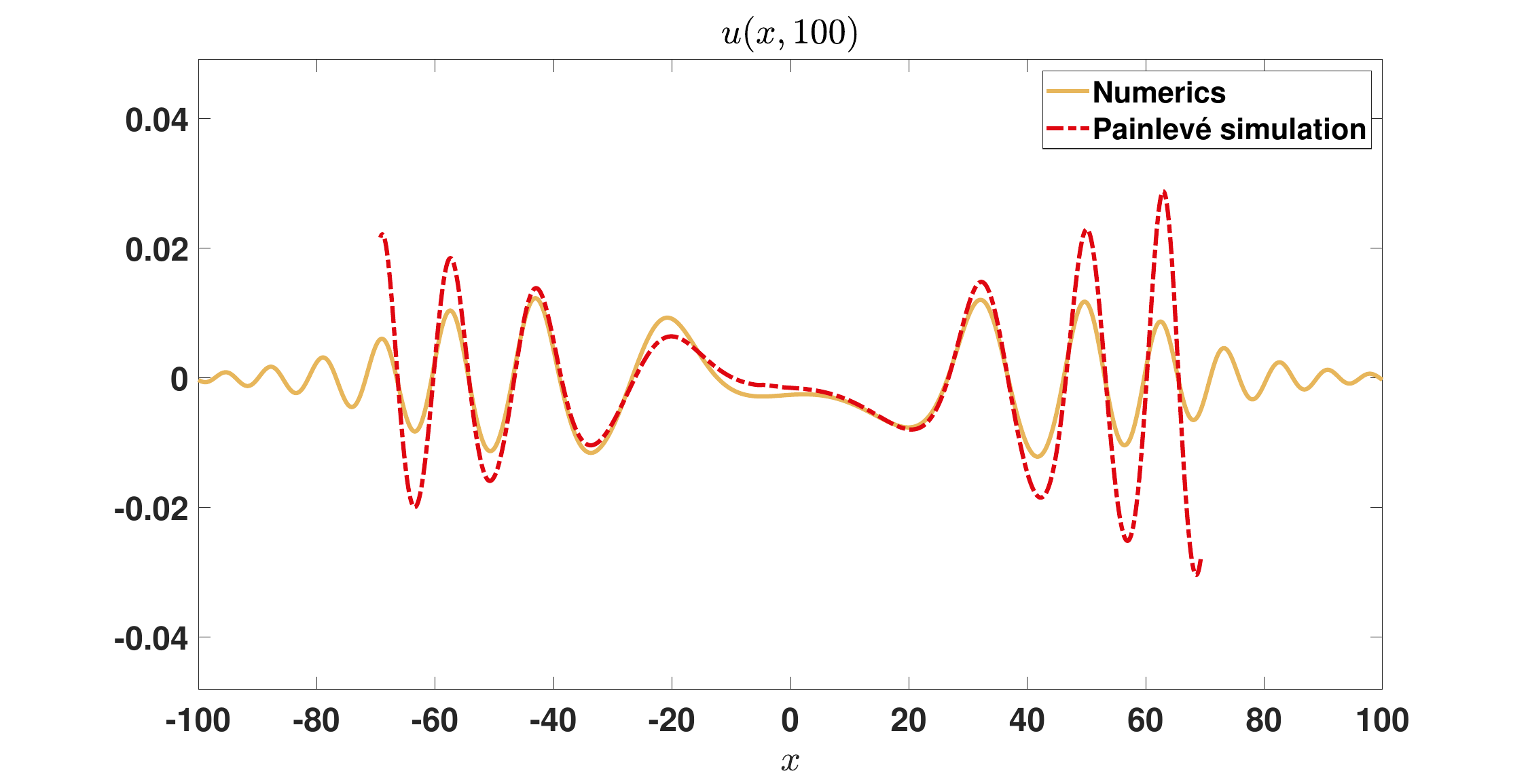}
    \includegraphics[width=7.9cm]{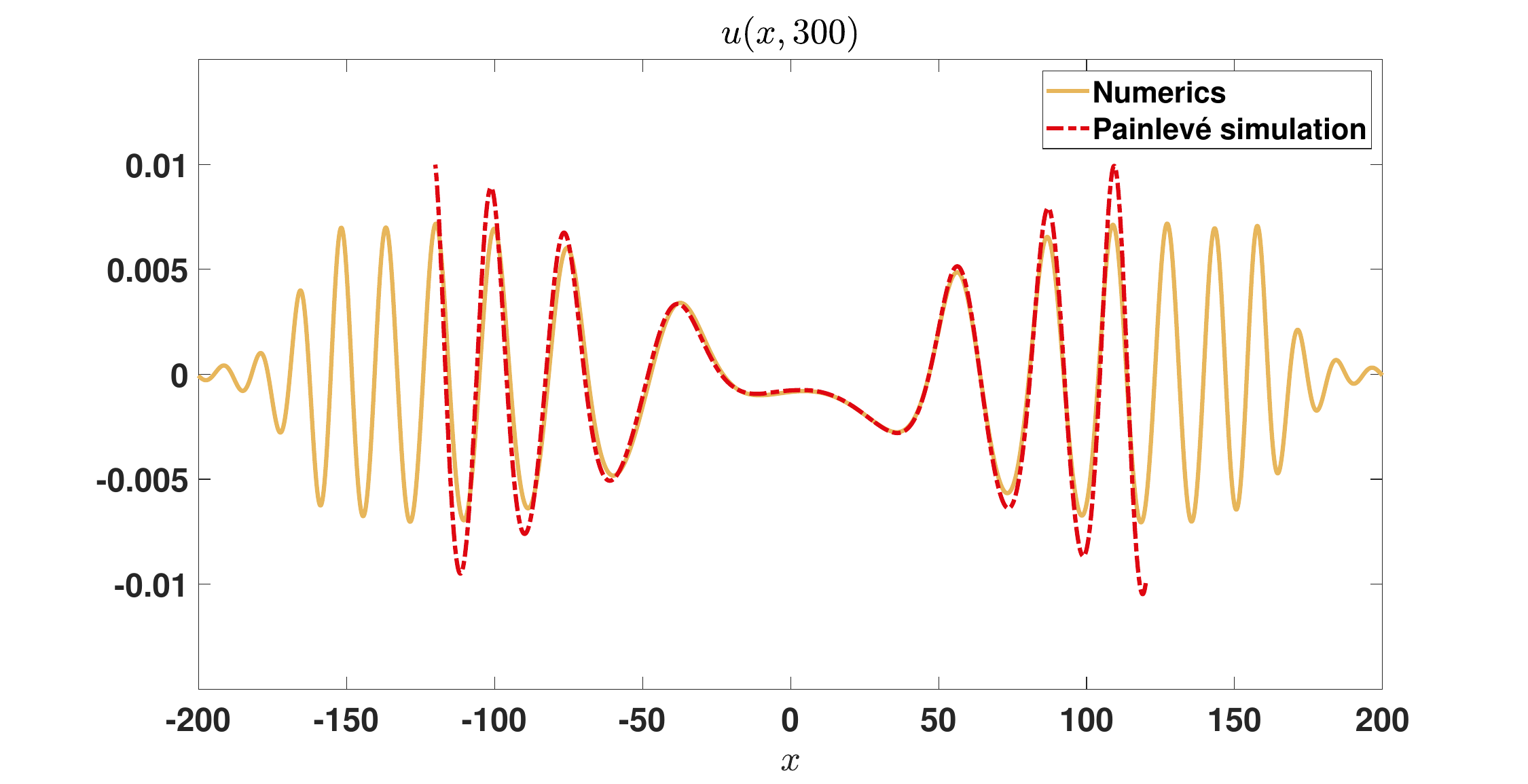}
    \caption{{\protect\small Comparisons of the asymptotic solutions based on the Painlevé~$\rm IV$ equation in (\ref{gbresult}) with the full numerical simulations of the good Boussinesq equation (\ref{good-boussinesq-1}) under the initial data \eqref{initialgb} at \( t = 100 \) and \( t = 300 \), respectively. The left and right panels correspond to two respective time points.}}
    \label{gBvsP4}
\end{figure}

       

\section{The Direct Scattering}\label{sec:direct}
In this section, we recall some results on the direct scattering problem for the modified Boussinesq equation; further details can be found in \cite{WangJMP}. Our aim is to establish that the reflection coefficients $r_1(k)$ and $r_2(k)$ belong to $H^1(\R_+)$ and $H^1(\R_-)$, respectively.

The modified Boussinesq equation \eqref{mbequation} admits the following Lax pair:
\begin{equation*}
    \begin{cases}
    \begin{aligned}
        \Psi_x(x,t,k)&=\widetilde{L}\Psi(x,t,k),\\
        \Psi_t(x,t,k)&=\widetilde{Z}\Psi(x,t,k),
    \end{aligned}
    \end{cases}
\end{equation*}
where 
\[
\widetilde{L}=
\begin{pmatrix}
p+q & k & 0\\
0 & -2q & k\\
k & 0 & q-p
\end{pmatrix},
\]
\[
\widetilde{Z}=
\begin{pmatrix}
\dfrac{p^2}{3}+q_x-\dfrac{p_x}{3}-q^2+2pq
& k(p-q)
& k^2\\[6pt]
k^2
& \dfrac{2p_x}{3}+2q^2-\dfrac{2p^2}{3}
& -k(p+q)\\[6pt]
2kq
& k^2
& \dfrac{p^2}{3}-q_x-\dfrac{p_x}{3}-q^2-2pq
\end{pmatrix}.
\]
Introduce the following gauge transformation

\[
\mu(x,t,k)=U^{-1}\Psi(x,t,k)e^{-\mathcal{L}(k)x-\mathcal{Z}(k)t},
\quad
U=
\begin{pmatrix}
\omega & \omega^2 & 1\\
\omega^2 & \omega & 1\\
1 & 1 & 1
\end{pmatrix},
\]
with 
\[
\mathcal{L}(k)=kJ=
\begin{pmatrix}
\omega k & 0 & 0\\
0 & \omega^2 k & 0\\
0 & 0 & k
\end{pmatrix},
\qquad
\mathcal{Z}(k)=k^2J^2=
\begin{pmatrix}
\omega^2 k^2 & 0 & 0\\
0 & \omega k^2 & 0\\
0 & 0 & k^2
\end{pmatrix},
\qquad \omega=e^{2\pi i/3}.
\]
Thus, $\mu(x,t,k)$ satisfies the following equations
\begin{equation}\label{eq:mu}
    \mu_x-[\mathcal{L},\mu]=V_1\mu,
\qquad
\mu_t-[\mathcal{Z},\mu]=V_2\mu,
\end{equation}
where
\[
V_1=
\frac{p}{3}
\begin{pmatrix}
0 & \omega-1 & \omega^2-1\\
\omega^2-1 & 0 & \omega-1\\
\omega-1 & \omega^2-1 & 0
\end{pmatrix}
-q
\begin{pmatrix}
0 & \omega^2 & \omega\\
\omega & 0 & \omega^2\\
\omega^2 & \omega & 0
\end{pmatrix}.
\]
In particular, at $t=0$, we denote the two Jost solutions of \eqref{eq:mu} by $X(x,k)$ and $Y(x,k)$, respectively. They satisfy the following Volterra integral equations:
\[
\begin{aligned}
X(x,k)
&=I-\int_x^\infty
e^{(x-s)\widehat{\mathcal{L}(k)}}(V_1X)(s,k)\,ds,\\
Y(x,k)
&=I+\int_{-\infty}^x
e^{(x-s)\widehat{\mathcal{L}(k)}}(V_1Y)(s,k)\,ds.
\end{aligned}
\]

For each fixed $x\in\mathbb{R}$, the columns
$[X(x,k)]_1$, $[X(x,k)]_2$, and $[X(x,k)]_3$
are well defined and continuous on
\(
\overline{D_1}\cup\overline{D_2},
\
\overline{D_5}\cup\overline{D_6},
\
\overline{D_3}\cup\overline{D_4},
\)
respectively, and analytic in the interiors of these sets,
where the sectors $D_j (j=1,2\cdots,6)$ are shown in Figure~\ref{Sigma}.

\begin{lemma}\label{lem:Jost-large-k}
Assume that $p_0(x),q_0(x)\in W^{2,1}_1(\mathbb R)$ and Assumption \ref{assumption} holds.
Then the Jost solutions admit the expansions
\[
X(x,k)=I+\frac{X_1(x)}{k}+\mathcal{O}(k^{-2}),
\qquad
Y(x,k)=I+\frac{Y_1(x)}{k}+\mathcal{O}(k^{-2}),
\qquad k\to\infty,
\]
uniformly for $x\in\mathbb{R}$, columnwise within their
respective regions of boundedness.
\end{lemma}

\begin{proof}
We prove the expansion for $X(x,k)$, and the proof for $Y(x,k)$ is analogous.
The argument follows the approach in \cite{HuangLenells2018}.
Substituting a formal expansion into
$X_x-k[J,X]=V_1X$ gives
\[
[J,X_{j+1}]=X_j'-V_1X_j,
\qquad X_0=I,
\]
with
\[
(X_j^{(d)})'=(V_1X_j)^{(d)},
\qquad
X_j(x)\to0 \quad \text{as }x\to+\infty,
\qquad j\ge1.
\]
In particular,
\[
\begin{aligned}
X_1(x)={}&
\begin{pmatrix}
\omega^2 & 0 & 0\\
0 & \omega & 0\\
0 & 0 & 1
\end{pmatrix}
\int_{+\infty}^{x}
\left(\frac{p_0(s)^2}{3}+q_0(s)^2\right)\,ds\\
&+\frac{p_0(x)}{3}
\begin{pmatrix}
0 & \omega^2 & \omega^2\\
\omega & 0 & \omega\\
1 & 1 & 0
\end{pmatrix}
+\frac{q_0(x)}{3}
\begin{pmatrix}
0 & \omega-1 & 1-\omega\\
\omega^2-1 & 0 & 1-\omega^2\\
\omega^2-\omega & \omega-\omega^2 & 0
\end{pmatrix}.
\end{aligned}
\]
Define
\[
X_{[2]}(x,k):=I+\frac{X_1(x)}{k}+\frac{X_2(x)}{k^2}.
\]
The recursion relations imply
\[
\partial_xX_{[2]}-k[J,X_{[2]}]-V_1X_{[2]}
=\frac{X_2'-V_1X_2}{k^2}.
\]
Since $p_0(x), q_0(x)\in W^{2,1}(\mathbb{R})$, we have
\[
X_j\in L^\infty(\mathbb{R}),
\qquad X_j'\in L^1(\mathbb{R}),
\qquad j=1,2.
\]
In particular, we have
\[
g(x):=|X_2'(x)|+|V_1(x)|\,|X_2(x)|
\in L^1(\mathbb{R}).
\]

For sufficiently large $|k|$, $X_{[2]}$ is uniformly invertible.
Set
\[
\Delta:=
-\left(
\partial_xX_{[2]}-k[J,X_{[2]}]-V_1X_{[2]}
\right)X_{[2]}^{-1}.
\]
Then $|\Delta(x,k)|\le C|k|^{-2}g(x)$, and
\[
\partial_x(X_{[2]}^{-1}X)
-k[J,X_{[2]}^{-1}X]
=X_{[2]}^{-1}\Delta X.
\]
Consequently, it follows that
\[
X(x,k)=X_{[2]}(x,k)
-\int_x^\infty
X_{[2]}(x,k)e^{k(x-y)\widehat J}
\left(
X_{[2]}^{-1}(y,k)\Delta(y,k)X(y,k)
\right)\,dy.
\]
All estimates below are understood columnwise in the
corresponding regions of boundedness, where the exponential
factors in the Volterra kernel have modulus at most one.
Gronwall's inequality gives
\[
|X(x,k)|
\le C\exp\left(
\frac{C}{|k|^2}\int_x^\infty g(y)\,dy
\right).
\]
Substituting this bound into the integral equation yields
\[
|X(x,k)-X_{[2]}(x,k)|
\le
\frac{C}{|k|^2}
\left(\int_x^\infty g(y)\,dy\right)
\exp\left(
\frac{C}{|k|^2}\int_x^\infty g(y)\,dy
\right).
\]
Since $g\in L^1(\mathbb{R})$ and $X_2\in L^\infty(\mathbb{R})$,
we conclude that
\[
\left|X(x,k)-I-\frac{X_1(x)}{k}\right|
\le
\frac{|X_2(x)|}{|k|^2}
+|X(x,k)-X_{[2]}(x,k)|
\le \frac{C}{|k|^2}.
\]
\end{proof}

Define the scattering matrix $s(k)$ by
\begin{equation*}
    s(k)=I-\int_{-\infty}^{\infty}
e^{-s\widehat{\mathcal{L}(k)}}(V_1X)(s,k)\,ds.
\end{equation*}
The entries of the scattering matrix $s(k)$ are well defined
and continuous in the following respective domains:
\[
k\in
\begin{pmatrix}
\overline{D_1}\cup\overline{D_2}
& \mathbb{R}_+ & \omega\mathbb{R}_+\\
\mathbb{R}_+
& \overline{D_5}\cup\overline{D_6}
& \omega^2\mathbb{R}_+\\
\omega\mathbb{R}_+ & \omega^2\mathbb{R}_+
& \overline{D_3}\cup\overline{D_4}
\end{pmatrix},
\qquad \mathbb{R}_+=(0,\infty).
\]
Here, the matrix of sets specifies the domain of each
corresponding entry of $s(k)$.
Moreover, the diagonal entries are analytic in the interiors
of their respective domains.

\begin{lemma}\label{lem:s-large-k}
Assume that $p_0(x), q_0(x)\in W^{2,1}_1(\mathbb R)$ and Assumption \ref{assumption} holds.
Then the scattering matrix $s(k)$ admits the expansion
\[
s(k)=I+\frac{s_1}{k}+\mathcal{O}(k^{-2}),
\qquad k\to\infty,
\]
entrywise in its respective domain of definition.
In particular, $s_1$ is diagonal, and
$s_{ij}(k)=\mathcal{O}(k^{-2})$ for $i\ne j$
along the corresponding rays of definition.
\end{lemma}

\begin{proof}
The proof follows the argument in \cite{Charlier-Lenells-2021}.
We use the integral representation
\[
s(k)=I-\int_{\mathbb{R}}
e^{-x\widehat{\mathcal{L}(k)}}(V_1X)(x,k)\,dx,
\qquad \mathcal{L}(k)=kJ.
\]
By the preceding estimates in Lemma \ref{lem:Jost-large-k}, it follows that
\[
X(x,k)=I+\frac{X_1(x)}{k}+R_1(x,k),
\qquad
\sup_{x\in\mathbb{R}}|R_1(x,k)|\le C|k|^{-2},
\]
columnwise in the corresponding regions of boundedness.
For diagonal entries, the exponential factor equals one;
for off-diagonal entries, it has modulus one on the corresponding
rays. Since $V_1\in L^1(\mathbb{R})$, it follows that
\begin{equation}\label{eq:s-truncated}
s(k)=
I-\sum_{m=0}^{1}\frac{1}{k^m}
\int_{\mathbb{R}}
e^{-x\widehat{\mathcal{L}(k)}}(V_1X_m)(x)\,dx
+\mathcal{O}(k^{-2}),
\qquad X_0=I.
\end{equation}

For the off-diagonal entries, set
\[
s_{m,ij}(x):=(V_1X_m)_{ij}(x),
\qquad i\ne j,\quad m=0,1.
\]
The assumptions and the recursion relations imply
$s_{m,ij}\in W^{2-m,1}(\mathbb{R})$.
Integrating by parts $2-m$ times, with vanishing boundary terms,
gives
\[
\int_{\mathbb{R}}
e^{-kx(\omega^i-\omega^j)}s_{m,ij}(x)\,dx
=
\frac{1}{[k(\omega^i-\omega^j)]^{2-m}}
\int_{\mathbb{R}}
e^{-kx(\omega^i-\omega^j)}
\partial_x^{2-m}s_{m,ij}(x)\,dx.
\]
Since $|e^{-kx(\omega^i-\omega^j)}|=1$ on the corresponding ray, we have
\[
\left|
\frac{1}{k^m}\int_{\mathbb{R}}
e^{-kx(\omega^i-\omega^j)}s_{m,ij}(x)\,dx
\right|
\le
\frac{\|\partial_x^{2-m}s_{m,ij}\|_{L^1(\mathbb{R})}}
{|k|^2|\omega^i-\omega^j|^{2-m}}.
\]
Thus, the equation \eqref{eq:s-truncated} yields
$s_{ij}(k)=\mathcal{O}(k^{-2})$ for $i\ne j$.

For the diagonal entries, the exponential factor is absent.
Since $V_1$ has zero diagonal, we obtain
\[
s(k)=I+\frac{s_1}{k}+\mathcal{O}(k^{-2}),
\qquad
s_1=-\int_{\mathbb{R}}(V_1X_1)^{(d)}(x)\,dx.
\]
Substituting the expression for $X_1$ gives
\[
s_1=
-J^2\int_{\mathbb{R}}
\left(\frac{p_0(x)^2}{3}+q_0(x)^2\right)\,dx.
\]
\end{proof}

To prove that $r_1(k)$ belongs to $H^1(\R_+)$, we also need the following two lemmas concerning the asymptotic expansions of $\partial_kX(x,k)$ and $\partial_ks(k)$.

\begin{lemma}
Suppose that $p_0(x), q_0(x)\in W^{2,1}_1(\mathbb R)$ and that
Assumption~\ref{assumption} holds.
Then, columnwise in the corresponding regions of boundedness,
$\partial_kX(x,k)$ exists for sufficiently large $|k|$ and satisfies
\[
\left|
\partial_k\left(
X(x,k)-I-\frac{X_1(x)}{k}
\right)
\right|
\le \frac{C(1+|x|)}{|k|^2}.
\]
In particular, for each fixed $x\in\mathbb{R}$, it is seen that
\[
\partial_kX(x,k)
=-\frac{X_1(x)}{k^2}+\mathcal{O}(k^{-2}),
\qquad k\to\infty.
\]
The analogous statements hold for $Y(x,k)$.
\end{lemma}

\begin{proof}
We use the notation of the preceding proof in Lemma \ref{lem:Jost-large-k} and set
\[
P(x,k):=X_{[2]}(x,k)
=I+\frac{X_1(x)}{k}+\frac{X_2(x)}{k^2}.
\]
The weighted assumption and the recursion relations imply
\[
\int_{\mathbb{R}}(1+|x|)g(x)\,dx<\infty,
\qquad
g(x):=|X_2'(x)|+|V_1(x)|\,|X_2(x)|.
\]
Recall that
\[
\Delta(x,k)
=-\frac{X_2'(x)-V_1(x)X_2(x)}{k^2}P(x,k)^{-1}.
\]
Since $X_1,X_2$ are bounded, for sufficiently large $|k|$,
\[
|P|+|P^{-1}|\le C,
\qquad
|\partial_kP|+|\partial_kP^{-1}|
\le C|k|^{-2}.
\]
Consequently, it follows that
\[
|\Delta(x,k)|\le \frac{Cg(x)}{|k|^2},
\qquad
|\partial_k\Delta(x,k)|
\le \frac{Cg(x)}{|k|^3}.
\]

Fix a column index $j\in\{1,2,3\}$ and write
\[
\Psi(x,k):=X(x,k)e_j,
\qquad
\Psi_0(x,k):=P(x,k)e_j.
\]
The integral equation in the preceding proof becomes
\[
\Psi(x,k)
=\Psi_0(x,k)
-\int_x^\infty E_j(x,y,k)\Delta(y,k)\Psi(y,k)\,dy,
\]
where
\[
E_j(x,y,k)
=P(x,k)e^{k(x-y)(J-\omega^j I)}P(y,k)^{-1}.
\]
In the region of boundedness of the $j$th column, we have
\[
|E_j(x,y,k)|\le C,
\qquad
|\partial_kE_j(x,y,k)|
\le C(1+y-x),
\qquad y\ge x.
\]
The additional factor $y-x$ comes from differentiating
the exponential.

For each fixed $x$, these estimates and the weighted
integrability of $g$ justify differentiating the Volterra
Neumann series on $[x,\infty)$. Thus
$\Lambda:=\partial_k\Psi$ satisfies
\[
\Lambda(x,k)
=\Lambda_0(x,k)
-\int_x^\infty
E_j(x,y,k)\Delta(y,k)\Lambda(y,k)\,dy,
\]
where
\[
\Lambda_0(x,k)
=\partial_k\Psi_0(x,k)
-\int_x^\infty
\partial_k\bigl(E_j(x,y,k)\Delta(y,k)\bigr)
\Psi(y,k)\,dy.
\]
Using the previously established bound $|\Psi(y,k)|\le C$,
we obtain
\[
|\Lambda_0(x,k)-\partial_k\Psi_0(x,k)|
\le
\frac{C}{|k|^2}
\int_x^\infty(1+y-x)g(y)\,dy.
\]
In particular, since $|\partial_k\Psi_0|\le C|k|^{-2}$,
\[
\sup_{z\ge x}|\Lambda_0(z,k)|
\le
\frac{C}{|k|^2}
\left[
1+\int_x^\infty(1+y-x)g(y)\,dy
\right].
\]
Applying the Volterra estimates to the equation for $\Lambda$
therefore yields
\[
\sup_{z\ge x}|\Lambda(z,k)|
\le
\frac{C}{|k|^2}
\left[
1+\int_x^\infty(1+y-x)g(y)\,dy
\right].
\]
Substituting this bound back into the differentiated
integral equation gives
\[
|\Lambda(x,k)-\partial_k\Psi_0(x,k)|
\le
\frac{C}{|k|^2}
\left[
1+\int_x^\infty(1+y-x)g(y)\,dy
\right].
\]
Moreover, we have
\[
1+\int_x^\infty(1+y-x)g(y)\,dy
\le C(1+|x|),
\]
by the weighted integrability of $g$. Since
\[
\partial_kP(x,k)
=-\frac{X_1(x)}{k^2}-\frac{2X_2(x)}{k^3},
\]
and $X_2$ is bounded, we conclude that
\[
\left|
\partial_k\left(
X(x,k)-I-\frac{X_1(x)}{k}
\right)e_j
\right|
\le \frac{C(1+|x|)}{|k|^2}.
\]
The other columns are treated in the same way.
\end{proof}

\begin{lemma}\label{lem:s-k-derivative}
Suppose that $p_0(x), q_0(x)\in W^{2,1}_1(\mathbb R)$ and that
Assumption~\ref{assumption} holds.
Then, entrywise in the corresponding domains of definition,
\[
\partial_k\left(s(k)-I-\frac{s_1}{k}\right)
=\mathcal{O}(k^{-2}),
\qquad k\to\infty.
\]
In particular, $\partial_ks(k)=\mathcal{O}(k^{-2})$.
For off-diagonal entries, derivatives are understood along
their respective rays of definition.
\end{lemma}

\begin{proof}
Write
\[
X(x,k)=I+\frac{X_1(x)}{k}+R_1(x,k).
\]
The preceding lemmas give
\[
|R_1(x,k)|\le C|k|^{-2},
\qquad
|\partial_kR_1(x,k)|
\le C(1+|x|)|k|^{-2},
\]
columnwise in the corresponding regions of boundedness.
Substituting this decomposition into the scattering integral,
we obtain
\[
s(k)=I-\sum_{m=0}^{1}\frac{1}{k^m}
\int_{\mathbb{R}}
e^{-x\widehat{\mathcal{L}(k)}}(V_1X_m)(x)\,dx
+\mathcal{E}(k),
\qquad X_0=I,
\]
where
\[
\mathcal{E}(k)
=-\int_{\mathbb{R}}
e^{-x\widehat{\mathcal{L}(k)}}(V_1R_1)(x,k)\,dx.
\]
Since the exponential factors have modulus one for admissible
off-diagonal entries and equal one for diagonal entries,
differentiation under the integral sign gives
\[
|\partial_k\mathcal{E}_{ij}(k)|
\le
C\int_{\mathbb{R}}|V_1(x)|
\left(
|x|\,|R_1(x,k)|+|\partial_kR_1(x,k)|
\right)\,dx
\le C|k|^{-2}.
\]

For $i\ne j$, set
\[
s_{m,ij}(x):=(V_1X_m)_{ij}(x),
\qquad m=0,1.
\]
The weighted assumptions and the recursion relations imply
\[
s_{m,ij},\;x s_{m,ij}
\in W^{2-m,1}(\mathbb{R}),
\qquad m=0,1.
\]
Thus, integrating by parts 
$2-m$ times yields
\[
\int_{\mathbb{R}}
e^{-kx(\omega^i-\omega^j)}s_{m,ij}(x)\,dx
=\mathcal{O}(|k|^{-(2-m)})
\]
and
\[
\int_{\mathbb{R}}
e^{-kx(\omega^i-\omega^j)}x s_{m,ij}(x)\,dx
=\mathcal{O}(|k|^{-(2-m)}).
\]
All boundary terms vanish, and the estimates hold along
the corresponding ray.

Differentiating each truncated term gives
\[
\begin{aligned}
\partial_k\left[
\frac{1}{k^m}\int_{\mathbb{R}}
e^{-kx(\omega^i-\omega^j)}s_{m,ij}(x)\,dx
\right]
={}&-\frac{m}{k^{m+1}}\int_{\mathbb{R}}
e^{-kx(\omega^i-\omega^j)}s_{m,ij}(x)\,dx\\
&-\frac{\omega^i-\omega^j}{k^m}
\int_{\mathbb{R}}
e^{-kx(\omega^i-\omega^j)}
x s_{m,ij}(x)\,dx.
\end{aligned}
\]
The first term is $\mathcal{O}(k^{-3})$ when $m=1$
and vanishes when $m=0$, while the second term is
$\mathcal{O}(k^{-2})$ for both values of $m$.
Together with the estimate for $\partial_k\mathcal{E}$,
this proves
\(
\partial_ks_{ij}(k)=\mathcal{O}(k^{-2}),
\ i\ne j.
\)

For the diagonal entries, the exponential factor is absent.
Since $V_1$ has zero diagonal, we have
\[
s_{ii}(k)
=1-\frac{1}{k}\int_{\mathbb{R}}(V_1X_1)_{ii}(x)\,dx
+\mathcal{E}_{ii}(k).
\]
Recalling that
\[
s_1=-\int_{\mathbb{R}}(V_1X_1)^{(d)}(x)\,dx,
\]
we obtain
\[
\partial_ks_{ii}(k)
=-\frac{(s_1)_{ii}}{k^2}+\mathcal{O}(k^{-2}).
\]
Since $s_1$ is diagonal, the diagonal and off-diagonal
estimates together give the claimed result.
\end{proof}

In addition to the Jost solutions $X(x,k)$ and $Y(x,k)$,
we introduce the cofactor Jost solutions $X^A(x,k)$ and
$Y^A(x,k)$, which satisfy the Volterra integral equations
\[
\begin{aligned}
X^A(x,k)
&=I+\int_x^\infty
e^{-(x-s)\widehat{\mathcal{L}(k)}}
(V_1^T X^A)(s,k)\,ds,\\
Y^A(x,k)
&=I-\int_{-\infty}^x
e^{-(x-s)\widehat{\mathcal{L}(k)}}
(V_1^T Y^A)(s,k)\,ds,
\end{aligned}
\]
where $V_1^T$ denotes the transpose of $V_1$.
The corresponding cofactor scattering matrix $s^A(k)$
is defined by
\[
s^A(k)=I+\int_{-\infty}^{\infty}
e^{s\widehat{\mathcal{L}(k)}}
(V_1^T X^A)(s,k)\,ds.
\]
The cofactor scattering matrix $s^A(k)$ admits analogous asymptotic expansions to those in Lemmas~\ref{lem:s-large-k} and~\ref{lem:s-k-derivative}. We omit the details for brevity.

We now define the reflection coefficients by
\begin{equation}\label{def:r1r2}
\begin{aligned}
        r_1(k)&:=\frac{s_{12}(k)}{s_{11}(k)},
&&k\in(0,\infty),\\
r_2(k)&:=\frac{s_{12}^A(k)}{s_{11}^A(k)},&&k\in(-\infty,0).
\end{aligned}
\end{equation}

\subsection*{The Proof of Proposition \ref{prop:r1r2-H1}}
\begin{proof}
    By the local regularity of the scattering coefficients
and the nonvanishing of the denominators in
\eqref{def:r1r2}, the functions $r_1(k)$ and $r_1'(k)$
are bounded on every finite interval.
Moreover, \eqref{def:r1r2} and
Lemmas~\ref{lem:s-large-k} and
\ref{lem:s-k-derivative} imply that
\[
r_1(k)=\mathcal O(k^{-1}),
\qquad
r_1'(k)=\mathcal O(k^{-1}),
\qquad k\to+\infty.
\]
Combining the local bounds with these estimates yields
\[
\int_0^\infty
\left(|r_1(k)|^2+|r_1'(k)|^2\right)\,\mathrm dk
<\infty.
\]
Hence $r_1(k)\in H^1(\mathbb R_+)$.
The proof of the property for $r_2(k)$ on $\mathbb R_-$ is analogous.
\end{proof}

\section{The Painlevé region}\label{sed:P4}
In this section, we employ the $\dbar$-method to analyze the long-time
asymptotic behavior in the Painlev\'e region, where
\(
|x|<c_1 t^{1/2}.
\)
The $\dbar$-method has been widely applied to the study of various
integrable systems, particularly in the analysis of Painlev\'e asymptotics
for the defocusing NLS equation \cite{Wang-Fan-2023} and the long-time asymptotic
analysis of the mKdV equation \cite{Lenellsmkdv,Chen-Liu-2022}.
\subsection{The first Deformation}
Define the matrix-valued function $m^{(1)}(x,t,k)$ by
\begin{equation*}
    m^{(1)}(x,t,k)
    =
    m(x,t,k)\mathcal{G}(x,t,k),
\end{equation*}
where the matrix-valued function $\mathcal{G}(x,t,k)$ is chosen so that this transformation eliminates the jumps on $\Sigma_2$, $\Sigma_4$, and $\Sigma_6$ shown in Figure~\ref{Sigma}. To construct $\mathcal{G}$, for each $j=1,\ldots,6$, we introduce a continuous function $G_j(k)$ in the domain $D_j$ shown in Figure~\ref{Sigma}, which admits a continuous extension to $\overline{D_j}$.
The boundary values of the functions $G_j$, $j=1,\ldots,6$, are
prescribed as follows:
\begin{equation}\label{def:G-boundary}
\begin{aligned}
G_1(k)&=
\begin{cases}
r_2^{*}(0),       & k\in\Sigma_1,\\
r_2^{*}(\omega k),& k\in\Sigma_2,
\end{cases}
&
G_2(k)&=
\begin{cases}
r_2(\omega k),& k\in\Sigma_2,\\
r_2(0),       & k\in\Sigma_3,
\end{cases}
\\[2mm]
G_3(k)&=
\begin{cases}
r_2^{*}(0),& k\in\Sigma_3,\\
r_2^{*}(k),& k\in\Sigma_4,
\end{cases}
&
G_4(k)&=
\begin{cases}
r_2(k),& k\in\Sigma_4,\\
r_2(0),& k\in\Sigma_5,
\end{cases}
\\[2mm]
G_5(k)&=
\begin{cases}
r_2^{*}(0),          & k\in\Sigma_5,\\
r_2^{*}(\omega^2 k), & k\in\Sigma_6,
\end{cases}
&
G_6(k)&=
\begin{cases}
r_2(\omega^2 k),& k\in\Sigma_6,\\
r_2(0),         & k\in\Sigma_1.
\end{cases}
\end{aligned}
\end{equation}
\begin{lemma}\label{lem:dbar-extension-r2}
Suppose that
\(
r_2(k)\in H^1(\mathbb{R}_{-}).
\)
Then, for each $j=1,\ldots,6$, there exists a function
\(
G_j\in C(\overline{D_j})
\)
satisfying the boundary conditions in \eqref{def:G-boundary}.
Moreover, for almost every $k\in D_j$,
\begin{equation*}
    \left|\bar\partial G_j(k)\right|
    \leq
    C\left(
       |k|^{-1/2}\|r_2'\|_{L^2(\mathbb{R}_{-})}
       +
       \left|r_2'\bigl(\Re(\alpha_j k)\bigr)\right|
    \right),
\end{equation*}
where $\alpha_1=\alpha_2=\omega,\alpha_3=\alpha_4=1$ and $\alpha_5=\alpha_6=\omega^2$.
The constant $C>0$ is independent of $k$.
\end{lemma}
\begin{proof}
We only prove the result for $G_4$, since the other cases follow
similarly by rotation and complex conjugation. Write
\(
k=\rho e^{i\phi},
\)
with $\rho\geq0,
\
\pi\leq\phi\leq\frac{4\pi}{3}$.

Let $\chi\in C^\infty([\pi,4\pi/3])$ be a cutoff function satisfying
\(
0\leq\chi\leq1,
\) and $
\chi(\pi)=1,
\chi\left(\frac{4\pi}{3}\right)=0$. In particular, $|\chi'(\phi)|\leq C$.
Define
\begin{equation*}
G_4(k)
=
\chi(\phi)r_2(\Re k)
+
\bigl(1-\chi(\phi)\bigr)r_2(0).
\end{equation*}

Using
\(
\bar\partial
=
\frac12e^{i\phi}
\left(\partial_\rho+i\rho^{-1}\partial_\phi\right),
\)
we obtain
\[
\bar\partial G_4(k)
=
\frac{i e^{i\phi}}{2\rho}\chi'(\phi)
\bigl(r_2(\Re k)-r_2(0)\bigr)
+
\frac12\chi(\phi)r_2'(\Re k).
\]
Therefore, it follows that
\[
|\bar\partial G_4(k)|
\leq
\frac{C}{|k|}
|r_2(\Re k)-r_2(0)|
+
C|r_2'(\Re k)|.
\]
By the Cauchy--Schwarz inequality, we have
\[
\begin{aligned}
|r_2(\Re k)-r_2(0)|
&\leq
|k|^{1/2}
\|r_2'\|_{L^2(\mathbb{R}_{-})}.
\end{aligned}
\]
Consequently,
\(
|\bar\partial G_4(k)|
\leq
C\left(
|k|^{-1/2}
+
|r_2'(\Re k)|
\right).
\)
\end{proof}

We now define the matrix-valued function $\mathcal{G}(x,t,k)$
piecewise by setting
\(
\mathcal{G}(x,t,k)
=
\mathcal{G}_j(x,t,k),
\ k\in D_j,\ j=1,\ldots,6,
\)
where the matrices $\mathcal{G}_j(x,t,k)$ are given by
\begin{equation}\label{def:mathcal-Gj}
\begin{aligned}
\mathcal{G}_1(x,t,k)
&:=
\begin{pmatrix}
1 & 0 & 0\\
0 & 1 & -G_1(k)e^{-t\Phi_{32}}\\
0 & 0 & 1
\end{pmatrix},
&
\mathcal{G}_2(x,t,k)
&:=
\begin{pmatrix}
1 & 0 & 0\\
0 & 1 & 0\\
0 & -G_2(k)e^{t\Phi_{32}} & 1
\end{pmatrix},
\\[4mm]
\mathcal{G}_3(x,t,k)
&:=
\begin{pmatrix}
1 & -G_3(k)e^{-t\Phi_{21}} & 0\\
0 & 1 & 0\\
0 & 0 & 1
\end{pmatrix},
&
\mathcal{G}_4(x,t,k)
&:=
\begin{pmatrix}
1 & 0 & 0\\
-G_4(k)e^{t\Phi_{21}} & 1 & 0\\
0 & 0 & 1
\end{pmatrix},
\\[4mm]
\mathcal{G}_5(x,t,k)
&:=
\begin{pmatrix}
1 & 0 & 0\\
0 & 1 & 0\\
-G_5(k)e^{t\Phi_{31}} & 0 & 1
\end{pmatrix},
&
\mathcal{G}_6(x,t,k)
&:=
\begin{pmatrix}
1 & 0 & -G_6(k)e^{-t\Phi_{31}}\\
0 & 1 & 0\\
0 & 0 & 1
\end{pmatrix}.
\end{aligned}
\end{equation}

Consequently, $m^{(1)}(x,t,k)$ satisfies the following mixed
Riemann--Hilbert--$\dbar$ problem, which will henceforth be abbreviated
as the RH--$\dbar$ problem.
\begin{RHdbarproblem}
Find a $3\times3$ matrix-valued function $m^{(1)}(x,t,k)$ satisfying
the following properties:
\begin{enumerate}
    \item[(i)]
    The function $m^{(1)}(x,t,k)$ is continuous in
    $\mathbb{C}\setminus\Sigma^{(1)}$, where the contour
    $\Sigma^{(1)}$ is depicted in Figure~\ref{Sigma1}. Moreover,
    $m^{(1)}(x,t,k)$ admits continuous boundary values
    $m^{(1)}_{\pm}(x,t,k)$ on $\Sigma^{(1)}$.

    \item[(ii)]
    The boundary values of $m^{(1)}(x,t,k)$ satisfy
    \begin{equation*}
        m^{(1)}_+(x,t,k)
        =
        m^{(1)}_-(x,t,k)v^{(1)}_j(x,t,k),
        \qquad
        k\in\Sigma_j^{(1)},\quad j=1,2,3.
    \end{equation*}
\begin{align*}
v_1^{(1)}(x,t,k)
&=
\begin{pmatrix}
1
&
-r_1(k)e^{-\vartheta_{21}(x,t,k)}
&
\bigl[r_1(k)r_2^*(0)+r_2(0)\bigr]
e^{-\vartheta_{31}(x,t,k)}
\\
r_1^*(k)e^{\vartheta_{21}(x,t,k)}
&
1-|r_1(k)|^2
&
-r_2^*(0)\bigl(1-|r_1(k)|^2\bigr)
e^{-\vartheta_{32}(x,t,k)}
\\
0&0&1
\end{pmatrix},
\\[3mm]
v_2^{(1)}(x,t,k)
&=
\begin{pmatrix}
1-|r_1(\omega^2k)|^2
&
-r_2^*(0)\bigl(1-|r_1(\omega^2k)|^2\bigr)
e^{-\vartheta_{21}(x,t,k)}
&
r_1^*(\omega^2k)e^{-\vartheta_{31}(x,t,k)}
\\
0&1&0
\\
-r_1(\omega^2k)e^{\vartheta_{31}(x,t,k)}
&
\bigl[r_1(\omega^2k)r_2^*(0)+r_2(0)\bigr]
e^{\vartheta_{32}(x,t,k)}
&
1
\end{pmatrix},
\\[3mm]
v_3^{(1)}(x,t,k)
&=
\begin{pmatrix}
1&0&0
\\
\bigl[r_1(\omega k)r_2^*(0)+r_2(0)\bigr]
e^{\vartheta_{21}(x,t,k)}
&
1
&
-r_1(\omega k)e^{-\vartheta_{32}(x,t,k)}
\\
-r_2^*(0)\bigl(1-|r_1(\omega k)|^2\bigr)
e^{\vartheta_{31}(x,t,k)}
&
r_1^*(\omega k)e^{\vartheta_{32}(x,t,k)}
&
1-|r_1(\omega k)|^2
\end{pmatrix}.
\end{align*}
    \item[(iii)]
    Away from $\Sigma^{(1)}$, the function $m^{(1)}(x,t,k)$ satisfies
    the $\dbar$ equation
    \begin{equation*}
        \dbar m^{(1)}(x,t,k)
        =
        m^{(1)}(x,t,k)\dbar\mathcal{G}(x,t,k),
    \end{equation*}
where $\mathcal{G}(x,t,k)$ is defined in \eqref{def:mathcal-Gj}.
    \item[(iv)]
    As $k\to\infty$,
    \begin{equation*}
        m^{(1)}(x,t,k)
        =
        I+\mathcal{O}\left(k^{-1}\right).
    \end{equation*}

    \item[(v)]
    The function $m^{(1)}(x,t,k)$ satisfies the $\mathbb{Z}_2$ and
    $\mathbb{Z}_3$ symmetries
    \begin{equation*}
        \mathcal{A}m^{(1)}(\omega k;x,t)\mathcal{A}^{-1}
        =
        m^{(1)}(x,t,k)
        =
        \mathcal{B}
        \overline{m^{(1)}(\bar{k};x,t)}
        \mathcal{B},
        \qquad
        k\in\mathbb{C}\setminus\Sigma^{(1)}.
    \end{equation*}
\end{enumerate}
\end{RHdbarproblem}

\begin{figure}[H]
    \centering
    \begin{tikzpicture}[
        >=latex,
        directed/.style={
            very thick,
            postaction={decorate},
            decoration={
                markings,
                mark=at position 0.5 with {\arrow{>}}
            }
        }
    ]
        \draw[directed]
            (0,0) -- (4,0);
        \node[above] at (2.7,0) {$\Sigma^{(1)}_1$};

        \draw[directed]
            (0,0) -- (-2,3.464);
        \node[left] at (-1.35,2.35) {$\Sigma^{(1)}_2$};

        \draw[directed]
            (0,0) -- (-2,-3.464);
        \node[left] at (-1.35,-2.35) {$\Sigma^{(1)}_3$};

        \draw[very thick,dashed]
            (0,0) -- (2,3.464);

        \draw[very thick,dashed]
            (0,0) -- (-4,0);

        \draw[very thick,dashed]
            (0,0) -- (2,-3.464);

        \draw[thin]
            (0.85,0)
            arc[start angle=0,end angle=60,radius=0.85];
        \node at (30:1.18) {$\dfrac{\pi}{3}$};

        \node[below left] at (0,0) {$0$};
        \node[right] at (4,0) {$\mathbb{R}$};
    \end{tikzpicture}
    \caption{The contour
    $\Sigma^{(1)}=\bigcup_{j=1}^{3}\Sigma_j^{(1)}$.
    The dashed rays indicate the removed parts of the original contour
    $\Sigma$.}
    \label{Sigma1}
\end{figure}

\subsection{The second deformation}

Before carrying out the second deformation of the mixed
Riemann--Hilbert--$\dbar$ problem, we introduce the scalar function
$\delta_1(k)$, which is analytic in
$\mathbb{C}\setminus[0,\infty)$ and satisfies the jump condition
\[
\delta_{1,+}(k)
=
\delta_{1,-}(k)\bigl(1-|r_1(k)|^2\bigr),
\qquad k\in(0,\infty),
\]
together with the normalization condition
\[
\delta_1(k)=1+\mathcal{O}(k^{-1}),
\qquad k\to\infty.
\]
By the Plemelj formula, $\delta_1(k)$ is given by
\begin{equation}\label{delta1o}
\delta_1(k)
=
\exp\left(
\frac{1}{2\pi i}
\int_0^\infty
\frac{\ln\bigl(1-|r_1(s)|^2\bigr)}{s-k}\,\mathrm{d}s
\right),
\qquad
k\in\mathbb{C}\setminus\mathbb{R}_+.
\end{equation}

\begin{proposition}\label{deltaproposition}
Suppose that
\(
r_1(k)\in H^1(\mathbb{R}_+),
\
\|r_1(k)\|_{L^\infty(\mathbb{R}_+)}\leq\sigma<1.
\)
Then function $\delta_1(k)$ has the following properties:
\begin{enumerate}
    \item[(i)]
    The function $\delta_1(k)$ can be written as
    \begin{equation}\label{delta1}
        \delta_1(k)
        =
        e^{-i\nu\log_0(k)}e^{-\chi_1(k)},
    \end{equation}
    where
    \[
        \nu
        =
        -\frac{1}{2\pi}
        \ln\bigl(1-|r_1(0)|^2\bigr)
    \]
    and
    \[
        \chi_1(k)
        =
        \frac{1}{2\pi i}
        \int_0^\infty
        \log_0(k-s)\,
        \mathrm{d}\ln\bigl(1-|r_1(s)|^2\bigr).
    \]

    \item[(ii)]
    The functions $\delta_1(k)$ and $\delta_1(k)^{-1}$ are
    analytic in $\mathbb{C}\setminus[0,\infty)$ and admit continuous
    boundary values on $(0,\infty)$. Moreover,
    \[
        \sup_{k\in\mathbb{C}\setminus[0,\infty)}
        |\delta_1(k)^{\pm1}|<\infty.
    \]

    \item[(iii)]
    As $k\to0$ along a path that is non-tangential to $(0,\infty)$,
    \begin{equation*}
        |\chi_1(k)-\chi_1(0)|
        \leq C|k|^{1/2}.
    \end{equation*}
\end{enumerate}
\end{proposition}

\begin{proof}
Properties \emph{(i)} and \emph{(ii)} follow from integration by
parts in \eqref{delta1o}, the Plemelj formula, and the inequality
$\|r_1\|_{L^\infty}\leq\sigma<1$. We only give the details for
\emph{(iii)}.

Set $h(s):=\ln(1-|r_1(s)|^2)$. Since
$\|r_1\|_{L^\infty(\mathbb{R}_+)}\leq\sigma<1$, it follows that
\[
\|h'\|_{L^2(\mathbb{R}_+)}
\leq
\frac{2\sigma}{1-\sigma^2}
\|r_1'\|_{L^2(\mathbb{R}_+)}.
\]
By the Cauchy--Schwarz inequality and the standard estimate
\[
\left\|
\log_0(k-\cdot)-\log_0(-\cdot)
\right\|_{L^2(\mathbb{R}_+)}
\leq C|k|^{1/2},
\]
which holds as $k\to0$ non-tangentially to $(0,\infty)$, we obtain
\[
\begin{aligned}
|\chi_1(k)-\chi_1(0)|
&\leq
\frac{1}{2\pi}
\|h'\|_{L^2(\mathbb{R}_+)}
\left\|
\log_0(k-\cdot)-\log_0(-\cdot)
\right\|_{L^2(\mathbb{R}_+)}\leq C|k|^{1/2}.
\end{aligned}
\]
\end{proof}

Similarly, define
\[
\delta_3(k):=\delta_1(\omega^2k),
\qquad
\delta_5(k):=\delta_1(\omega k).
\]
More explicitly,
\begin{equation*}
\begin{aligned}
\delta_3(k)
&=
e^{-i\nu\log_0(\omega^2k)}
e^{-\chi_1(\omega^2k)},
&&
k\in\mathbb{C}\setminus\omega\mathbb{R}_+,
\\
\delta_5(k)
&=
e^{-i\nu\log_0(\omega k)}
e^{-\chi_1(\omega k)},
&&
k\in\mathbb{C}\setminus\omega^2\mathbb{R}_+.
\end{aligned}
\end{equation*}
We then introduce the diagonal matrix
\begin{equation}\label{Deltadefination}
\Delta(k)
=
\begin{pmatrix}
\dfrac{\delta_1(k)}{\delta_3(k)} & 0 & 0
\\[2mm]
0 & \dfrac{\delta_5(k)}{\delta_1(k)} & 0
\\[2mm]
0 & 0 & \dfrac{\delta_3(k)}{\delta_5(k)}
\end{pmatrix}.
\end{equation}

Define the matrix-valued function $m^{(2)}(x,t,k)$ by
\begin{equation*}
m^{(2)}(x,t,k)
=
m^{(1)}(x,t,k)\Delta(k)\mathcal{R}(x,t,k).
\end{equation*}
To construct $\mathcal{R}(x,t,k)$, we introduce the scalar extension
functions $R_j(k)$ in the corresponding regions $\Omega_j$ shown in
Figure~\ref{Sigma2}. For convenience, set
\[
\rho_j(k)
:=
\frac{r_j(k)}
     {1-r_j(k)r_j^*(k)},
\qquad j=1,2.
\]
The boundary values of the relevant functions $R_j(k)$ are prescribed
as follows:
\begin{equation}\label{boundary:Rj}
\begin{aligned}
R_1(k)
&=
\begin{cases}
\displaystyle
\rho_1^*(k)
\frac{\delta_{1,+}^2(k)}
     {\delta_3(k)\delta_5(k)},
& k\in[k_0,\infty),
\\[3mm]
\displaystyle
\rho_1^*(k_0)
\frac{\delta_1^2(k)}
     {\delta_3(k)\delta_5(k)},
& k\in\Sigma_1^{(2)},
\end{cases}
&
R_{6}(k)
&=
\begin{cases}
\displaystyle
\rho_1(k)
\frac{\delta_3(k)\delta_5(k)}
     {\delta_{1,-}^2(k)},
& k\in[k_0,\infty),
\\[3mm]
\displaystyle
\rho_1(k_0)
\frac{\delta_3(k)\delta_5(k)}
     {\delta_1^2(k)},
& k\in\Sigma_6^{(2)},
\end{cases}
\\[5mm]
R_2(k)
&=
\begin{cases}
\displaystyle
\rho_1(\omega^2k)
\frac{\delta_1(k)\delta_5(k)}
     {\delta_{3,-}^2(k)},
& k\in\omega[k_0,\infty),
\\[3mm]
\displaystyle
\rho_1(k_0)
\frac{\delta_1(k)\delta_5(k)}
     {\delta_3^2(k)},
& k\in\Sigma_2^{(2)},
\end{cases}
&
R_3(k)
&=
\begin{cases}
\displaystyle
\rho_1^*(\omega^2k)
\frac{\delta_{3,+}^2(k)}
     {\delta_1(k)\delta_5(k)},
& k\in\omega[k_0,\infty),
\\[3mm]
\displaystyle
\rho_1^*(k_0)
\frac{\delta_3^2(k)}
     {\delta_1(k)\delta_5(k)},
& k\in\Sigma_3^{(2)},
\end{cases}
\\[5mm]
R_4(k)
&=
\begin{cases}
\displaystyle
\rho_1(\omega k)
\frac{\delta_1(k)\delta_3(k)}
     {\delta_{5,-}^2(k)},
& k\in\omega^2[k_0,\infty),
\\[3mm]
\displaystyle
\rho_1(k_0)
\frac{\delta_1(k)\delta_3(k)}
     {\delta_5^2(k)},
& k\in\Sigma_4^{(2)},
\end{cases}
&
R_5(k)
&=
\begin{cases}
\displaystyle
\rho_1^*(\omega k)
\frac{\delta_{5,+}^2(k)}
     {\delta_1(k)\delta_3(k)},
& k\in\omega^2[k_0,\infty),
\\[3mm]
\displaystyle
\rho_1^*(k_0)
\frac{\delta_5^2(k)}
     {\delta_1(k)\delta_3(k)},
& k\in\Sigma_5^{(2)}.
\end{cases}
\end{aligned}
\end{equation}

\begin{lemma}\label{lem:dbar-extension-Rj}
Suppose that
\(
r_1(k)\in H^1(\mathbb R_+),
\
\|r_1(k)\|_{L^\infty(\mathbb R_+)}
\leq \sigma<1.
\)
Then, for each $j=1,\cdots,6$, there exists a function
\(
R_j(k)\in C(\overline{\Omega_j})
\)
satisfying the boundary conditions in \eqref{boundary:Rj}.
Furthermore, we have the following estimate
\begin{equation*}
\left|\dbar R_j(k)\right|
\leq
C\left(
|k-\alpha_jk_0|^{-1/2}
\|r_1'\|_{L^2(\mathbb R_+)}
+
\left|
r_1'\bigl(\Re(\alpha_j^2k)\bigr)
\right|
\right),
\end{equation*}
where $\alpha_1=\alpha_{6}=1,\alpha_2=\alpha_{3}=\omega$
and $\alpha_4=\alpha_{5}=\omega^2$.
\end{lemma}
\begin{proof}
The proof is analogous to that of Lemma~\ref{lem:dbar-extension-r2}, and the details are therefore omitted.
\end{proof}

\begin{figure}[h]
    \centering
    \begin{tikzpicture}[
        >=latex,
        directed/.style={
            very thick,
            postaction={decorate},
            decoration={
                markings,
                mark=at position 0.5 with {\arrow{>}}
            }
        }
    ]
        \draw[directed]
            (0,0) -- (1.6,0)
            node[pos=0.72,below] {$\small 7$};

        \draw[directed]
            (0,0) -- (-0.8,1.386)
            node[pos=0.72,left] {$\small 8$};

        \draw[directed]
            (0,0) -- (-0.8,-1.386)
            node[pos=0.72,left] {$\small 9$};

        \draw[very thick,dashed] (0,0) -- (-1.6,0);
        \draw[very thick,dashed] (0,0) -- (0.8,1.386);
        \draw[very thick,dashed] (0,0) -- (0.8,-1.386);

        \draw[very thick,dashed] (1.6,0) -- (4,0);
        \draw[very thick,dashed] (-1.6,0) -- (-4,0);

        \draw[very thick,dashed]
            (0.8,1.386) -- (2,3.464);
        \draw[very thick,dashed]
            (-0.8,-1.386) -- (-2,-3.464);

        \draw[very thick,dashed]
            (-0.8,1.386) -- (-2,3.464);
        \draw[very thick,dashed]
            (0.8,-1.386) -- (2,-3.464);

        \filldraw[black]
            (1.6,0)
            node[below=0.5mm] {$k_0$}
            circle (1.5pt);

        \filldraw[black]
            (-0.8,1.386)
            node[right=1mm] {$\omega k_0$}
            circle (1.5pt);

        \filldraw[black]
            (-0.8,-1.386)
            node[right=1mm] {$\omega^2k_0$}
            circle (1.5pt);

        \draw[directed]
            (1.6,0) -- (3.678,1.2)
            node[pos=0.72,above] {$\small 1$};

        \draw[directed]
            (1.6,0) -- (3.678,-1.2)
            node[pos=0.72,below] {$\small 6$};

        \draw[directed]
            (-0.8,1.386) -- (-0.8,3.786)
            node[pos=0.72,right] {$\small 2$};

        \draw[directed]
            (-0.8,1.386) -- (-2.878,2.586)
            node[pos=0.72,above] {$\small 3$};

        \draw[directed]
            (-0.8,-1.386) -- (-2.878,-2.586)
            node[pos=0.72,above] {$\small 4$};

        \draw[directed]
            (-0.8,-1.386) -- (-0.8,-3.786)
            node[pos=0.72,right] {$\small 5$};

        \draw[thin]
            (2.15,0)
            arc[start angle=0,end angle=30,radius=0.55];
        \node at (2.5,0) {$\dfrac{\pi}{3}$};
        \draw[thin]
            (2.15,0)
            arc[start angle=0,end angle=-30,radius=0.55];

        \node at (3.15,0.55)   {$\Omega_1$};
        \node at (1.60,1.65)   {$\Omega_7$};
        \node at (0,2.30)      {$\Omega_8$};
        \node at (-1.20,2.70)  {$\Omega_2$};
        \node at (-1.70,2.40)  {$\Omega_3$};
        \node at (-2.70,0.80)  {$\Omega_9$};

        \node at (-2.70,-0.80) {$\Omega_{10}$};
        \node at (-1.70,-2.40) {$\Omega_4$};
        \node at (-1.20,-2.70) {$\Omega_5$};
        \node at (0,-2.30)     {$\Omega_{11}$};
        \node at (1.60,-1.65)  {$\Omega_{12}$};
        \node at (3.15,-0.55)  {$\Omega_{6}$};

        \node[right] at (4,0) {$\mathbb{R}$};
        \node[above] at (-2,3.464) {$\omega\mathbb{R}$};
        \node[below] at (-2,-3.464) {$\omega^2\mathbb{R}$};

        \node[below right] at (0,0) {$0$};
    \end{tikzpicture}
    \caption{\protect\small
        The jump contour $\Sigma^{(2)}$ and the saddle points
        $\omega^j k_0$ for $j=0,1,2$.}
    \label{Sigma2}
\end{figure}

Define the matrix-valued function $\mathcal{R}(x,t,k)$ piecewise by setting $\mathcal{R}(x,t,k)=\mathcal{R}_j(x,t,k)$ for $k\in\Omega_j$, $j=1,\ldots,12$, where the matrices $\mathcal{R}_j(x,t,k)$ are defined as follows:
\begin{align}
\mathcal{R}_1(x,t,k)
&:=
\begin{pmatrix}
1&0&0\\[1mm]
-R_1(k)e^{\vartheta_{21}(x,t,k)}
&
1
&
\displaystyle
r_2^*(0)
\frac{\delta_1(k)\delta_3(k)}
     {\delta_5^2(k)}
e^{-\vartheta_{32}(x,t,k)}
\\[3mm]
0&0&1
\end{pmatrix},
\qquad k\in\Omega_1,
\label{def:mathcal-R1}
\\[4mm]
\mathcal{R}_{6}(x,t,k)
&:=
\begin{pmatrix}
1
&
-R_{6}(k)e^{-\vartheta_{21}(x,t,k)}
&
\displaystyle
r_2(0)
\frac{\delta_3^2(k)}
     {\delta_1(k)\delta_5(k)}
e^{-\vartheta_{31}(x,t,k)}
\\[3mm]
0&1&0\\
0&0&1
\end{pmatrix},
\qquad k\in\Omega_{6}.
\end{align}

Using the symmetry relations, we define
\[
\begin{aligned}
\mathcal R_2(x,t,k)
&=
\mathcal A^{-1}
\mathcal R_6(\omega^2k;x,t)
\mathcal A,
&
\mathcal R_4(x,t,k)
&=
\mathcal A^{-2}
\mathcal R_6(\omega k;x,t)
\mathcal A^2,\\
\mathcal R_3(x,t,k)
&=
\mathcal A^{-1}
\mathcal R_1(\omega^2k;x,t)
\mathcal A,
&
\mathcal R_5(x,t,k)
&=
\mathcal A^{-2}
\mathcal R_1(\omega k;x,t)
\mathcal A^2.
\end{aligned}
\]
For the remaining regions, we set
$\mathcal R_j(x,t,k)=I$ for $j=7,\ldots,12$.

As a result, we derive the following RH $\dbar$-Problem for $3\times3$ matrix valued function $m^{(2)}(x,t,k)$.
\begin{RHdbarproblem}\label{rhp-4.5}
Find a $3\times3$ matrix-valued function $m^{(2)}(x,t,k)$
satisfying the following properties.

\begin{enumerate}
\item[(i)]
The function $m^{(2)}(x,t,k)$ is continuous in
$\mathbb C\setminus\Sigma^{(2)}$, where the contour
$\Sigma^{(2)}$ is depicted in Figure~\ref{Sigma2}.
Moreover, $m^{(2)}(x,t,k)$ admits continuous boundary values
$m_\pm^{(2)}(x,t,k)$ on each smooth component of
$\Sigma^{(2)}$.

\item[(ii)]
The boundary values of $m^{(2)}(x,t,k)$ satisfy
\[
m_+^{(2)}(x,t,k)
=
m_-^{(2)}(x,t,k)v_j^{(2)}(x,t,k),
\qquad
k\in\Sigma_j^{(2)},\quad j=1,\ldots,9.
\]
The jump matrices on $\Sigma_1^{(2)}$, $\Sigma_6^{(2)}$, and
$\Sigma_7^{(2)}$ are given by
\begin{equation}\label{jumps2}
\scalebox{0.85}{$
\begin{aligned}
v_1^{(2)}
&=
\begin{pmatrix}
1&0&0\\
R_1(k)e^{\vartheta_{21}(x,t,k)}
&
1
&
\displaystyle
-r_2^*(0)
\frac{\delta_1(k)\delta_3(k)}
     {\delta_5^2(k)}
e^{-\vartheta_{32}(x,t,k)}
\\
0&0&1
\end{pmatrix},\
v_6^{(2)}
=
\begin{pmatrix}
1
&
-R_6(k)e^{-\vartheta_{21}(x,t,k)}
&
\displaystyle
r_2(0)
\frac{\delta_3^2(k)}
     {\delta_1(k)\delta_5(k)}
e^{-\vartheta_{31}(x,t,k)}
\\
0&1&0\\
0&0&1
\end{pmatrix},
\\[4mm]
v_7^{(2)}
&=
\begin{pmatrix}
1-|r_1(k)|^2
&
\displaystyle
-\rho_1(k)
\frac{\delta_3(k)\delta_5(k)}
     {\delta_{1,-}^2(k)}
e^{-\vartheta_{21}(x,t,k)}
&
\displaystyle
\bigl[r_1(k)r_2^*(0)+r_2(0)\bigr]
\frac{\delta_3^2(k)}
     {\delta_{1,-}(k)\delta_5(k)}
e^{-\vartheta_{31}(x,t,k)}
\\[4mm]
\displaystyle
\rho_1^*(k)
\frac{\delta_{1,+}^2(k)}
     {\delta_3(k)\delta_5(k)}
e^{\vartheta_{21}(x,t,k)}
&
1
&
\displaystyle
-r_2^*(0)
\frac{\delta_{1,+}(k)\delta_3(k)}
     {\delta_5^2(k)}
e^{-\vartheta_{32}(x,t,k)}
\\
0&0&1
\end{pmatrix}.
\end{aligned}
$}
\end{equation}
The remaining jump matrices are determined by the
$\mathbb Z_3$ symmetry:
\[
\begin{aligned}
v_3^{(2)}(x,t,k)
&=
\mathcal A^{-1}
v_1^{(2)}(\omega^2k;x,t)
\mathcal A,
&
v_5^{(2)}(x,t,k)
&=
\mathcal A^{-2}
v_1^{(2)}(\omega k;x,t)
\mathcal A^2,
\\
v_2^{(2)}(x,t,k)
&=
\mathcal A^{-1}
v_6^{(2)}(\omega^2k;x,t)
\mathcal A,
&
v_4^{(2)}(x,t,k)
&=
\mathcal A^{-2}
v_6^{(2)}(\omega k;x,t)
\mathcal A^2,
\\
v_8^{(2)}(x,t,k)
&=
\mathcal A^{-1}
v_7^{(2)}(\omega^2k;x,t)
\mathcal A,
&
v_9^{(2)}(x,t,k)
&=
\mathcal A^{-2}
v_7^{(2)}(\omega k;x,t)
\mathcal A^2.
\end{aligned}
\]

\item[(iii)]
For $k\in\mathbb C\setminus\Sigma^{(2)}$, the function
$m^{(2)}(x,t,k)$ satisfies
\begin{equation}\label{dbar:m2-RHP}
\dbar m^{(2)}(x,t,k)
=
m^{(2)}(x,t,k)\mathcal R^{(2)}(x,t,k),
\end{equation}
where
\[
\mathcal R^{(2)}(x,t,k)
=
\mathcal R_j^{(2)}(x,t,k),
\qquad k\in\Omega_j,\quad j=1,\ldots,12.
\]
The functions in $\Omega_1$, $\Omega_6$, $\Omega_7$, and
$\Omega_{12}$ are given by
\begin{align*}
\mathcal R_1^{(2)}(x,t,k)
&=
\begin{pmatrix}
0&0&0\\
-\dbar R_1(k)e^{\vartheta_{21}(x,t,k)}
&
0
&
\displaystyle
-\dbar G_1(k)
\frac{\delta_1(k)\delta_3(k)}
     {\delta_5^2(k)}
e^{-\vartheta_{32}(x,t,k)}
\\
0&0&0
\end{pmatrix},
\\[4mm]
\mathcal R_6^{(2)}(x,t,k)
&=
\begin{pmatrix}
0
&
-\dbar R_6(k)e^{-\vartheta_{21}(x,t,k)}
&
\displaystyle
-\dbar G_6(k)
\frac{\delta_3^2(k)}
     {\delta_1(k)\delta_5(k)}
e^{-\vartheta_{31}(x,t,k)}
\\
0&0&0\\
0&0&0
\end{pmatrix},
\\[4mm]
\mathcal R_7^{(2)}(x,t,k)
&=
\begin{pmatrix}
0&0&0\\
0&0&
\displaystyle
-\dbar G_1(k)
\frac{\delta_1(k)\delta_3(k)}
     {\delta_5^2(k)}
e^{-\vartheta_{32}(x,t,k)}
\\
0&0&0
\end{pmatrix},
\\[4mm]
\mathcal R_{12}^{(2)}(x,t,k)
&=
\begin{pmatrix}
0&0&
\displaystyle
-\dbar G_6(k)
\frac{\delta_3^2(k)}
     {\delta_1(k)\delta_5(k)}
e^{-\vartheta_{31}(x,t,k)}
\\
0&0&0\\
0&0&0
\end{pmatrix}.
\end{align*}
The remaining functions are determined by symmetry:
\[
\begin{aligned}
\mathcal R_3^{(2)}(x,t,k)
&=
\omega\mathcal A^{-1}
\mathcal R_1^{(2)}(\omega^2k;x,t)
\mathcal A,
&
\mathcal R_5^{(2)}(x,t,k)
&=
\omega^2\mathcal A^{-2}
\mathcal R_1^{(2)}(\omega k;x,t)
\mathcal A^2,
\\
\mathcal R_2^{(2)}(x,t,k)
&=
\omega\mathcal A^{-1}
\mathcal R_6^{(2)}(\omega^2k;x,t)
\mathcal A,
&
\mathcal R_4^{(2)}(x,t,k)
&=
\omega^2\mathcal A^{-2}
\mathcal R_6^{(2)}(\omega k;x,t)
\mathcal A^2,
\\
\mathcal R_9^{(2)}(x,t,k)
&=
\omega\mathcal A^{-1}
\mathcal R_7^{(2)}(\omega^2k;x,t)
\mathcal A,
&
\mathcal R_{11}^{(2)}(x,t,k)
&=
\omega^2\mathcal A^{-2}
\mathcal R_7^{(2)}(\omega k;x,t)
\mathcal A^2,
\\
\mathcal R_8^{(2)}(x,t,k)
&=
\omega\mathcal A^{-1}
\mathcal R_{12}^{(2)}(\omega^2k;x,t)
\mathcal A,
&
\mathcal R_{10}^{(2)}(x,t,k)
&=
\omega^2\mathcal A^{-2}
\mathcal R_{12}^{(2)}(\omega k;x,t)
\mathcal A^2.
\end{aligned}
\]

\item[(iv)]
As $k\to\infty$, we have
\[
m^{(2)}(x,t,k)
=
I+\mathcal O(k^{-1}).
\]

\item[(v)]
The function $m^{(2)}(x,t,k)$ satisfies the
$\mathbb Z_2$ and $\mathbb Z_3$ symmetries
\[
\mathcal A m^{(2)}(\omega k;x,t)\mathcal A^{-1}
=
m^{(2)}(x,t,k)
=
\mathcal B
\overline{m^{(2)}(\bar k;x,t)}
\mathcal B,
\qquad
k\in\mathbb C\setminus\Sigma^{(2)}.
\]
\end{enumerate}
\end{RHdbarproblem}

\subsection{Analysis of the Pure RH Problem}

In what follows, we analyze the pure Riemann-Hilbert part of the mixed
RH--$\dbar$ problem for $m^{(2)}(x,t,k)$. To this end, we introduce
a $3\times3$ matrix-valued function $S(x,t,k)$ having the same jump
conditions as $m^{(2)}(x,t,k)$.

\begin{RHproblem}\label{rhp:S}
Find a $3\times3$ matrix-valued function $S(x,t,k)$ satisfying the
following properties:
\begin{enumerate}
\item[(i)]
The function $S(x,t,k)$ is analytic in
$\mathbb C\setminus\Sigma^{(2)}$ and admits continuous boundary values
$S_\pm(x,t,k)$ on each smooth component of $\Sigma^{(2)}$.

\item[(ii)]
The boundary values satisfy
\[
S_+(x,t,k)
=
S_-(x,t,k)v^{(2)}(x,t,k),
\qquad k\in\Sigma^{(2)},
\]
where the jump matrix $v^{(2)}(x,t,k)$ is defined in
\eqref{jumps2}.

\item[(iii)]
As $k\to\infty$,
\[
S(x,t,k)
=
I+\mathcal O(k^{-1}).
\]

\item[(iv)]
The function $S(x,t,k)$ satisfies the $\mathbb Z_2$ and
$\mathbb Z_3$ symmetries
\[
\mathcal A S(\omega k;x,t)\mathcal A^{-1}
=
S(x,t,k)
=
\mathcal B\overline{S(\bar k;x,t)}\mathcal B,
\qquad
k\in\mathbb C\setminus\Sigma^{(2)}.
\]
\end{enumerate}
\end{RHproblem}

Before analyzing RH problem~\ref{rhp:S}, we first record several useful identities satisfied by the functions $\delta_j(k)$,
$j=1,3,5$.
\begin{proposition}\label{deltaproposition1}
    Restrict the spectral parameter \( k \) to the contours \( \Sigma_{1,6,7}^{(2)} \) in Figure \ref{Sigma2}. The entries of the jump matrices involving \( \delta_{j}(k) ~(j = 1, 3, 5)\) can be rewritten as follows:
    \begin{enumerate}
        \item[$(\rm i)$] On \( \Sigma_1^{(2)} \),
        $
        \frac{\delta_1^2}{\delta_3 \delta_5} = e^{-2\pi \nu} e^{-(2\chi_1(k) - \chi_1(\omega k) - \chi_1(\omega^2 k))}, \quad
        \frac{\delta_1 \delta_3}{\delta_5^2} = e^{-(\chi_1(k) + \chi_1(\omega^2 k) - 2\chi_1(\omega k))}.
        $
        \item[$(\rm ii)$] On \( \Sigma_6^{(2)} \),
        $
        \frac{\delta_3 \delta_5}{\delta_1^2} = e^{-2\pi \nu} e^{(2\chi_1(k) - \chi_1(\omega k) - \chi_1(\omega^2 k))}, \quad
        \frac{\delta_3^2}{\delta_1 \delta_5} = e^{-( -\chi_1(k) + 2\chi_1(\omega^2 k) - \chi_1(\omega k))}.
        $
        \item[$(\rm iii)$] On \( \Sigma_{7}^{(2)} \), the following identities hold:
        \[
        \begin{cases}
        \begin{aligned}
        &\frac{\delta_{1+}^2}{\delta_{1-}^2} = e^{-4\pi \nu} e^{-(2\chi_{1+}(k) - 2\chi_{1-}(k))}, \\
        &\frac{\delta_3 \delta_5}{\delta_{1-}^2} = e^{-2\pi \nu} e^{(2\chi_{1-}(k) - \chi_1(\omega k) - \chi_1(\omega^2 k))}, \quad
        \frac{\delta_{1+}^2}{\delta_3 \delta_5} = e^{-2\pi \nu} e^{-(2\chi_{1+}(k) - \chi_1(\omega k) - \chi_1(\omega^2 k))}, \\
        &\frac{\delta_{3}^2}{\delta_{1-} \delta_5} = e^{-(2\chi_1(\omega^2 k) - \chi_{1-}(k) - \chi_1(\omega k))}, \quad
        \frac{\delta_{1+} \delta_3}{\delta_{5}^2} = e^{-(\chi_1(\omega^2 k) + \chi_{1+}(k) - 2\chi_1(\omega k))}.
        \end{aligned}
        \end{cases}
        \]
    \end{enumerate}
\end{proposition}
\begin{proof}
    The proposition can be proved directly from the formulas in (\ref{delta1o}) and (\ref{delta1}). However, it should be noted that the argument of \( \log_0(k) \) lies within the interval \( [0, 2\pi) \). Specifically, one has
\[
\frac{\delta_1^2}{\delta_3 \delta_5} = e^{i\nu \left(-2\log_0(k) + \log_0(\omega^2 k) + \log_0(\omega k)\right)} e^{-(2\chi_1(k) - \chi_1(\omega k) - \chi_1(\omega^2 k))}.
\]
This implies that for \( \arg(k) \in (0, \frac{\pi}{6}] \), it follows that
$
e^{i\nu \left(-2\log_0(k) + \log_0(\omega^2 k) + \log_0(\omega k)\right)} = e^{-2\pi \nu}.
$
The other properties follow in a similar manner and are omitted for brevity.
\end{proof}

In what follows, we analyze RH problem~\ref{rhp:S} in the
self-similar region. For simplicity, we restrict our attention to
\(
0\leq \frac{x}{\sqrt{t}}<c_1;
\)
the case $x<0$ can be treated analogously. In this region, the critical
point satisfies $k_0=\mathcal O(t^{-1/2})$ as $t\to\infty$. Consequently,
the three critical points $k_0$, $\omega k_0$, and $\omega^2k_0$
coalesce at the origin, and their local analysis can be incorporated
into a single local model centered at $k=0$. Define \( D_{\epsilon}(0) := \{ k \in \mathbb{C} : |k| < \epsilon \} \), and let \( \omega^j k_0 \in D_{\epsilon}(0) \) for \( j = 0, 1, 2 \). 

To construct the local parametrix for $S(x,t,k)$, we
introduce, for $k\in D_\epsilon(0)$, the self-similar scaling
\begin{equation}\label{palnetrans}
\lambda
=
-\frac{2\sqrt{t}}{3\sqrt{3}}\,k,
\qquad
y
=
-\frac{\sqrt{3}\,x}{2\sqrt{t}}.
\end{equation}
Under this change of variables, the phase functions
$\vartheta_{ij}(x,t,k)$, $1\leq j<i\leq3$, take the form
\begin{equation*}
\vartheta_{ij}(y;\lambda)
=
3\left[
(\omega^i-\omega^j)y\lambda
+
\frac{9}{4}
(\omega^{2i}-\omega^{2j})\lambda^2
\right].
\end{equation*}
Furthermore, since $k_0=\mathcal O(t^{-1/2})$, the continuity of the
reflection data at the origin implies that
$\rho_1(k_0)\to\rho_1(0)$ and $r_1(k_0)\to r_1(0)$ as
$t\to\infty$. Together with Proposition~\ref{deltaproposition1},
this shows that, under the scaling \eqref{palnetrans}, the jump
matrices inside $D_\epsilon(0)$ are asymptotically matched by those of
the Painlevé IV model RH problem for $m^p(y;\lambda)$ introduced in
Appendix~\ref{Appendix}.

\begin{lemma}\label{v2-vp}
Let $C_0>0$ and suppose that
\(
0\leq x/\sqrt{t}\leq C_0.
\)
Assume that the reflection coefficients satisfy the hypotheses of
Lemmas~\ref{lem:dbar-extension-r2} and
\ref{lem:dbar-extension-Rj}. Then, as $t\to\infty$, the jump matrices
$v^{(2)}$ of RH problem~\ref{rhp:S} and the jump matrices $v^p$ of the
Painlevé IV model RH problem satisfy
\begin{equation*}
\left\|v^{(2)}-v^p\right\|_
{L^p(\Sigma^{(2)}\cap D_\epsilon(0))}
=
\mathcal O\left(
t^{-\frac14-\frac{1}{2p}}
\right),
\qquad 1\leq p\leq\infty,
\end{equation*}
where $1/(2p)$ is understood to be zero when $p=\infty$. 

Moreover, writing
\(
m^p(k):=m^p(y;\lambda(k)),
\)
we have
\begin{equation}\label{invmpestimate}
\left\|(m^p)^{-1}-I\right\|_
{L^1(\partial D_\epsilon(0))}
+
\left\|(m^p)^{-1}-I\right\|_
{L^\infty(\partial D_\epsilon(0))}
=
\mathcal O(t^{-1/2}),
\end{equation}
and, with $\partial D_\epsilon(0)$ oriented counterclockwise,
\begin{equation}\label{integral:mp}
\frac{1}{2\pi i}
\int_{\partial D_\epsilon(0)}
\left((m^p)^{-1}(k)-I\right)\,\mathrm dk
=
\frac{3\sqrt{3}}{2\sqrt{t}}\,m_1^p(y)
+
\mathcal O(t^{-1}).
\end{equation}
All estimates are uniform for
\(0\leq x/\sqrt{t}\leq C_0\).
\end{lemma}

\begin{proof}
Under the assumptions of
Lemma~\ref{lem:dbar-extension-Rj}, we have
$r_1(k), \rho_1(k)\in H^1(\mathbb R_+)$. Hence, the Cauchy--Schwarz
inequality gives
\begin{equation}\label{holder:r-rho}
|r_1(k)-r_1(0)|
+
|\rho_1(k)-\rho_1(0)|
\leq Ck^{1/2},
\qquad k\geq0.
\end{equation}
Furthermore, Proposition~\ref{deltaproposition} implies
\(
|\chi_1(k)-\chi_1(0)|
\leq C|k|^{1/2}.
\)
Combining this estimate with
Proposition~\ref{deltaproposition1}, we obtain, in the corresponding
sectors,
\begin{equation*}
\begin{aligned}
\left|
\frac{\delta_1^2(k)}
     {\delta_3(k)\delta_5(k)}
-e^{-2\pi\nu}
\right|
&\leq C|k|^{1/2},
\left|
\frac{\delta_1(k)\delta_3(k)}
     {\delta_5^2(k)}
-1
\right|
\leq C|k|^{1/2},
\end{aligned}
\end{equation*}
together with the analogous estimates obtained by the
$\mathbb Z_3$ symmetry.

We first consider
$k\in\Sigma_1^{(2)}\cap D_\epsilon(0)$. By the definition of
$R_1(k)$ on $\Sigma_1^{(2)}$,
\[
R_1(k)
=
\rho_1^*(k_0)
\frac{\delta_1^2(k)}
     {\delta_3(k)\delta_5(k)}.
\]
Only the $(2,1)$ and $(2,3)$ entries of
$v_1^{(2)}-v_1^p$ are nonzero. For the $(2,1)$ entry, we have
\[
\begin{aligned}
[v_1^{(2)}-v_1^p]_{21}
={}&
\left[
\rho_1^*(k_0)
\frac{\delta_1^2(k)}
     {\delta_3(k)\delta_5(k)}
-
e^{-2\pi\nu}\rho_1^*(0)
\right]
e^{\vartheta_{21}(x,t,k)}
\\
={}&
\frac{\delta_1^2(k)}
     {\delta_3(k)\delta_5(k)}
\bigl(\rho_1^*(k_0)-\rho_1^*(0)\bigr)
e^{\vartheta_{21}(x,t,k)}+
\rho_1^*(0)
\left[
\frac{\delta_1^2(k)}
     {\delta_3(k)\delta_5(k)}
-
e^{-2\pi\nu}
\right]
e^{\vartheta_{21}(x,t,k)}.
\end{aligned}
\]
It follows from \eqref{holder:r-rho} that
\[
|[v_1^{(2)}-v_1^p]_{21}|
\leq
C\bigl(k_0^{1/2}+|k|^{1/2}\bigr)
\left|e^{\vartheta_{21}(x,t,k)}\right|.
\]

Parameterize $\Sigma_1^{(2)}$ by
\(
k=k_0+ue^{i\theta_1},
\ u\geq0,
\)
where $\theta_1$ is the direction of $\Sigma_1^{(2)}$. 
We obtain
\[
|[v_1^{(2)}-v_1^p]_{21}|
\leq
C(k_0+u)^{1/2}e^{-ctu^2}.
\]
Similarly, for the $(2,3)$ entry, we have
\begin{equation}\label{pointwise:lens-error}
|v_1^{(2)}(x,t,k)-v_1^p(x,t,k)|
\leq
C(k_0+u)^{1/2}e^{-ctu^2}.
\end{equation}

Since $k_0=\mathcal O(t^{-1/2})$, for $1\leq p<\infty$, we have
\[
\begin{aligned}
\|v_1^{(2)}-v_1^p\|_{L^p}^p
&\leq
C\int_0^\infty
(k_0+u)^{p/2}e^{-cptu^2}\,\mathrm du
\leq
Ct^{-\frac{p}{4}-\frac12}.
\end{aligned}
\]
For $p=\infty$, \eqref{pointwise:lens-error} gives
\[
\|v_1^{(2)}-v_1^p\|_{L^\infty}
\leq
C\sup_{u\geq0}
(k_0+u)^{1/2}e^{-ctu^2}
=
\mathcal O(t^{-1/4}).
\]
The same estimates hold on
$\Sigma_j^{(2)}\cap D_\epsilon(0)$, $j=2,\ldots,6$, by the
$\mathbb Z_3$ symmetry.

It remains to establish the estimates on
$\partial D_\epsilon(0)$. Recall that
\[
m^p(y;\lambda)
=
I+\frac{m_1^p(y)}{\lambda}
+\frac{m_2^p(y)}{\lambda^2}
+\mathcal O(\lambda^{-3}),
\qquad \lambda\to\infty.
\]
Since
\(
\lambda=-\frac{2\sqrt{t}}{3\sqrt{3}}\,k,
\)
we obtain, uniformly for $k\in\partial D_\epsilon(0)$,
\[
m^p(y;\lambda(k))
=
I-\frac{3\sqrt{3}}{2\sqrt{t}\,k}m_1^p(y)
+\mathcal O(t^{-1}).
\]
It follows that
\[
(m^p(y;\lambda(k)))^{-1}
=
I+\frac{3\sqrt{3}}{2\sqrt{t}\,k}m_1^p(y)
+\mathcal O(t^{-1}),
\]
which proves \eqref{invmpestimate}. Finally, the residue theorem
yields
\[
\begin{aligned}
\frac{1}{2\pi i}
\int_{\partial D_\epsilon(0)}
\left((m^p)^{-1}-I\right)\,\mathrm dk
&=
\frac{1}{2\pi i}
\int_{\partial D_\epsilon(0)}
\frac{3\sqrt{3}}{2\sqrt{t}\,k}
m_1^p(y)\,\mathrm dk
+\mathcal O(t^{-1})
\\
&=
\frac{3\sqrt{3}}{2\sqrt{t}}m_1^p(y)
+\mathcal O(t^{-1}),
\end{aligned}
\]
which proves \eqref{integral:mp}.
\end{proof}

Let $\hat{\Sigma} := \Sigma^{(2)} \cup \partial D_\epsilon(0)$, as illustrated in Figure \ref{Sigmahat}. The function $\hat{m}(x, t,k)$ is defined by  
\[
\hat{m}(x, t,k) = 
\begin{cases} 
   S(x, t,k) \left(m^p\right)^{-1}(y;\lambda), & k \in D_\epsilon(0), \\ 
   S(x, t,k), & k \in \mathbb{C} \setminus D_\epsilon(0),
\end{cases}
\]
with jump matrices given by  
\[
\hat{v}(x, t,k) = 
\begin{cases} 
    m_{-}^p(y;\lambda(k)) v^{(2)} \left(m_{+}^p\right)^{-1}(y;\lambda(k)), & k \in D_\epsilon(0) \cap \Sigma^{(2)}, \\ 
    \left(m^p\right)^{-1}(y;\lambda(k)), & k \in \partial D_\epsilon(0), \\ 
    v^{(2)}(x,t,k), & k \in \hat{\Sigma} \setminus \bar{D}_\epsilon(0).
\end{cases}
\]

\begin{figure}[h]
    \centering
    \begin{tikzpicture}[
        >=latex,
        directed/.style={
            very thick,
            postaction={decorate},
            decoration={
                markings,
                mark=at position 0.5 with {\arrow{>}}
            }
        }
    ]
        \draw[directed]
            (0,0) -- (1.6,0)
            node[pos=0.72,below] {$\small 7$};

        \draw[directed]
            (0,0) -- (-0.8,1.386)
            node[pos=0.72,left] {$\small 8$};

        \draw[directed]
            (0,0) -- (-0.8,-1.386)
            node[pos=0.72,left] {$\small 9$};

        \draw[very thick,dashed] (0,0) -- (-1.6,0);
        \draw[very thick,dashed] (0,0) -- (0.8,1.386);
        \draw[very thick,dashed] (0,0) -- (0.8,-1.386);

        \draw[very thick,dashed] (1.6,0) -- (4,0);
        \draw[very thick,dashed] (-1.6,0) -- (-4,0);

        \draw[very thick,dashed]
            (0.8,1.386) -- (2,3.464);
        \draw[very thick,dashed]
            (-0.8,-1.386) -- (-2,-3.464);

        \draw[very thick,dashed]
            (-0.8,1.386) -- (-2,3.464);
        \draw[very thick,dashed]
            (0.8,-1.386) -- (2,-3.464);

        \filldraw[black]
            (1.6,0)
            node[below=0.5mm] {$k_0$}
            circle (1.5pt);

        \filldraw[black]
            (-0.8,1.386)
            node[right=1mm] {$\omega k_0$}
            circle (1.5pt);

        \filldraw[black]
            (-0.8,-1.386)
            node[right=1mm] {$\omega^2k_0$}
            circle (1.5pt);

        \draw[directed]
            (1.6,0) -- (3.678,1.2)
            node[pos=0.72,above] {$\small 1$};

        \draw[directed]
            (1.6,0) -- (3.678,-1.2)
            node[pos=0.72,below] {$\small 6$};

        \draw[directed]
            (-0.8,1.386) -- (-0.8,3.786)
            node[pos=0.72,right] {$\small 2$};

        \draw[directed]
            (-0.8,1.386) -- (-2.878,2.586)
            node[pos=0.72,above] {$\small 3$};

        \draw[directed]
            (-0.8,-1.386) -- (-2.878,-2.586)
            node[pos=0.72,above] {$\small 4$};

        \draw[directed]
            (-0.8,-1.386) -- (-0.8,-3.786)
            node[pos=0.72,right] {$\small 5$};

        \node at (3.15,0.55)   {$\Omega_1$};
        \node at (3.0,2.35)   {$\Omega_7$};
        \node at (0.5,3.30)      {$\Omega_8$};
        \node at (-1.20,2.70)  {$\Omega_2$};
        \node at (-1.70,2.40)  {$\Omega_3$};
        \node at (-2.70,0.80)  {$\Omega_9$};

        \node at (-2.70,-0.80) {$\Omega_{10}$};
        \node at (-1.70,-2.40) {$\Omega_4$};
        \node at (-1.20,-2.70) {$\Omega_5$};
        \node at (1,-3.30)     {$\Omega_{11}$};
        \node at (2.60,-2.65)  {$\Omega_{12}$};
        \node at (3.15,-0.55)  {$\Omega_{6}$};

        \node[right] at (4,0) {$\mathbb{R}$};
        \node[above] at (-2,3.464) {$\omega\mathbb{R}$};
        \node[below] at (-2,-3.464) {$\omega^2\mathbb{R}$};

        \node[below right] at (0,0) {$0$};
\draw[
    very thick,
    postaction={decorate},
    decoration={
        markings,
        mark=at position 0.18 with {\arrow{>}}
    }
]
(2.2,0) arc[start angle=0,end angle=360,radius=2.2];

\node[above right] at (0.15,2.2)
    {$\partial D_{\epsilon}(0)$};
    \end{tikzpicture}
    \caption{\protect\small
        The jump contour $\hat{\Sigma}$ and the saddle points
        $\omega^j k_0$ for $j=0,1,2$.}
    \label{Sigmahat}
\end{figure}
	\begin{lemma}\label{w estimate}
    Let $(x, t)$ satisfy $\frac{x}{\sqrt{t}} \leq C$ with $x > 0$ as $t\to\infty$, and define $\hat{w} = \hat{v} - I$. Then, for $1 \leq p \leq \infty$, the following estimates hold:
    \begin{equation*}
        \|\hat{w}\|_{L^p\left(\partial D_\epsilon(0)\right)} \leq C t^{-\frac{1}{2}}, \quad 
        \|\hat{w}\|_{L^p\left(\hat{\Sigma} \setminus \bar{D}_\epsilon(0)\right)} \leq C t^{-N}, \quad 
        \|\hat{w}\|_{L^p\left(D_\epsilon(0) \cap \Sigma^{(2)}\right)} \leq C t^{-\frac{1}{4} - \frac{1}{2p}},
    \end{equation*}
    for any positive integer $N$.
\end{lemma}

	\begin{proof}
    The function $\hat{w}$ can be expressed as
    $$
    \hat{w}(x, t,k) =
    \begin{cases}
        \begin{aligned}
            &m_{-}^p(y;\lambda(k)) (v^{(2)} - v^p) \left(m_{+}^p\right)^{-1}(y;\lambda(k)), && k \in D_\epsilon(0) \cap \Sigma^{(2)}, \\ 
            &\left(m^p\right)^{-1}(y;\lambda(k)) - I, && k \in \partial D_\epsilon(0), \\ 
            &v^{(2)}(x, t,k) - I, && k \in \hat{\Sigma} \setminus \bar{D}_\epsilon(0).
        \end{aligned}
    \end{cases}
    $$
    
    Since $k_0=\mathcal O(t^{-1/2})$, all the stationary points lie
inside $D_\epsilon(0)$ for sufficiently large $t$. Moreover, the
signature properties of the phase functions imply that
\(
\Re\vartheta_{ij}(x,t,k)\leq-c_\epsilon t
\)
on the relevant components of
$\Sigma^{(2)}\setminus\overline{D_\epsilon(0)}$. Since the remaining
coefficients in the jump matrices are uniformly bounded, it follows
that
\[
\left\|\widehat{w}\right\|_
{L^p(\Sigma^{(2)}\setminus\overline{D_\epsilon(0)})}
=
\mathcal O(e^{-c_\epsilon t}),
\qquad 1\leq p\leq\infty.
\]
Moreover, by Lemma \ref{v2-vp} and the interpolation theorem, one has
    $$
    \|\hat{w}\|_{L^p\left(\partial D_\epsilon(0)\right)} \leq 
    \|\hat{w}\|_{L^1\left(\partial D_\epsilon(0)\right)}^{\frac{1}{p}}
    \|\hat{w}\|_{L^{\infty}\left(\partial D_\epsilon(0)\right)}^{\frac{p-1}{p}} 
    \leq C t^{-\frac{1}{2}}.
    $$  
    Similarly, according to Lemma \ref{v2-vp}, it follows that
    $$
    \|\hat{w}\|_{L^p\left(D_\epsilon(0) \cap \Sigma^{(2)}\right)} 
    \leq C t^{-\frac{1}{4} - \frac{1}{2p}}.
    $$
\end{proof}
	
	For a function $h$ defined on $\hat{\Sigma}$, introduce the Cauchy operator $\mathcal{C}$ as 
$$
(\mathcal{C}h)(k) := \frac{1}{2\pi i}\int_{\hat{\Sigma}} \frac{h(k')}{k' - k} \, \mathrm{d}k', \quad k \in \mathbb{C} \setminus \hat{\Sigma},
$$
and denote the left and right non-tangential boundary values of $(\mathcal{C}h)(k)$ as $(\mathcal{C}_+h)(k)$ and $(\mathcal{C}_-h)(k)$, respectively. Define the operator 
$$
\mathcal{C}_{\hat{w}}h = \mathcal{C}_-\left(h \hat{w}\right),
$$
where $\hat{w}$ is the jump matrix defined on $\hat{\Sigma}$. By Lemma \ref{w estimate}, the operator $\mathcal{C}_{\hat{w}}$ is well-posed from ${L}^2(\hat{\Sigma}) + L^{\infty}(\hat{\Sigma})$ to ${L}^2(\hat{\Sigma})$. Furthermore, using the estimates in Lemma \ref{w estimate}, we obtain that
$$
\|\mathcal{C}_{\hat{w}}\|_{L^2(\hat{\Sigma}) \to L^2(\hat{\Sigma})} \leq C \|\hat{w}\|_{L^{\infty}(\hat{\Sigma})} \leq C t^{-1/4}.
$$

Consequently, for sufficiently large $t$, the operator $I - \mathcal{C}_{\hat{w}}$ is invertible as a map from $L^{2}(\hat{\Sigma})$ to $L^{2}(\hat{\Sigma})$. Therefore, by the standard small-norm theory, the solution $\hat{m}(x, t,k)$ exists and is unique, which is given by 
$$
\hat{m}(x, t,k) = I + \mathcal{C}(\hat{\mu} \hat{w}) 
= I + \int_{\hat{\Sigma}} \frac{\hat{\mu}(\zeta;x, t) \hat{w}(\zeta;x, t)}{\zeta - k} \frac{\mathrm{d} \zeta}{2 \pi i}, \quad k \in \mathbb{C} \setminus \hat{\Sigma},
$$
where 
$$
\hat{\mu} = I + \left(I - \mathcal{C}_{\hat{w}}\right)^{-1} \mathcal{C}_{\hat{w}}I.
$$

Moreover, for $1 < p < \infty$, the following estimate holds:
$$
\|\hat{\mu} - I\|_{L^{p}(\hat{\Sigma})} 
\leq \sum_{j=1}^\infty \|\mathcal{C}_{\hat{w}}\|_{L^{p}(\hat{\Sigma}) \to L^p(\hat{\Sigma})}^{j-1} \|\mathcal{C}_{\hat{w}}I\|_{L^p(\hat{\Sigma})} 
\leq C_p\|\hat{w}\|_{L^p(\hat{\Sigma})} 
\leq C_p (t^{-\frac{1}{4} - \frac{1}{2p}}+t^{-\frac{1}{2}}).
$$

As a result, as $t \to \infty$, the solution $\hat{m}(x, t,k)$ satisfies the following asymptotics.

	\begin{lemma}\label{mestimate}
    For \( 0\leq\frac{x}{\sqrt{t}} \leq C \) with \( x \geq 0 \), and as \( t \to \infty \), the following non-tangential limit exists:
    \begin{equation*}
        \lim_{k \to \infty} k(\hat{m}(x, t,k) - I) = -\frac{1}{2 \pi i} \int_{\partial D_\epsilon(0)} \hat{w}(\zeta;x, t) \, \mathrm{d}\zeta + \mathcal{O}(t^{-\frac{3}{4}}).
    \end{equation*}
\end{lemma}

\begin{proof}
    Decompose the term \( -\int_{\hat{\Sigma}} {\hat{\mu}(x, t,k) \hat{w}(x, t,k)} \frac{\mathrm{d} k}{2 \pi i} \) into:
    \[
    -\frac{1}{2 \pi i} \int_{\partial D_\epsilon(0)} \hat{w}(\zeta;x, t) \, \mathrm{d}\zeta + Q_1(x, t) + Q_2(x, t),
    \]
    where
    \[
    Q_1(x, t) := -\frac{1}{2 \pi i} \int_{\hat{\Sigma}} (\hat{\mu}(\zeta;x, t) - I) \hat{w}(\zeta;x, t) \, \mathrm{d}\zeta,
    \]
    and
    \[
    Q_2(x, t) := -\frac{1}{2 \pi i} \int_{\hat{\Sigma} \backslash \partial {D}_\epsilon(0)} \hat{w}(\zeta;x, t) \, \mathrm{d}\zeta.
    \]

    For \( Q_1(x, t) \), Hölder's inequality indicates that
    \[
    |Q_1(x, t)| \leq C \|\hat{\mu}(x, t; \cdot) - I\|_{L^2(\hat{\Sigma})} \|\hat{w}(x, t; \cdot)\|_{L^2(\hat{\Sigma})} \leq t^{-1}.
    \]

    For \( Q_2(x, t) \), it is derived that
    \[
    |Q_2(x, t)| \leq C \|\hat{w}(\cdot;x, t)\|_{L^1\left(\hat{\Sigma} \backslash \partial D_\epsilon(0)\right)} \leq t^{-\frac{3}{4}}.
    \]
\end{proof}

In conclusion, for the pure RH problem $S(x,t,k)$, we can construct the parametrix by
\begin{equation}\label{def:S}
    S(x,t,k)=\begin{cases}
        \begin{aligned}
            &\hat{m}(x,t,k),&&k\in\CC\setminus D_{\epsilon}(0),\\
            &\hat{m}(x,t,k)m^p(y;\lambda(k)),&&k\in D_{\epsilon}(0).\\
        \end{aligned}
    \end{cases}
\end{equation}

\subsection{The Pure $\dbar$ Problem}

To isolate the pure $\dbar$ contribution, define
\begin{equation*}
m^{(3)}(x,t,k)
:=
m^{(2)}(x,t,k)S^{-1}(x,t,k).
\end{equation*}
Since $m^{(2)}(x,t,k)$ and $S(x,t,k)$ satisfy the same jump conditions on
$\Sigma^{(2)}$, the function $m^{(3)}(x,t,k)$ has no jumps across
$\Sigma^{(2)}$. Moreover, define
\begin{equation}\label{def:W}
\mathcal W(x,t,k)
:=
S(x,t,k)\mathcal R^{(2)}(x,t,k)S^{-1}(x,t,k).
\end{equation}
It follows that $m^{(3)}(x,t,k)$ satisfies the following pure $\dbar$
problem.

\begin{dbarproblem}\label{dbarproblem:m3}
Find a continuous $3\times3$ matrix-valued function
$m^{(3)}(x,t,k)$ satisfying the following properties:
\begin{enumerate}
    \item[(i)]
    For almost every $k\in\mathbb C$,
   $
        \dbar m^{(3)}(x,t,k)
        =
        m^{(3)}(x,t,k)\mathcal W(x,t,k),
    $
    where $\mathcal W(x,t,k)$ is defined in \eqref{def:W}.

    \item[(ii)]
    As $k\to\infty$,
    $
        m^{(3)}(x,t,k)
        =
        I+\mathcal O(k^{-1}).
    $
\end{enumerate}
\end{dbarproblem}

By the Cauchy--Green formula (see, e.g., \cite{Chen-Liu-2022}), the pure
$\dbar$ problem~\ref{dbarproblem:m3} is equivalent to the following
integral equation:
\begin{equation}\label{solution:m3}
m^{(3)}(x,t,k)
=
I-\frac{1}{\pi}
\iint_{\mathbb C}
\frac{m^{(3)}(\zeta;x,t)\mathcal W(\zeta;x,t)}
     {\zeta-k}\,
\mathrm dA(\zeta),
\end{equation}
where
\(
\mathrm dA(\zeta)
=
\mathrm d(\Re\zeta)\,\mathrm d(\Im\zeta)
\)
denotes the Lebesgue area measure on $\mathbb C$.

\begin{lemma}\label{lem:exists}
    Let $C_0>0$ and suppose that
\(
0\leq x/\sqrt{t}\leq C_0.
\)
Assume that the reflection coefficients satisfy the hypotheses of
Lemmas~\ref{lem:dbar-extension-r2} and
\ref{lem:dbar-extension-Rj}.  Then as $t\to\infty$, there exists a unique solution $m^{(3)}(x,t,k)$ of the $\dbar$-Problem~\ref{dbarproblem:m3}.
\end{lemma}
\begin{proof}
    First, we will show that as $t\to+\infty$, the solution of the $\dbar$-problem exists. Introduce the following integral operator $K_{\mathcal{W}}$ by
    $$
    (K_{\mathcal{W}}f)(k)=\frac{1}{\pi}\iint_{\CC}\frac{f(\zeta)\mathcal{W}(\zeta)}{k-\zeta}dA(\zeta).
    $$
    By the definition of $\mathcal{W}(x,t,k)$ in \eqref{def:W} and the boundedness of $m^p(y;\lambda)$ and  $\left(m^p(y;\lambda)\right)^{-1}$ in Appendix~\ref{Appendix}, it follows that
    $$
    |\mathcal{W}(x,t,k)|\le C|\mathcal{R}^{(2)}(x,t,k)|.
    $$
    In what follows, we will focus on the estimate for $|\mathcal{R}^{(2)}(x,t,k)|$ in $\Omega_1$ and $\Omega_7$ in Figure~\ref{Sigma2}, the estimate for the other regions can be obtained by the similar method combining with symmetry conditions.
    \begin{itemize}
        \item For $k\in\Omega_1$, let's parameterize $k:=u+iv$ with $u\ge k_0,\ 0\le v\le \frac{u-k_0}{\sqrt{3}}$, then we have
        $$
        \Re \vartheta_{21}(x,t,k)=-2\sqrt{3}t(u-k_0)v,\ \Re\vartheta_{32}(x,t,k)
\geq
2\sqrt{3}\,tuv
+\frac{9-4\sqrt{3}}{8}\,tu^2
\geq
2\sqrt{3}\,tuv.
        $$
    Moreover, by the definition of $\mathcal{R}_1^{(2)}(x,t,k)$ in RH $\dbar$-Problem \ref{rhp-4.5}, we have 
    $$
    \begin{aligned}
     |(K_{\mathcal{W}}f)(k)|&\leq\frac{1}{\pi}\iint_{\Omega_1}\frac{|f(\zeta)||\mathcal{W}(\zeta)|}{|\zeta-k|}dA(\zeta)\leq C \|f\|_{L^{\infty}(\Omega_1)} \iint_{\Omega_1}\frac{|\mathcal{W}(\zeta)|}{|\zeta-k|}dA(\zeta)\\
     &\leq C  \|f\|_{L^{\infty}(\Omega_1)}\left( \iint_{\Omega_1}\frac{|\mathcal{\bar{\partial}}R_1(\zeta)e^{\vartheta_{21}(\zeta)}|}{|\zeta-k|}dA(\zeta)+\iint_{\Omega_1}\frac{|\mathcal{\bar{\partial}}G_1(\zeta)e^{-\vartheta_{32}(\zeta)}|}{|\zeta-k|}dA(\zeta)\right)\\
     &:= C  \|f\|_{L^{\infty}(\Omega_1)}\left(I_1+I_2\right).
    \end{aligned}
    $$
    By Lemma \ref{lem:dbar-extension-Rj}, we obtain that
   \begin{equation*}
\scalebox{0.925}{$
\begin{aligned}
I_1
&\leq
\|r_1'\|_{L^2(\mathbb R_+)}
\iint_{\Omega_1}
\frac{|\zeta-k_0|^{-1/2}}
     {|\zeta-k|}
e^{-2\sqrt3\,t(u-k_0)v}\,
\mathrm dA(\zeta)
+
\iint_{\Omega_1}
\frac{|r_1'(u)|}
     {|\zeta-k|}
e^{-2\sqrt3\,t(u-k_0)v}\,
\mathrm dA(\zeta)
\\
&\leq
C\left(t^{-1/4}\|r_1'\|_{L^2(\mathbb R_+)}
+
\int_0^\infty
e^{-4\sqrt3\,tv^2}
\left(
\int_{k_0+2v}^\infty
|r_1'(u)|^2\,\mathrm du
\right)^{1/2}\right.\\
&\left.\hspace{42mm}\times
\left(
\int_{k_0+2v}^\infty
\frac{\mathrm du}
{(u-\xi)^2+(v-\eta)^2}
\right)^{1/2}
\mathrm dv\right)
\\
&\leq
C\left(
t^{-1/4}\|r_1'\|_{L^2(\mathbb R_+)}
+
\|r_1'\|_{L^2(\mathbb R_+)}
\int_0^\infty
e^{-4\sqrt3\,tv^2}
|v-\eta|^{-1/2}\,\mathrm dv
\right)\leq
Ct^{-1/4}\|r_1'\|_{L^2(\mathbb R_+)}.
\end{aligned}
$}
\end{equation*}
Similarly, by Lemma \ref{lem:dbar-extension-r2}, we have
\begin{equation*}
\scalebox{0.95}{$
\begin{aligned}
I_2
&\leq
\|r_2'\|_{L^2(\mathbb R_-)}
\iint_{\Omega_1}
\frac{|\zeta|^{-1/2}}
     {|\zeta-k|}
e^{-2\sqrt3\,tuv}\,
\mathrm dA(\zeta)
+
\iint_{\Omega_1}
\frac{
\left|r_2'\left(-\frac{u+\sqrt3v}{2}\right)\right|
}
{|\zeta-k|}
e^{-2\sqrt3\,tuv}\,
\mathrm dA(\zeta)
\\
&\leq
Ct^{-1/4}\|r_2'\|_{L^2(\mathbb R_-)}
+
C\int_0^\infty
e^{-4\sqrt3\,tv^2}
\left(
\int_{k_0+2v}^\infty
\left|
r_2'\left(-\frac{u+\sqrt3v}{2}\right)
\right|^2
\mathrm du
\right)^{1/2}
\\
&\hspace{42mm}\times
\left(
\int_{k_0+2v}^\infty
\frac{\mathrm du}
{(u-\xi)^2+(v-\eta)^2}
\right)^{1/2}
\mathrm dv
\\
&\leq
C\left(
t^{-1/4}\|r_2'\|_{L^2(\mathbb R_-)}
+
\|r_2'\|_{L^2(\mathbb R_-)}
\int_0^\infty
e^{-4\sqrt3\,tv^2}
|v-\eta|^{-1/2}\,\mathrm dv
\right)
\\
&\leq
Ct^{-1/4}\|r_2'\|_{L^2(\mathbb R_-)}.
\end{aligned}
$}
\end{equation*}
 \item For $k\in\Omega_7$, set
 \begin{equation}\label{def:Omega7}
     \begin{cases}
 \begin{aligned}
     &\Omega_7^{(1)}
:=
\left\{\zeta=u+iv:
0\leq u\leq k_0,\ 
0\leq v\leq\frac{\sqrt3u}{2}\right\},\\
&\Omega_7^{(2)}
:=
\left\{\zeta=u+iv:
u\geq k_0,\ 
\frac{u-k_0}{2}\leq v\leq\frac{\sqrt3u}{2}\right\},\\     
 \end{aligned}
 \end{cases} 
 \end{equation}
 then we have
 $$
\Re\bigl(-\vartheta_{32}(x,t,k)\bigr)
\leq
\begin{cases}
-\dfrac{3}{2}tk_0u,
& k\in\Omega_7^{(1)},\\[2mm]
-\dfrac{3}{4}tu^2,
&k\in\Omega_7^{(2)}.
\end{cases}
$$

Consequently, by the  $\dbar$-condition in RH $\dbar$-Problem \ref{rhp-4.5} and Lemma \ref{lem:dbar-extension-r2}, we have
\[
\begin{aligned}
 |(K_{\mathcal{W}}f)(k)|&\leq  C  \|f\|_{L^{\infty}(\Omega_7)} \iint_{\Omega_7}\frac{|\mathcal{\bar{\partial}}G_1(\zeta)e^{-\vartheta_{32}(\zeta)}|}{|\zeta-k|}dA(\zeta)
 \\
&\leq C  \|f\|_{L^{\infty}(\Omega_7)} \|r_2'\|_{L^2(\mathbb R_-)}
\iint_{\Omega_7^{(1)}\cup\Omega_7^{(2)}}
\frac{|\zeta|^{-1/2}}
     {|\zeta-k|}
e^{-\Re\vartheta_{32}(\zeta)}\,dA(\zeta)\\
&+
C \|f\|_{L^{\infty}(\Omega_7)} \iint_{\Omega_7^{(1)}\cup \Omega_7^{(2)}}
\frac{
\left|r_2'\left(-\frac{u+\sqrt3v}{2}\right)\right|
}
{|\zeta-k|}
e^{-\Re\vartheta_{32}(\zeta)}\,dA(\zeta).
\\
\end{aligned}
\]
The first integral is estimated by scaling
$U=\sqrt t\,u$ and $V=\sqrt t\,v$. Since
$\sqrt t\,k_0$ is uniformly bounded in the self-similar region, they
are bounded by
\(
Ct^{-1/4}\|r_2'\|_{L^2(\mathbb R_-)}.
\)
For the remaining integral, the Cauchy--Schwarz inequality in
the $u$-variable yields
\[
\begin{aligned}
\iint_{\Omega_7^{(1)}\cup \Omega_7^{(2)}}
\frac{
\left|r_2'\left(-\frac{u+\sqrt3v}{2}\right)\right|
}
{|\zeta-k|}
e^{-\Re\vartheta_{32}(\zeta)}\,dA(\zeta)
\leq{}
&
C\|r_2'\|_{L^2(\mathbb R_-)}
\int_0^{\frac{\sqrt3}{2}k_0}
e^{-\sqrt3tk_0v}|v-\eta|^{-1/2}\,dv
\\
&+
C\|r_2'\|_{L^2(\mathbb R_-)}
\int_0^\infty
e^{-tv^2}|v-\eta|^{-1/2}\,dv,
\end{aligned}
\]
where $k=\xi+i\eta$. Since
\(
\sup_{\eta\in\mathbb R}
\int_0^{\frac{\sqrt3}{2}k_0}
e^{-\sqrt3tk_0v}|v-\eta|^{-1/2}\,dv
\leq Ct^{-1/4}
\)
and
\(
\sup_{\eta\in\mathbb R}
\int_0^\infty
e^{-tv^2}|v-\eta|^{-1/2}\,dv
\leq Ct^{-1/4},
\)
we conclude that
\begin{equation*}
 |(K_{\mathcal{W}}f)(k)|
\leq
Ct^{-1/4}\|r_2'\|_{L^2(\mathbb R_-)} \|f\|_{L^{\infty}(\Omega_7)}.
\end{equation*}
    \end{itemize}
Combining the above estimates with the symmetry relations, we conclude
that, as $t\to\infty$,
\begin{equation*}
\|K_{\mathcal W}\|_
{L^\infty(\mathbb C)\to L^\infty(\mathbb C)}
\leq Ct^{-1/4},
\end{equation*}
where $C>0$ is independent of $t$ and depends on
$\|r_1'\|_{L^2(\mathbb R_+)}$ and
$\|r_2'\|_{L^2(\mathbb R_-)}$. Consequently, for all sufficiently large $t$, the operator
$I-K_{\mathcal W}$ is invertible on $L^\infty(\mathbb C)$ by the
convergence of the Neumann series. Hence, the integral
equation~\eqref{solution:m3}, and therefore the pure $\dbar$-Problem \ref{dbarproblem:m3},
admits a unique solution $m^{(3)}(x,t,k)$.
\end{proof}

\begin{lemma}\label{lem:m3 expansion}
 Let $C_0>0$ and suppose that
\(
0\leq x/\sqrt{t}\leq C_0.
\)
Assume that the reflection coefficients satisfy the hypotheses of
Lemmas~\ref{lem:dbar-extension-r2} and
\ref{lem:dbar-extension-Rj}. 
    Then as $t$ large enough and as $k\to\infty$, we have the following expansion:
\begin{equation*}
    m^{(3)}(x,t,k)=I+\frac{1}{k}m_1^{(3)}(x,t)+{o}\left(\frac{1}{k}\right),
\end{equation*}
where
$
|m_1^{(3)}(x,t)|\le C t^{-\frac{3}{4}}.
$
\end{lemma}
\begin{proof}
Before deriving the asymptotic expansion of $m^{(3)}(x,t,k)$ as
$k\to\infty$, we first estimate $\|\mathcal W(x,t,k)\|_{L^1}$ as
$t\to\infty$. It suffices to consider the regions $\Omega_1$ and
$\Omega_7$, since the estimates in the remaining regions follow
analogously by symmetry.
\begin{itemize}
    \item For $k\in\Omega_1$, denote $k:=u+iv$ with $u\ge k_0,\ 0\le v\le \frac{u-k_0}{\sqrt{3}}$, then we have
    $$
    \iint_{\Omega_1}|\mathcal{W}(\zeta;x,t)|d\zeta\leq \iint_{\Omega_1}|\dbar\mathcal{R}_1(\zeta;x,t)e^{\vartheta_{21}(\zeta)}|dA(\zeta)+\iint_{\Omega_1}|\dbar G_1(\zeta;x,t)e^{-\vartheta_{32}(\zeta)}|dA(\zeta).
    $$
    By Lemma \ref{lem:dbar-extension-Rj}, and with $u'=u-k_0$, we have
    $$
    \begin{aligned}
\iint_{\Omega_1}|\dbar\mathcal{R}_1(\zeta;x,t)e^{\vartheta_{21}(\zeta)}|dA(\zeta) &\leq
C\|r_1'\|_{L^2(\mathbb R_+)}
\int_0^\infty\int_0^{u'/2}
|u'+iv|^{-1/2}
e^{-2\sqrt3\,t u'v}\,
\mathrm dv\,\mathrm du'
\\
&+
C\int_0^\infty
\bigl|r_1'(u'+k_0)\bigr|
\int_0^{u'/2}
e^{-2\sqrt3\,t u'v}\,
\mathrm dv\,\mathrm du'\\
&\leq C t^{-\frac{3}{4}} \|r'_1\|_{L^2(\R_+)}.
    \end{aligned}
    $$
    Similarly, by Lemma \ref{lem:dbar-extension-r2}, it follows
    $$
    \begin{aligned}
        \iint_{\Omega_1}|\dbar G_1(\zeta;x,t)e^{-\vartheta_{32}(\zeta)}|dA(\zeta)&\leq
C\|r_2'\|_{L^2(\mathbb R_-)}
\int_{k_0}^{\infty}
u^{-1/2}e^{-c_0tu^2}
\int_0^{(u-k_0)/2}\mathrm dv\,\mathrm du
\\
&+
C\int_{k_0}^{\infty}
e^{-c_0tu^2}
\int_0^{(u-k_0)/2}
\left|
r_2'\left(-\frac{u+\sqrt3v}{2}\right)
\right|
\mathrm dv\,\mathrm du,
\\[2mm]
&\leq
C\int_0^{\infty}
u^{1/2}e^{-c_0tu^2}\,\mathrm du
+
\int_{0}^{\infty}u^{1/2}e^{-c_0tu^2}\|r_2'\|_{L^2(\mathbb R_-)}
\mathrm dv\\
&\leq
C t^{-\frac{3}{4}}\|r_2'\|_{L^2(\mathbb R_-)}.
    \end{aligned}
    $$
\item For \(k\in\Omega_7\), we decompose \(\Omega_7\) into the two
subregions \(\Omega_7^{(1)}\) and \(\Omega_7^{(2)}\) defined in
\eqref{def:Omega7}. We then obtain
$$
\begin{aligned}
     \iint_{\Omega_7}|\mathcal{W}(\zeta;x,t)|d\zeta&\leq \iint_{\Omega_7}{|\mathcal{\bar{\partial}}G_1(\zeta)e^{-\vartheta_{32}(\zeta)}|}dA(\zeta)\\
     &\leq \iint_{\Omega_7^{(1)}}{|\mathcal{\bar{\partial}}G_1(\zeta)e^{-\vartheta_{32}(\zeta)}|}dA(\zeta)+\iint_{\Omega_7^{(2)}}{|\mathcal{\bar{\partial}}G_1(\zeta)e^{-\vartheta_{32}(\zeta)}|}dA(\zeta)\\
     &\leq C k_0^{\frac{3}{2}}\|r_2'\|_{L^2(\mathbb R_-)}+ C t^{-\frac{3}{4}}\|r_2'\|_{L^2(\mathbb R_-)}\leq C t^{-\frac{3}{4}}\|r_2'\|_{L^2(\mathbb R_-)}.
\end{aligned}
$$
\end{itemize}

    Now, rewrite the equation \eqref{solution:m3} as
    $$\
    m^{(3)}(x,t,k)
=
I+\frac{1}{k \pi}
\iint_{\mathbb C}
{m^{(3)}(\zeta;x,t)\mathcal W(\zeta;x,t)}
\mathrm dA(\zeta)+\frac{1}{k \pi}\iint_{\mathbb C}
\frac{\zeta }{k-\zeta} m^{(3)}(\zeta;x,t)\mathcal W(\zeta;x,t)
\mathrm dA(\zeta).
    $$
    Then by Lemma \ref{lem:exists}, it follows that when $t$ is sufficiently large, $m^{(3)}(\zeta;x,t)$ is bounded. As $k\to\infty$, by $\mathcal{W}(\zeta;x,t)\in L^1(\CC)$, it follows that 
    $$
    \lim_{k\to\infty}\iint_{\mathbb C}
\frac{\zeta }{k-\zeta} m^{(3)}(\zeta;x,t)\mathcal W(\zeta;x,t)
\mathrm dA(\zeta)=0
    $$
    and
    $$
   |m_1^{(3)}(x,t)|:=\left|\frac{1}{ \pi}
\iint_{\mathbb C}
{m^{(3)}(\zeta;x,t)\mathcal W(\zeta;x,t)}
\mathrm dA(\zeta)\right|\leq C\iint_{\CC}|\mathcal{W}(\zeta;x,t)|d\zeta\le C t^{-\frac{3}{4}}.
    $$
\end{proof}

\subsection{Long-Time Asymptotics}

We are now in a position to derive the long-time asymptotics of the solution to the modified Boussinesq equation. Tracing back the transformations introduced above, we obtain
\begin{equation*}
m(x,t,k)
=
m^{(3)}(x,t,k)S(x,t,k)
\mathcal R^{-1}(x,t,k)\Delta^{-1}(k)
\mathcal G^{-1}(x,t,k),
\end{equation*}
where \(m^{(3)}(x,t,k)\) solves the pure \(\dbar\)-problem~\ref{dbarproblem:m3}, \(S(x,t,k)\) solves RH problem~\ref{rhp:S}, and the functions \(\mathcal R(x,t,k)\), \(\Delta(k)\), and \(\mathcal G(x,t,k)\) are defined in \eqref{def:mathcal-R1}, \eqref{Deltadefination}, and \eqref{def:mathcal-Gj}, respectively.

By the definition of \(S(x,t,k)\) in \eqref{def:S} and Lemma~\ref{mestimate}, we have
\begin{equation}\label{expansion:S-infinity}
\begin{aligned}
\lim_{k\to\infty}k\bigl(S(x,t,k)-I\bigr)
&=
-\frac{1}{2\pi i}
\int_{\partial D_\epsilon(0)}
\hat w(\zeta;x,t),\mathrm d\zeta
+\mathcal O(t^{-3/4}).
\end{aligned}
\end{equation}

Combining \eqref{expansion:S-infinity} with the large-\(k\) expansion of \(m^{(3)}(x,t,k)\) established in Lemma~\ref{lem:m3 expansion}, and then using the reconstruction formulas in \eqref{recovermb}, we obtain, as \(t\to\infty\),
\begin{equation}\label{pq Painleve}
\begin{aligned}
p(x,t)
&=
-\frac{3}{2}
\left[
\frac{1}{2\pi i}
\int_{\partial D_\epsilon(0)}
\left(
[\hat w(x,t,k)]_{31}
+
[\hat w(x,t,k)]_{32}
\right)\mathrm dk
\right]
+\mathcal O(t^{-3/4})
\\
&=
-\frac{9\sqrt3}{4\sqrt t}
\left(
[m_1^p(y)]_{31}
+
[m_1^p(y)]_{32}
\right)
+\mathcal O(t^{-3/4})
\\
&=
\frac{\sqrt3}{4\sqrt t}
\frac{P_{\rm IV}'(y)+\frac23}{P_{\rm IV}(y)}
+\mathcal O(t^{-3/4}),
\\[2mm]
q(x,t)
&=
-\frac{9\sqrt3}{4\sqrt t}
\frac{
[m_1^p(y)]_{31}-[m_1^p(y)]_{32}
}{
2\omega(\omega-1)
}
+\mathcal O(t^{-3/4})
\\
&=
\frac{\sqrt3}{4\sqrt t}
\left(
P_{\rm IV}(y)+\frac23y
\right)
+\mathcal O(t^{-3/4}),
\end{aligned}
\end{equation}
where
$
y=-\frac{\sqrt3\,x}{2\sqrt t},
$
and \(P_{\rm IV}(y)\) denotes the relevant solution of the Painlevé IV equation~\eqref{PIV} with parameters
$
\alpha=-\frac16,\
\beta=-\frac23.
$
All the above estimates hold uniformly in the self-similar region under consideration.
\begin{remark}
    Indeed, substituting the leading-order term in (\ref{pq Painleve}) into the modified Boussinesq equation (\ref{mbequation}) with the self-similar transformation \( y = -\frac{\sqrt{3}x}{2\sqrt{t}} \) and integrating once with respect to the variable \( y \), one can also derive the Painlevé IV equation in (\ref{PIV}) with parameters  $\alpha = -\frac{1}{6}$ and $\beta = -\frac{2}{3}$.
\end{remark}
\begin{remark}
    For the case  of \( \frac{|x|}{\sqrt{t}} < c_1 \) with \( x \leq 0 \) and $c_1>0$, the analysis proceeds in a similar manner. Specifically, let \( \zeta := -\frac{x}{t} \) with \( x \geq 0 \), and rewrite the phase functions
$
\vartheta_{ij}(x, t; k) = t \left[ \left( \omega^j - \omega^i \right) k \zeta + \left( \omega^{2i} - \omega^{2j} \right) k^2 \right] := t \Phi_{ij}(\zeta, k),
$
then the analysis follows the same approach as described above.
\end{remark}

\section{The Painlevé transition region}\label{Transition1}
This section investigates the Painlevé IV transition region characterized by \(c_1t^{1/2}< x< c_2t^{3/5}\). Recall that \(k_0=x/(2t)\) and set \(\tau=tk_0^2=x^2/(4t)\). In this region, \(k_0\to0\) and \(\tau\to\infty\) as \(t\to\infty\). More precisely, we assume that \(\tau=\mathcal{O}(t^l)\) with \(0<l<1/5\), which implies that \(k_0=\mathcal{O}(t^{(l-1)/2})\) and \(x=\mathcal{O}(t^{(1+l)/2})=o(t^{3/5})\). In this region, the leading-order contribution is described by the Painlevé IV model. The transformations of the RH problem are the same as those performed in Section~\ref{sed:P4} and lead to the $\dbar$-RH problem~\ref{rhp-4.5}. To construct the local parametrix for \(S(x,t,k)\), we introduce the rescaled spectral variable \(z=k/k_0\). Recalling that \(\tau=tk_0^2\), the phase functions \(\vartheta_{ij}(x,t,k)\), \(1\leq j<i\leq3\), can be rewritten as
\begin{equation*}
    \vartheta_{ij}(\tau;z)
    =\tau\left(
    2(\omega^i-\omega^j)z
    +(\omega^{2i}-\omega^{2j})z^2
    \right).
\end{equation*}
Then the estimates in Lemma~\ref{v2-vp} are modified as follows.

\begin{lemma}
Let \(C_0>0\) and \(0<l<1/5\). Suppose that
\(
1\ll\tau=tk_0^2\leq C_0t^l,
\
k_0=\frac{x}{2t},
\)
and define
\[
D_\tau(0):=
\left\{
k\in\mathbb C:\ |k|<\frac{1}{\sqrt{\tau}}
\right\}.
\]
Assume that the reflection coefficients satisfy the hypotheses of
Lemmas~\ref{lem:dbar-extension-r2} and
\ref{lem:dbar-extension-Rj}. Then, as \(t\to\infty\), the jump matrices
\(v^{(2)}\) of RH problem~\ref{rhp:S} and the jump matrices \(v^p\) of
the Painlevé IV model RH problem satisfy
\begin{equation*}
\begin{aligned}
\left\|v^{(2)}-v^p\right\|_
{L^p(\Sigma^{(2)}\cap D_\tau(0))}
=
\mathcal{O}\left(
t^{-\frac14-\frac{1}{2p}}
+k_0^{\frac12}t^{-\frac{1}{2p}}
+k_0^{\frac12+\frac1p}
\right),
\qquad 1\leq p\leq\infty.
\end{aligned}
\end{equation*}

Moreover, writing
\(
m^p(k):=m^p(y;\lambda(k)),
\
y=-\sqrt{3\tau},
\
\lambda(k)=-\frac{2\sqrt{t}}{3\sqrt{3}}k,
\)
we have
\begin{equation*}
\left\|(m^p)^{-1}-I\right\|_
{L^1(\partial D_\tau(0))}
=
\mathcal{O}\left(\sqrt{\frac{\tau}{t}}\right),
\qquad
\left\|(m^p)^{-1}-I\right\|_
{L^\infty(\partial D_\tau(0))}
=
\mathcal{O}\left(\frac{\tau}{\sqrt{t}}\right).
\end{equation*}
Furthermore, with \(\partial D_\tau(0)\) oriented counterclockwise,
\begin{equation*}
\frac{1}{2\pi i}
\int_{\partial D_\tau(0)}
\left((m^p)^{-1}(k)-I\right)\,\mathrm{d}k
=
\frac{3\sqrt{3}}{2\sqrt{t}}\,m_1^p(y)
+
\mathcal{O}\left(\frac{\tau^{3/2}}{t}\right).
\end{equation*}
All estimates are uniform for
\(
1\ll\tau\leq C_0t^l
\)
with \(0<l<1/5\).
\end{lemma}
\begin{proof}
The proof is similar to that of Lemma~\ref{v2-vp}, so we only present
the modifications required for the contours
\(\Sigma_1^{(2)}\) and \(\Sigma_7^{(2)}\).
 By the lemma \ref{lem:dbar-extension-r2} and lemma \ref{lem:dbar-extension-Rj}, we have
\[
\begin{aligned}
\left|[v_1^{(2)}-v_1^p]_{21}\right|
&\leq
C\left(
|k-k_0|^{1/2}+k_0^{1/2}
\right)
\left|e^{\vartheta_{21}(x,t,k)}\right|.
\end{aligned}
\]
Parameterize \(\Sigma_1^{(2)}\) by
\(
k=k_0+ue^{i\theta_1}
\)
with \(u\geq0\), where \(\theta_1\) denotes the direction of
\(\Sigma_1^{(2)}\). Along this contour,
\(
\operatorname{Re}\vartheta_{21}(x,t,k)\leq-ctu^2
\)
for some \(c>0\). Consequently,
\[
\left|[v_1^{(2)}-v_1^p]_{21}\right|
\leq
C\left(k_0^{1/2}+u^{1/2}\right)e^{-ctu^2}.
\]
The same estimate holds for the \((2,3)\) entry, and hence
\begin{equation*}
\left|v_1^{(2)}(x,t,k)-v_1^p(x,t,k)\right|
\leq
C\left(k_0^{1/2}+u^{1/2}\right)e^{-ctu^2}.
\end{equation*}
It follows that, for \(1\leq p<\infty\),
\[
\begin{aligned}
\left\|v_1^{(2)}-v_1^p\right\|_
{L^p(\Sigma_1^{(2)}\cap D_\tau(0))}
&\leq
Ck_0^{1/2}
\left(\int_0^\infty e^{-cptu^2}\,\mathrm du\right)^{1/p}+
C\left(\int_0^\infty
u^{p/2}e^{-cptu^2}\,\mathrm du\right)^{1/p}\\
&\leq
C\left(
k_0^{1/2}t^{-1/(2p)}
+t^{-1/4-1/(2p)}
\right).
\end{aligned}
\]
For \(p=\infty\), the corresponding estimate is
\[
\left\|v_1^{(2)}-v_1^p\right\|_
{L^\infty(\Sigma_1^{(2)}\cap D_\tau(0))}
\leq
C\left(k_0^{1/2}+t^{-1/4}\right).
\]

We next consider \(\Sigma_7^{(2)}\). On its nondecaying part,
parameterized by \(0\leq k\leq k_0\), the \(H^1\)-regularity of the
reflection coefficients gives
\(
\left|v_7^{(2)}-v_7^p\right|
\leq Ck^{1/2},
\)
where $v_7^p:=v_6^pv_1^p$.

Therefore, for \(1\leq p<\infty\),
\[
\begin{aligned}
\left\|v_7^{(2)}-v_7^p\right\|_
{L^p(\Sigma_7^{(2)}\cap D_\tau(0))}
&\leq
C\left(\int_0^{k_0}k^{p/2}\,\mathrm dk\right)^{1/p}
\leq
Ck_0^{1/2+1/p},
\end{aligned}
\]
whereas
\(
\left\|v_7^{(2)}-v_7^p\right\|_
{L^\infty(\Sigma_7^{(2)}\cap D_\tau(0))}
\leq Ck_0^{1/2}.
\) Combining these
estimates yields
\[
\left\|v^{(2)}-v^p\right\|_
{L^p(\Sigma^{(2)}\cap D_\tau(0))}
=
\mathcal{O}\left(
t^{-1/4-1/(2p)}
+k_0^{1/2}t^{-1/(2p)}
+k_0^{1/2+1/p}
\right),
\]
with the usual convention when \(p=\infty\). 
\end{proof}

In addition to the modified estimates for the RH problem, the estimates for the pure \(\dbar\)-problem in Lemma~\ref{lem:m3 expansion} are modified as follows.
\begin{lemma}
Let \(C_0>0\) and \(0<l<1/5\). Suppose that
\(
1\ll\tau=tk_0^2\leq C_0t^l,
\
k_0=\frac{x}{2t}.
\)
Assume that the reflection coefficients satisfy the hypotheses of
Lemmas~\ref{lem:dbar-extension-r2} and
\ref{lem:dbar-extension-Rj}. Then, for sufficiently large \(t\),
the function \(m^{(3)}(x,t,k)\) admits the expansion
\begin{equation*}
m^{(3)}(x,t,k)
=
I+\frac{m_1^{(3)}(x,t)}{k}
+o\left(\frac1k\right),
\qquad k\to\infty,
\end{equation*}
where
\begin{equation*}
\left|m_1^{(3)}(x,t)\right|
\leq
C\left(t^{-3/4}+k_0^{3/2}\right).
\end{equation*}
All estimates are uniform for
\(1\ll\tau\leq C_0t^l\) with \(0<l<1/5\).
\end{lemma}
\begin{proof}
The proof follows from that of Lemma~\ref{lem:m3 expansion}. In particular, the modified estimate results from the estimate of
\(
\iint_{\Omega_7}
|\mathcal{W}(\zeta;x,t)|\,\mathrm dA(\zeta).
\)
\end{proof}

As a result, following the same procedure as in Section~\ref{sed:P4}, tracing back the transformations and using the reconstruction formulas for \(p(x,t)\) and \(q(x,t)\), we obtain the following asymptotic formulas.

\begin{equation}\label{pq Painleve-t1}
\begin{aligned}
p(x,t)
&=
-\frac{3}{2}
\left[
\frac{1}{2\pi i}
\int_{\partial D_\tau(0)}
\left(
[\hat w(x,t,k)]_{31}
+
[\hat w(x,t,k)]_{32}
\right)\mathrm dk
\right]
+\mathcal O\left(
t^{-\frac34}
+k_0^{\frac12}t^{-\frac12}
+k_0^{\frac32}
\right)
\\
&=
-\frac{9\sqrt3}{4\sqrt t}
\left(
[m_1^p(y)]_{31}
+
[m_1^p(y)]_{32}
\right)
+\mathcal O\left(
t^{-\frac34}
+k_0^{\frac12}t^{-\frac12}
+k_0^{\frac32}
\right)
\\
&=
\frac{\sqrt3}{4\sqrt t}
\frac{P_{\rm IV}'(y)+\frac23}{P_{\rm IV}(y)}
+\mathcal O\left(
t^{-\frac34}
+k_0^{\frac12}t^{-\frac12}
+k_0^{\frac32}
\right),
\\[2mm]
q(x,t)
&=
-\frac{9\sqrt3}{4\sqrt t}
\frac{
[m_1^p(y)]_{31}-[m_1^p(y)]_{32}
}{
2\omega(\omega-1)
}
+\mathcal O\left(
t^{-\frac34}
+k_0^{\frac12}t^{-\frac12}
+k_0^{\frac32}
\right)
\\
&=
\frac{\sqrt3}{4\sqrt t}
\left(
P_{\rm IV}(y)+\frac23y
\right)
+\mathcal O\left(
t^{-\frac34}
+k_0^{\frac12}t^{-\frac12}
+k_0^{\frac32}
\right).
\end{aligned}
\end{equation}

\begin{remark}
When \(l=0\), we have \(\tau=\mathcal{O}(1)\) and \(k_0=\mathcal{O}(t^{-1/2})\). Consequently, \(k_0^{1/2}t^{-1/2}=\mathcal{O}(t^{-3/4})\) and \(k_0^{3/2}=\mathcal{O}(t^{-3/4})\). Thus, the error estimate \(\mathcal{O}\bigl(t^{-3/4}+k_0^{1/2}t^{-1/2}+k_0^{3/2}\bigr)\) reduces to \(\mathcal{O}(t^{-3/4})\), which agrees with the error estimate obtained previously in the Painlevé region.
\end{remark}

\section{The transition from Painlevé IV to dispersive waves}\label{sec:trans2}
In this section, we consider the transition region \(c_4t^{3/5}\leq x<c_5t\), where the leading-order behavior is described by dispersive waves. In contrast to the preceding analysis of the Painlevé region, we start from the RH \(\dbar\)-problem for \(m^{(1)}(x,t,k)\). We introduce the rescaled spectral variable \(z=k/k_0\) and carry out the analysis in the \(z\)-plane.
Under the scaling \(k=k_0z\), with \(\tau=tk_0^2\), define
\(m^{(1)}(z;x,t):=m^{(1)}(k_0z;x,t)\) and
\(\Gamma_j^{(1)}:=k_0^{-1}\Sigma_j^{(1)}\).
The boundary values of \( m^{(1)}\) satisfy
\[
 m^{(1)}_+(z;x,t)
=
 m^{(1)}_-(z;x,t) v_j^{(1)}(z;x,t),
\qquad
z\in\Gamma_j^{(1)},\quad j=1,2,3,
\]
where
\begin{align*}
 v_1^{(1)}(z;x,t)
&=
\begin{pmatrix}
1
&
-r_1(k_0z)e^{-\vartheta_{21}(\tau;z)}
&
\bigl[r_1(k_0z)r_2^*(0)+r_2(0)\bigr]
e^{-\vartheta_{31}(\tau;z)}
\\
r_1^*(k_0z)e^{\vartheta_{21}(\tau;z)}
&
1-|r_1(k_0z)|^2
&
-r_2^*(0)\bigl(1-|r_1(k_0z)|^2\bigr)
e^{-\vartheta_{32}(\tau;z)}
\\
0&0&1
\end{pmatrix},
\\[3mm]
 v_2^{(1)}(z;x,t)
&=
\begin{pmatrix}
1-|r_1(\omega^2k_0z)|^2
&
-r_2^*(0)\bigl(1-|r_1(\omega^2k_0z)|^2\bigr)
e^{-\vartheta_{21}(\tau;z)}
&
r_1^*(\omega^2k_0z)e^{-\vartheta_{31}(\tau;z)}
\\
0&1&0
\\
-r_1(\omega^2k_0z)e^{\vartheta_{31}(\tau;z)}
&
\bigl[r_1(\omega^2k_0z)r_2^*(0)+r_2(0)\bigr]
e^{\vartheta_{32}(\tau;z)}
&
1
\end{pmatrix},
\\[3mm]
 v_3^{(1)}(z;x,t)
&=
\begin{pmatrix}
1&0&0
\\
\bigl[r_1(\omega k_0z)r_2^*(0)+r_2(0)\bigr]
e^{\vartheta_{21}(\tau;z)}
&
1
&
-r_1(\omega k_0z)e^{-\vartheta_{32}(\tau;z)}
\\
-r_2^*(0)\bigl(1-|r_1(\omega k_0z)|^2\bigr)
e^{\vartheta_{31}(\tau;z)}
&
r_1^*(\omega k_0z)e^{\vartheta_{32}(\tau;z)}
&
1-|r_1(\omega k_0z)|^2
\end{pmatrix}.
\end{align*}
Here, the rescaled phase functions are given by
\begin{equation*}
\vartheta_{ij}(\tau;z)
=
\tau\left(
2(\omega^i-\omega^j)z
+(\omega^{2i}-\omega^{2j})z^2
\right),
\qquad 1\leq j<i\leq3.
\end{equation*}
The subsequent contour deformations, construction of the local parametrices, and small-norm analysis closely follow the corresponding arguments in \cite{Yang-2026}. We therefore omit the repeated details and focus mainly on the estimates of the resulting error terms. Particular attention will be paid to the uniformity of these estimates as \(k_0\to0\) and \(\tau\to\infty\), which is the main difference between the present transition region and the standard dispersive-wave region.
\subsection{The second transformation}
In the \(z\)-plane, we introduce the scalar function
\(\delta_1(z;k_0)\), which is analytic in
\(\mathbb{C}\setminus[1,\infty)\) and satisfies the jump condition
\[
\delta_{1,+}(z;k_0)
=
\delta_{1,-}(z;k_0)
\left(1-|r_1(k_0z)|^2\right),
\qquad z\in(1,\infty),
\]
together with the normalization condition
\[
\delta_1(z;k_0)
=
1+\mathcal{O}(z^{-1}),
\qquad z\to\infty.
\]
By the Plemelj formula, \(\delta_1(z;k_0)\) is given by
\begin{equation*}
\delta_1(z;k_0)
=
\exp\left(
\frac{1}{2\pi i}
\int_1^\infty
\frac{\ln\left(1-|r_1(k_0s)|^2\right)}
{s-z}\,\mathrm ds
\right),
\qquad
z\in\mathbb{C}\setminus[1,\infty).
\end{equation*}
For simplicity, the dependence of \(\delta_1(z;k_0)\) on \(k_0\) will be
suppressed whenever no confusion can arise.

\begin{proposition}
Suppose that \(r_1\in H^1(\mathbb{R}_+)\) and
\(\|r_1\|_{L^\infty(\mathbb{R}_+)}\leq\sigma<1\). Then
\(\delta_1(z;k_0)\) has the following properties:
\begin{enumerate}
    \item[(i)]
    The function \(\delta_1(z;k_0)\) admits the representation
    \begin{equation*}
    \delta_1(z;k_0)
    =
    e^{-i\nu_1(k_0)\log_0(z-1)}
    e^{-\chi_1(z;k_0)},
    \end{equation*}
    where
    \[
    \nu_1(k_0)
    =
    -\frac{1}{2\pi}
    \ln\left(1-|r_1(k_0)|^2\right)
    \]
    and
    \[
    \chi_1(z;k_0)
    =
    \frac{1}{2\pi i}
    \int_1^\infty
    \log_0(z-s)\,
    \mathrm d\ln\left(1-|r_1(k_0s)|^2\right).
    \]

    \item[(ii)]
    The functions \(\delta_1(z;k_0)\) and
    \(\delta_1(z;k_0)^{-1}\) are analytic in
    \(\mathbb{C}\setminus[1,\infty)\) and admit continuous boundary
    values on \((1,\infty)\). Moreover,
    \[
    \sup_{z\in\mathbb{C}\setminus[1,\infty)}
    |\delta_1(z;k_0)^{\pm1}|
    \leq C,
    \]
    where \(C\) is independent of \(k_0\).

    \item[(iii)]
    As \(z\to1\) along a path non-tangential to \([1,\infty)\),
    \begin{equation*}
    |\chi_1(z;k_0)-\chi_1(1;k_0)|
    \leq
    Ck_0^{1/2}|z-1|^{1/2}.
    \end{equation*}
\end{enumerate}
\end{proposition}

\begin{proposition}\label{prop:Rj-PC}
Suppose that \(r_1\in H^1(\mathbb R_+)\) and
\(\|r_1\|_{L^\infty(\mathbb R_+)}\leq\sigma<1\).
Let \(k_0>0\), set \(a=(\sqrt3-1)/2\), and define
\(\rho_1(s)=r_1(s)/(1-|r_1(s)|^2)\).
There exist bounded functions \(R_j(z)\), \(j=1,\ldots,6\), in the
corresponding regions \(V_j\), with the following boundary values:
\begin{align*}
R_1(z)
&=
\begin{cases}
\rho_1^*(k_0z), & z\in(1,\infty),\\
f_1(z), & z\in\Gamma_1^{(2)},
\end{cases}
&
R_2(z)
&=
\begin{cases}
r_1(k_0z), & z\in(a,1),\\
f_2(z), & z\in\Gamma_2^{(2)},
\end{cases}
\\[2mm]
R_3(z)
&=
\begin{cases}
r_1^*(k_0z), & z\in(a,1),\\
f_3(z), & z\in\Gamma_3^{(2)},
\end{cases}
&
R_4(z)
&=
\begin{cases}
\rho_1(k_0z), & z\in(1,\infty),\\
f_4(z), & z\in\Gamma_4^{(2)},
\end{cases}
\\[2mm]
R_5(z)
&=
\begin{cases}
r_1(k_0z), & z\in(0,a),\\
r_1(0), & z\in\Gamma_{19}^{(2)},
\end{cases}
&
R_6(z)
&=
\begin{cases}
r_1^*(k_0z), & z\in(0,a),\\
r_1^*(0), & z\in\Gamma_{21}^{(2)},
\end{cases}
\end{align*}
where
\[
\nu_1
=
-\frac{1}{2\pi}\ln\left(1-|r_1(k_0)|^2\right),
\qquad
\delta_3(z)=\delta_1(\omega^2z),
\qquad
\delta_5(z)=\delta_1(\omega z),
\]
and
\begin{align*}
f_1(z)
&=
\rho_1^*(k_0)(z-1)^{-2i\nu_1}
e^{-2\chi_1(1;k_0)}
\frac{\delta_3(z)\delta_5(z)}
{\delta_3(1)\delta_5(1)\delta_1(z)^2},
\\
f_2(z)
&=
r_1(k_0)(z-1)^{2i\nu_1}
e^{2\chi_1(1;k_0)}
\frac{\delta_3(1)\delta_5(1)\delta_1(z)^2}
{\delta_3(z)\delta_5(z)},
\\
f_3(z)
&=
r_1^*(k_0)(z-1)^{-2i\nu_1}
e^{-2\chi_1(1;k_0)}
\frac{\delta_3(z)\delta_5(z)}
{\delta_3(1)\delta_5(1)\delta_1(z)^2},
\\
f_4(z)
&=
\rho_1(k_0)(z-1)^{2i\nu_1}
e^{2\chi_1(1;k_0)}
\frac{\delta_3(1)\delta_5(1)\delta_1(z)^2}
{\delta_3(z)\delta_5(z)}.
\end{align*}
The powers of \(z-1\) are defined using \(\log_0(z-1)\), consistently
with the factorization of \(\delta_1\). The values of \(\delta_1(z)\)
on each lens contour are taken from the corresponding sector.

Moreover, the extensions can be chosen so that
\[
|R_j(z)|\leq C,
\qquad z\in V_j,\quad j=1,\ldots,6,
\]
and, almost everywhere in the respective regions,
\begin{align*}
|\dbar_z R_j(z)|
&\leq
Ck_0^{1/2}\|r_1'\|_{L^2(\mathbb R_+)}
|z-1|^{-1/2}
+
Ck_0\left|r_1'(k_0\operatorname{Re}z)\right|
+
CH_j(z),
&&j=1,\ldots,4,
\\
|\dbar_z R_j(z)|
&\leq
Ck_0^{1/2}\|r_1'\|_{L^2(\mathbb R_+)}
|z|^{-1/2}
+
Ck_0\left|r_1'(k_0\operatorname{Re}z)\right|,
&&j=5,6.
\end{align*}
Here,
\[
H_j(z)
=
\frac{\mathbf1_{E_j}(z)}{1+|z-1|},
\qquad j=1,\ldots,4,
\]
where $E_j\subset V_j$ is a fixed angular subregion containing
the support of the derivative of the angular cutoff function
used to construct $R_j$.
The constants $C,c>0$ depend only on $\sigma$,
$\|r_1\|_{H^1(\mathbb R_+)}$, and the fixed contour and cutoff
geometry, and are independent of $k_0$ and $\tau$.
\end{proposition}
\begin{proof}
We prove the estimate for $j=1$; the remaining cases are
analogous. Write
\[
z-1=re^{i\phi},
\qquad r>0,\quad 0<\phi<\frac{\pi}{4}.
\]
Fix $0<\phi_a<\phi_b<\pi/4$ and choose
$\eta\in C^\infty([0,\pi/4])$ such that
\[
0\leq\eta\leq1,
\qquad
\eta(\phi)=1\quad\text{for }\phi\leq\phi_a,
\qquad
\eta(\phi)=0\quad\text{for }\phi\geq\phi_b.
\]
Set
\[
E_1
=
\{z\in V_1:\phi_a\leq\arg(z-1)\leq\phi_b\},
\qquad
H_1(z)=\frac{\mathbf1_{E_1}(z)}{1+|z-1|},
\]
and define
\[
R_1(z)
=
\eta(\phi)\rho_1^*(k_0\operatorname{Re}z)
+
(1-\eta(\phi))f_1(z).
\]
This construction satisfies the prescribed boundary conditions.
Since $f_1$ is analytic in $V_1$, differentiation gives,
almost everywhere,
\begin{equation}\label{dbar:R1-PC-proof}
\dbar_zR_1(z)
=
\frac{k_0}{2}\eta(\phi)
(\rho_1^*)'(k_0\operatorname{Re}z)
+
\frac{ie^{i\phi}}{2r}\eta'(\phi)
\left[
\rho_1^*(k_0\operatorname{Re}z)-f_1(z)
\right].
\end{equation}
Denote $D(z)=\frac{\delta_3(1)\delta_5(1)}
{\delta_3(z)\delta_5(z)}$, then we have
\begin{equation*}
|D(z)-1|\leq C|z-1|,
\qquad |z-1|<\epsilon.
\end{equation*}
Using the same estimate as in Lemma \ref{lem:dbar-extension-Rj}, we obtain
\[
\left|
\rho_1^*(k_0\operatorname{Re}z)-f_1(z)
\right|
\leq
Ck_0^{1/2}\|r_1'\|_{L^2(\mathbb R_+)}r^{1/2}
+Cr.
\]
For $r\geq\epsilon$, the uniform boundedness of $\rho_1$ and
$f_1$ gives a constant bound for the same difference.
Consequently, after adjusting $C$, these estimates combine as
\begin{equation}\label{estimate:interpolation-difference-PC}
\left|
\rho_1^*(k_0\operatorname{Re}z)-f_1(z)
\right|
\leq
Ck_0^{1/2}\|r_1'\|_{L^2(\mathbb R_+)}r^{1/2}
+
C\frac{r}{1+r},
\qquad z\in E_1.
\end{equation}

Finally, $|\eta'(\phi)|\leq C\mathbf1_{E_1}(z)$.
Substituting \eqref{estimate:interpolation-difference-PC}
into \eqref{dbar:R1-PC-proof}, we conclude that
\[
\begin{aligned}
|\dbar_zR_1(z)|
&\leq
Ck_0|r_1'(k_0\operatorname{Re}z)|
\\
&\quad+
\frac{C\mathbf1_{E_1}(z)}{r}
\left(
k_0^{1/2}\|r_1'\|_{L^2(\mathbb R_+)}r^{1/2}
+\frac{r}{1+r}
\right)
\\
&\leq
Ck_0^{1/2}\|r_1'\|_{L^2(\mathbb R_+)}
|z-1|^{-1/2}
+
Ck_0|r_1'(k_0\operatorname{Re}z)|
+
CH_1(z).
\end{aligned}
\]
\end{proof}
\begin{figure}[H]
    \centering
    \begin{tikzpicture}[scale=1.6]
        \coordinate (center) at (0,0);

        \foreach \angle in {0,120,240}{
            \begin{scope}[rotate around={\angle:(center)}]
                \draw[dashed,very thick,black] (3,0) -- (-3,0);

                \fill (2,0) circle (1.5pt);

                \draw[very thick,black]
                    (2/2.732,2*1.732/2.732) -- (2.8,-0.8);
                \draw[very thick,black]
                    (2/2.732,-2*1.732/2.732) -- (2.8,0.8);

                \draw[very thick,black,latex-]
                    (1.3,0.7) -- (2,0);
                \draw[very thick,black,-latex]
                    (2,0) -- (2.5,-0.5);
                \draw[very thick,black,latex-]
                    (1.3,-0.7) -- (2,0);
                \draw[very thick,black,-latex]
                    (2,0) -- (2.5,0.5);

                \draw[very thick,black]
                    (2/2.732,2*1.732/2.732)
                    -- (2/2.732,-2*1.732/2.732);
                \draw[very thick,black,-latex]
                    (2/2.732,0) -- (2/2.732,-0.7);
                \draw[very thick,black,latex-]
                    (2/2.732,0.7) -- (2/2.732,0);

                \fill (2/2.732,0) circle (1pt);
            \end{scope}
        }

        \node[red!70!black,above] at (2.3,0.4) {$1$};
        \node[red!70!black,above] at (1.4,0.65) {$2$};
        \node[red!70!black,below] at (1.4,-0.65) {$3$};
        \node[red!70!black,below] at (2.3,-0.4) {$4$};
        \node[gray,below] at (0.25,0.7) {$19$};
        \node[red!70!black,right] at (2/2.732,0.6) {$5$};
        \node[red!70!black,right] at (2/2.732,-0.6) {$6$};

        \begin{scope}[rotate around={120:(center)}]
            \node[red!70!black,below] at (2.35,0.45) {$7$};
            \node[red!70!black,above] at (1.35,0.85) {$8$};
            \node[red!70!black,below] at (1.4,-0.6) {$9$};
            \node[red!70!black,right] at (2.4,-0.4) {$10$};
            \node[gray,below] at (0.45,0.45) {$20$};
            \node[red!70!black,right] at (2.7/2.732,0.7) {$11$};
            \node[red!70!black,right] at (2.7/2.6,-0.4) {$12$};
        \end{scope}

        \begin{scope}[rotate around={240:(center)}]
            \node[red!70!black,above] at (2.5,0.7) {$13$};
            \node[red!70!black,above] at (1.35,0.6) {$14$};
            \node[red!70!black,below] at (1.35,-0.9) {$15$};
            \node[red!70!black,above] at (2.4,-0.45) {$16$};
            \node[gray,below] at (0.2,0.45) {$21$};
            \node[red!70!black,right] at (2.7/2.732,0.5) {$17$};
            \node[red!70!black,right] at (2.7/2.732,-0.7) {$18$};
        \end{scope}

        \node[above] at (2,0.1) {$1$};
        \node[right] at (-0.9,1.9) {$\omega$};
        \node[right] at (-0.9,-1.9) {$\omega^2$};

        \node[above] at (2.7,0.1) {$V_1$};
        \node[above] at (1.25,0.1) {$V_2$};
        \node[below] at (1.25,-0.2) {$V_3$};
        \node[below] at (2.7,-0.2) {$V_4$};
        \node[above] at (0.5,0.1) {$V_5$};
        \node[below] at (0.5,-0.1) {$V_6$};
        \node[below] at (2,1.3) {$V_7$};
        \node[below] at (2,-1.3) {$V_8$};

        \node[right] at (3,0) {$\operatorname{Re}z$};

        \draw[very thick,dashed,black]
            (2+0.2/1.414,0.2/1.414) -- (2+0.2*1.414,0);
        \draw[very thick,dashed,black]
            (2+0.2*1.414,0) -- (2+0.2/1.414,-0.2/1.414);

        \fill (0,0) circle (1pt);
    \end{tikzpicture}
    \caption{The jump contour $\Gamma^{(2)}$ in the complex $z$-plane.}
    \label{figGamma2}
\end{figure}
Now introduce the matrix-valued function $R^{(2)}(z;\tau)$ piecewise by setting $R^{(2)}(z;\tau)=\{R_j^{(2)}(z;\tau)\}_{j=1}^{6}$ for $z\in V_j$, $j=1,\ldots,6$, where the matrices $R_j^{(2)}(z;\tau)$ are defined as follows: 
\[
\scalebox{0.9}{$\displaystyle
\begin{aligned}
R_1^{(2)}
&=
\begin{pmatrix}
1 & 0 & 0\\
-\frac{\delta_{1+}^2}{\delta_3\delta_5}
R_1(z)e^{\vartheta_{21}(\tau;z)}
& 1 &
\frac{\delta_{1+}\delta_3}{\delta_5^2}
r_2^*(0)e^{-\vartheta_{32}(\tau;z)}\\
0 & 0 & 1
\end{pmatrix},
&
R_2^{(2)}
&=
\begin{pmatrix}
1 &
\frac{\delta_3\delta_5}{\delta_1^2}
R_2(z)e^{-\vartheta_{21}(\tau;z)}
&
-\frac{\delta_3^2}{\delta_1\delta_5}
\alpha_2(z)e^{-\vartheta_{31}(\tau;z)}\\
0 & 1 & 0\\
0 & 0 & 1
\end{pmatrix},
\\[4mm]
R_3^{(2)}
&=
\begin{pmatrix}
1 & 0 & 0\\
\frac{\delta_1^2}{\delta_3\delta_5}
R_3(z)e^{\vartheta_{21}(\tau;z)}
& 1 &
-\frac{\delta_1\delta_3}{\delta_5^2}
\beta_3(z)e^{-\vartheta_{32}(\tau;z)}\\
0 & 0 & 1
\end{pmatrix},
&
R_4^{(2)}
&=
\begin{pmatrix}
1 &
-\frac{\delta_3\delta_5}{\delta_{1-}^2}
R_4(z)e^{-\vartheta_{21}(\tau;z)}
&
\frac{\delta_3^2}{\delta_{1-}\delta_5}
r_2(0)e^{-\vartheta_{31}(\tau;z)}\\
0 & 1 & 0\\
0 & 0 & 1
\end{pmatrix},
\\[4mm]
R_5^{(2)}
&=
\begin{pmatrix}
1 &
\frac{\delta_3\delta_5}{\delta_1^2}
R_5(z)e^{-\vartheta_{21}(\tau;z)}
&
-\frac{\delta_3^2}{\delta_1\delta_5}
\alpha_5(z)e^{-\vartheta_{31}(\tau;z)}\\
0 & 1 & 0\\
0 & 0 & 1
\end{pmatrix},
&
R_6^{(2)}
&=
\begin{pmatrix}
1 & 0 & 0\\
\frac{\delta_1^2}{\delta_3\delta_5}
R_6(z)e^{\vartheta_{21}(\tau;z)}
& 1 &
-\frac{\delta_1\delta_3}{\delta_5^2}
\beta_6(z)e^{-\vartheta_{32}(\tau;z)}\\
0 & 0 & 1
\end{pmatrix}.
\end{aligned}
$}
\]
Here, \(\alpha_j(z)=r_2(0)+R_j(z)r_2^*(0)\) for
\(z\in V_j\), \(j=2,5\), and
\(\beta_j(z)=r_2^*(0)+R_j(z)r_2(0)\) for
\(z\in V_j\), \(j=3,6\).

We now introduce the second transformation
\[
m^{(2)}(z;\tau)
=
m^{(1)}(z;\tau)\Delta(z)R^{(2)}(z;\tau).
\]
Then the matrix-valued function \(m^{(2)}(z;\tau)\) satisfies the following RH \(\dbar\)-problem. 
\begin{RHdbarproblem}
Find a \(3\times3\) matrix-valued function \(m^{(2)}(z;\tau)\)
satisfying the following properties.
\begin{enumerate}
\item[(i)]
The function \(m^{(2)}(z;\tau)\) is continuous in
\(\mathbb C\setminus\Gamma^{(2)}\), where the jump contour
\(\Gamma^{(2)}=\bigcup_{j=1}^{18}\Gamma_j^{(2)}\) is depicted in
Figure~\ref{figGamma2}. Moreover, \(m^{(2)}(z;\tau)\) admits
continuous boundary values \(m_\pm^{(2)}(z;\tau)\) on each open
smooth component of \(\Gamma^{(2)}\), away from the contour
junctions.

\item[(ii)]
The boundary values satisfy
\[
m_+^{(2)}(z;\tau)
=
m_-^{(2)}(z;\tau)v_j^{(2)}(z;\tau),
\qquad
z\in\Gamma_j^{(2)},\quad j=1,\ldots,18.
\]
With the orientations shown in Figure~\ref{figGamma2}, the first
six jump matrices are given by
\begin{equation*}
\scalebox{0.9}{$\displaystyle
\begin{aligned}
v_1^{(2)}
&=
\begin{pmatrix}
1&0&0\\
\frac{\delta_{1+}^2}{\delta_3\delta_5}
R_1(z)e^{\vartheta_{21}}
&
1
&
-\frac{\delta_{1+}\delta_3}{\delta_5^2}
r_2^*(0)e^{-\vartheta_{32}}
\\
0&0&1
\end{pmatrix},
\quad
v_2^{(2)}
=
\begin{pmatrix}
1
&
\frac{\delta_3\delta_5}{\delta_1^2}
R_2(z)e^{-\vartheta_{21}}
&
-\frac{\delta_3^2}{\delta_1\delta_5}
\alpha_2(z)e^{-\vartheta_{31}}
\\
0&1&0\\
0&0&1
\end{pmatrix},
\\[4mm]
v_3^{(2)}
&=
\begin{pmatrix}
1&0&0\\
\frac{\delta_1^2}{\delta_3\delta_5}
R_3(z)e^{\vartheta_{21}}
&
1
&
-\frac{\delta_1\delta_3}{\delta_5^2}
\beta_3(z)e^{-\vartheta_{32}}
\\
0&0&1
\end{pmatrix},
\quad
v_4^{(2)}
=
\begin{pmatrix}
1
&
-\frac{\delta_3\delta_5}{\delta_{1-}^2}
R_4(z)e^{-\vartheta_{21}}
&
\frac{\delta_3^2}{\delta_{1-}\delta_5}
r_2(0)e^{-\vartheta_{31}}
\\
0&1&0\\
0&0&1
\end{pmatrix},
\\[4mm]
v_5^{(2)}
&=
\begin{pmatrix}
1
&
\frac{\delta_3\delta_5}{\delta_1^2}
\bigl(R_2(z)-R_5(z)\bigr)e^{-\vartheta_{21}}
&
-\frac{\delta_3^2}{\delta_1\delta_5}
\bigl(\alpha_2(z)-\alpha_5(z)\bigr)e^{-\vartheta_{31}}
\\
0&1&0\\
0&0&1
\end{pmatrix},\\
v_6^{(2)}
&=
\begin{pmatrix}
1&0&0\\
-\frac{\delta_1^2}{\delta_3\delta_5}
\bigl(R_3(z)-R_6(z)\bigr)e^{\vartheta_{21}}
&
1
&
\frac{\delta_1\delta_3}{\delta_5^2}
\bigl(\beta_3(z)-\beta_6(z)\bigr)e^{-\vartheta_{32}}
\\
0&0&1
\end{pmatrix}.
\end{aligned}
$}
\end{equation*}
The remaining jump matrices are determined by the
\(\mathbb Z_3\) symmetry:
\[
\begin{aligned}
v_{j+6}^{(2)}(z;\tau)
&=
\mathcal A^{-1}
v_j^{(2)}(\omega^2z;\tau)\mathcal A,
\\
v_{j+12}^{(2)}(z;\tau)
&=
\mathcal A^{-2}
v_j^{(2)}(\omega z;\tau)\mathcal A^2,
\qquad j=1,\ldots,6.
\end{aligned}
\]
The compatibility relation
\(r_1(0)+r_2^*(0)+r_1^*(0)r_2(0)=0\) implies that
\(v_{19}^{(2)}=v_{20}^{(2)}=v_{21}^{(2)}=I\)
and the three auxiliary segments
\(\Gamma_{19}^{(2)}\), \(\Gamma_{20}^{(2)}\), and
\(\Gamma_{21}^{(2)}\) carry no jumps.

\item[(iii)]
Away from \(\Gamma^{(2)}\), the function \(m^{(2)}(z;\tau)\)
satisfies, almost everywhere,
\begin{equation*}
\dbar_zm^{(2)}(z;\tau)
=
m^{(2)}(z;\tau)\mathcal W^{(2)}(z;\tau),
\end{equation*}
where
\(
\mathcal W^{(2)}(z;\tau)=\mathcal W_j^{(2)}(z;\tau)
\)
for \(z\in V_j\), the coefficients in the first eight regions
are given by
\begin{align*}
\mathcal W_1^{(2)}
&=
\begin{pmatrix}
0&0&0\\
-\frac{\delta_1^2}{\delta_3\delta_5}
\dbar_zR_1\,e^{\vartheta_{21}}
&
0
&
-\frac{\delta_1\delta_3}{\delta_5^2}
\dbar_z G_1\,e^{-\vartheta_{32}}
\\
0&0&0
\end{pmatrix},
\\[3mm]
\mathcal W_j^{(2)}
&=
\begin{pmatrix}
0
&
\frac{\delta_3\delta_5}{\delta_1^2}
\dbar_zR_j\,e^{-\vartheta_{21}}
&
\frac{\delta_3^2}{\delta_1\delta_5}
\bigl(R_j\dbar_z G_1-\dbar_z\alpha_j\bigr)
e^{-\vartheta_{31}}
\\
0&0&
-\frac{\delta_1\delta_3}{\delta_5^2}
\dbar_z G_1\,e^{-\vartheta_{32}}
\\
0&0&0
\end{pmatrix},
\qquad j=2,5,
\\[3mm]
\mathcal W_j^{(2)}
&=
\begin{pmatrix}
0&0&
-\frac{\delta_3^2}{\delta_1\delta_5}
\dbar_z G_6\,e^{-\vartheta_{31}}
\\
\frac{\delta_1^2}{\delta_3\delta_5}
\dbar_zR_j\,e^{\vartheta_{21}}
&
0
&
\frac{\delta_1\delta_3}{\delta_5^2}
\bigl(R_j\dbar_z G_6-\dbar_z\beta_j\bigr)
e^{-\vartheta_{32}}
\\
0&0&0
\end{pmatrix},
\qquad j=3,6,
\\[3mm]
\mathcal W_4^{(2)}
&=
\begin{pmatrix}
0
&
-\frac{\delta_3\delta_5}{\delta_1^2}
\dbar_zR_4\,e^{-\vartheta_{21}}
&
-\frac{\delta_3^2}{\delta_1\delta_5}
\dbar_z G_6\,e^{-\vartheta_{31}}
\\
0&0&0\\
0&0&0
\end{pmatrix},
\\[3mm]
\mathcal W_7^{(2)}
&=
\begin{pmatrix}
0&0&0\\
0&0&
-\frac{\delta_1\delta_3}{\delta_5^2}
\dbar_z G_1\,e^{-\vartheta_{32}}
\\
0&0&0
\end{pmatrix},
\quad
\mathcal W_8^{(2)}
=
\begin{pmatrix}
0&0&
-\frac{\delta_3^2}{\delta_1\delta_5}
\dbar_z G_6\,e^{-\vartheta_{31}}
\\
0&0&0\\
0&0&0
\end{pmatrix}.
\end{align*}
All functions in these formulas are evaluated at \(z\), and the
\(\delta\)-functions are taken in the corresponding sectors.
In particular,
\[
\dbar_z\alpha_j=r_2^*(0)\dbar_zR_j,\quad j=2,5,
\qquad
\dbar_z\beta_j=r_2(0)\dbar_zR_j,\quad j=3,6.
\]

\item[(iv)]
As \(z\to\infty\),
\[
m^{(2)}(z;\tau)\to I.
\]

\item[(v)]
The function \(m^{(2)}(z;\tau)\) satisfies the
\(\mathbb Z_2\)- and \(\mathbb Z_3\)-symmetries
\[
\mathcal A m^{(2)}(\omega z;\tau)\mathcal A^{-1}
=
m^{(2)}(z;\tau)
=
\mathcal B\overline{m^{(2)}(\bar z;\tau)}\mathcal B,
\qquad
z\in\mathbb C\setminus\Gamma^{(2)}.
\]
\end{enumerate}
\end{RHdbarproblem}
\subsection{Long-Time Asymptotics and Error Analysis}
We now focus on the error estimates for the pure $\dbar$ problem
associated with the RH $\dbar$-problem for $m^{(2)}(z;\tau)$.

\begin{lemma}\label{lem:m3-expansion-PC}
Let $c_4,c_5>0$ and suppose that
\(
c_4t^{3/5}\leq x<c_5t,
\
k_0=\frac{x}{2t},
\
\tau=tk_0^2.
\)
Assume that $r_1\in H^1(\mathbb R_+)$,
$r_2\in H^1(\mathbb R_-)$, and
$\|r_1\|_{L^\infty(\mathbb R_+)}\leq\sigma<1$.
Let $S(z;\tau)$ denote the solution of the associated pure RH
problem, and assume that $S$ and $S^{-1}$ are uniformly bounded
on the support of $\mathcal W^{(2)}$.
Then, for sufficiently large $t$, the function
$m^{(3)}(z;\tau)=m^{(2)}(z;\tau)S(z;\tau)^{-1}$
admits the expansion
\begin{equation*}
m^{(3)}(z;\tau)
=
I+\frac{m_{1,z}^{(3)}(x,t)}{z}
+o\left(\frac1z\right),
\qquad z\to\infty,
\end{equation*}
nontangentially to the original contour, where
\begin{equation*}
|m_{1,z}^{(3)}(x,t)|
\leq
C\left(k_0^{1/2}\tau^{-3/4}+\tau^{-1}\right).
\end{equation*}
Equivalently, in the original variable $k=k_0z$,
\[
m^{(3)}(k/k_0;\tau)
=
I+\frac{m_1^{(3)}(x,t)}{k}
+o\left(\frac1k\right),
\]
where
\begin{equation*}
m_1^{(3)}=k_0m_{1,z}^{(3)},
\qquad
|m_1^{(3)}(x,t)|
\leq
C\left(t^{-3/4}+\frac1{tk_0}\right).
\end{equation*}
The coefficient estimates are uniform for
$c_4t^{3/5}\leq x<c_5t$.
\end{lemma}

\begin{proof}
The function $m^{(3)}$ satisfies
\[
\dbar_zm^{(3)}=m^{(3)}\mathcal W,
\qquad
\mathcal W=S\mathcal W^{(2)}S^{-1}.
\]
By the uniform boundedness of $S$ and $S^{-1}$,
\[
|\mathcal W(z;\tau)|
\leq C|\mathcal W^{(2)}(z;\tau)|.
\]
We first estimate $\|\mathcal W\|_{L^1(\mathbb C)}$.

For the terms involving $\dbar_zR_j$, we apply
Proposition~\ref{prop:Rj-PC}.
Since the angular subregions $E_j$ are separated from the
corresponding real-axis boundaries by a fixed positive angle,
\[
e^{-|\operatorname{Re}\vartheta_{21}(\tau;z)|}
\leq e^{-c\tau|z-1|^2},
\qquad z\in E_j,\quad j=1,\ldots,4.
\]
Then, we have
\begin{equation}\label{estimate:angular-PC}
\begin{aligned}
\iint_{V_j}
H_j(z)e^{-|\operatorname{Re}\vartheta_{21}(\tau;z)|}
\,\mathrm dA(z)
\leq C\tau^{-1},
\qquad j=1,\ldots,4.
\end{aligned}
\end{equation}
Thus, Proposition~\ref{prop:Rj-PC},
\eqref{estimate:angular-PC}, and the Cauchy--Schwarz inequality
yield
\[
\begin{aligned}
\iint_{V_1}
|\dbar_zR_1(z)e^{\vartheta_{21}}|\,\mathrm dA(z)
&\leq
C\left(k_0^{1/2}\tau^{-3/4}+\tau^{-1}\right).
\end{aligned}
\]
The estimates in the remaining regions follow by analogous arguments, as in \cite{Yang-2026,Cuccagna-Jenkins-2016}. Consequently, we obtain
\begin{equation}\label{estimate:W-PC}
\|\mathcal W\|_{L^1(\mathbb C)}
\leq
C\left(k_0^{1/2}\tau^{-3/4}+\tau^{-1}\right).
\end{equation}

Furthermore, the equation \eqref{estimate:W-PC} gives
\[
|m_{1,z}^{(3)}(x,t)|
\leq
C\left(k_0^{1/2}\tau^5
{-3/4}+\tau^{-1}\right).
\]
Finally, since $k=k_0z$ and $\tau=tk_0^2$, we have
\[
\begin{aligned}
|m_1^{(3)}(x,t)|
&=k_0|m_{1,z}^{(3)}(x,t)|\\
&\leq
C\left(k_0^{3/2}\tau^{-3/4}+k_0\tau^{-1}\right)
=
C\left(t^{-3/4}+\frac1{tk_0}\right).
\end{aligned}
\]
All coefficient estimates are uniform for
$c_4t^{3/5}\leq x<c_5t$.
\end{proof}

We construct a local parametrix from the parabolic cylinder model in a small disk centered at $z=1$ and obtain the corresponding parametrices near $z=\omega$ and $z=\omega^2$ by symmetry. The remaining RH problem is then treated by small-norm analysis. We omit the details of this construction and refer to \cite{Yang-2026} for analogous arguments. Combining this analysis with the $\dbar$ estimates established in Lemma \ref{lem:m3-expansion-PC}, we obtain the following long-time asymptotic formulas, uniformly in the transition region.
\begin{equation}\label{pq_longtime}
\begin{aligned}
    p(x,t) &= \frac{3^{3/4} \sqrt{\nu_1}}{\sqrt{2 t}} \cos \left(\frac{5 \pi}{12} + \arg r_1(k_0) + \arg \Gamma(i \nu_1) + \sqrt{3} k_0^2 t - \nu_1 \ln \left(6 \sqrt{3} k_0^2 t\right) - \theta_1\right) + \mathcal{O}\left(t^{-3/4}+\frac1{tk_0}\right), \\
    q(x,t) &= -\frac{3^{1/4} \sqrt{\nu_1}}{\sqrt{2 t}} \sin \left(\frac{5 \pi}{12} + \arg r_1(k_0) + \arg \Gamma(i \nu_1) + \sqrt{3} k_0^2 t - \nu_1 \ln \left(6 \sqrt{3} k_0^2 t\right) - \theta_1\right) + \mathcal{O}\left(t^{-3/4}+\frac1{tk_0}\right),
\end{aligned}
\end{equation}
where \( k_0 = x /2t \), \( \Gamma(\cdot) \) is the Gamma function, and \( \nu_1 \) and \( \theta_1 \) are given by
\[
\nu_1 = -\frac{1}{2 \pi} \ln \left(1 - \left|r_1(k_0)\right|^2\right), \quad 
\theta_1 = \frac{1}{\pi} \int_{k_0}^{\infty} \ln \left|\frac{s - k_0}{s - \omega k_0}\right| d \ln \left(1 - \left|r_1(s)\right|^2\right).
\]


\subsection{Matching the Painlev\'e Transition Asymptotics with the Dispersive Wave Asymptotics}

Its and Kapaev \cite{ItsJPA} established the following asymptotic formula for the solution $P_{\mathrm{IV}}(y)$ of equation~\eqref{PIV} as \(y\to-\infty\):
\begin{equation}\label{p4_long_asymp}
P_{\mathrm{IV}}(y) = -\frac{2y}{3} + 2\sqrt{2} a \cos \Theta + \mathcal{O}\left(\frac{1}{y}\right),
\end{equation}
where
\[
\begin{aligned}
    & a^2 = -\frac{1}{2\sqrt{3}\pi} \ln \left(1 - \left|s_{-}\right|^2\right), \quad a > 0, \\
    & \Theta = \frac{y^2}{\sqrt{3}} - \sqrt{3} a^2 \ln \left(2\sqrt{3} y^2\right) + \phi, \\
    & \phi = -\frac{3\pi}{4} - \frac{2\pi}{3}(\alpha - \beta) - \arg \Gamma\left(-\mathrm{i}\sqrt{3} a^2\right) - \arg s_{-}.
\end{aligned}
\]
Substituting \(y=-\frac{\sqrt{3}x}{2\sqrt{t}}\) into \eqref{p4_long_asymp} and setting \(\alpha=-\frac{1}{6}\) and \(\beta=-\frac{2}{3}\), we find that the leading-order term of $q(x,t)$ in \eqref{pq Painleve-t1} has the following asymptotic behavior:
\begin{equation}
\label{q_longtime-mathch}
\begin{aligned}
    q(x, t) &\sim \frac{\sqrt{3}}{4\sqrt{t}} \left( P_{\mathrm{IV}}(y) + \frac{2}{3} y \right) \\
    &\sim \frac{3^{1/4} \sqrt{-\frac{1}{2\pi} \ln(1 - |s_-|^2)}}{\sqrt{2t}} \cos \left( \frac{\sqrt{3}x^2}{4t} + \frac{1}{2\pi} \ln(1 - |s_-|^2) \ln \left(\frac{3\sqrt{3}x^2}{2t}\right) \right. \\
    &\quad \left. - \frac{13\pi}{12} - \arg \Gamma\left(-\mathrm{i} \sqrt{3} a^2\right) - \arg s_- \right).
\end{aligned}
\end{equation}
\par
Recalling that \(k_0=\frac{x}{2t}\) and identifying \(s_-=\overline{r_1(k_0)}\), we obtain from \eqref{q_longtime-mathch} that
\[
q(x, t) \sim -\frac{3^{1/4} \sqrt{\nu_1}}{\sqrt{2t}} \sin \left( \frac{5\pi}{12} + \arg r_1(k_0) + \arg \Gamma(i \nu_1) + \sqrt{3} k_0^2 t - \nu_1 \ln \left(6\sqrt{3}k_0^2 t\right) \right),
\]
which matches the leading-order term in \eqref{pq_longtime} as \(k_0\to0\), using \(\theta(k_0)\to0\) and the conjugation symmetry of the Gamma function. The matching of the leading-order terms for $p(x,t)$ follows by an analogous calculation. Moreover, the error estimates are consistent at the transition scale: when $x\sim t^{3/5}$, the error terms in both \eqref{pq Painleve-t1} and \eqref{pq_longtime} are of order $\mathcal O(t^{-3/5})$.

\begin{remark}
   Although the matching of leading-order terms between (\ref{pq Painleve-t1}) and (\ref{pq_longtime}) is established informally by the asymptotics of Clarkson-McLeod solution \cite{ItsJPA} of the Painlev\'{e} \(\mathrm{IV}\) equation, the asymptotic analysis in \cite{ItsJPA} is based on a second-order system, in contrast to the \( 3 \times 3 \) RH problem \( m^p(y;\lambda) \) described in the Appendix \ref{Appendix}. It remains to derive the asymptotic behaviors of \( m^p(y;\lambda) \) for \( y \to -\infty \).
\end{remark}

\appendix

\section{The Painlev\'e IV Model problem}\label{Appendix}

	Introduce the jump matrices of the RH problem of function $m^{p}(y;\lambda)$ with contour in Figure \ref{SigmaP}, denoted as $v^{p}_j(y;\lambda)$ in the form:
	\begin{equation}\label{jump matrices of mp}
		\begin{aligned}
			&v_1^{p}=e^{3y\lambda \hat{J}+\frac{27}{4}\lambda^2\hat{J^2}}\begin{pmatrix}
				1 & 0 & 0\\
				s_1& 1 &s_3\\
				0 & 0 &1
			\end{pmatrix},
			&&v_6^{p}=e^{3y\lambda \hat{J}+\frac{27}{4}\lambda^2\hat{J^2}}\begin{pmatrix}
				1 & s_2 & s_4\\
				0 & 1 & 0\\
				0 & 0 & 1
			\end{pmatrix},\\
			&v_3^{p}=e^{3y\lambda \hat{J}+\frac{27}{4}\lambda^2\hat{J^2}}\begin{pmatrix}
				1 & s_3 & s_1 \\
				0 & 1 & 0 \\
				0 & 0 & 1 \\
			\end{pmatrix},
			&&v_2^{p}=e^{3y\lambda \hat{J}+\frac{27}{4}\lambda^2\hat{J^2}}\begin{pmatrix}
				1 & 0 & 0 \\
				0 & 1 & 0 \\
				s_2 & s_4 & 1 \\
			\end{pmatrix},\\
			&v_5^{p}=e^{3y\lambda \hat{J}+\frac{27}{4}\lambda^2\hat{J^2}}\begin{pmatrix}
				1 & 0 & 0 \\
				0 & 1 & 0 \\
				s_3 & s_1 & 1 \\
			\end{pmatrix},
			&&v_4^{p}=e^{3y\lambda \hat{J}+\frac{27}{4}\lambda^2\hat{J^2}}\begin{pmatrix}
				1 & 0 & 0 \\
				s_4 & 1 & s_2\\
				0 & 0 & 1 \\
			\end{pmatrix},	
		\end{aligned}\\
	\end{equation}
	and 
	$$
	v_j=I,~j=7,\cdots,12,\quad v_{13}=v_6v_1,\ v_{14}=v_2v_3,\ v_{15}=v_4v_5,
	$$
	where
	$$
	s_1=e^{-2\pi\nu}\rho_1^*(0)=r_1^*(0),\ s_3=-r_2^*(0),\ s_2=-\bar{s_1},\ s_4=-\bar{s_3},\ J=\begin{pmatrix}
        \omega&0&0\\
        0&\omega^2&0\\
        0&0&1
    \end{pmatrix}.
	$$
    
\begin{remark}
Indeed, the above jump matrices can be simplified by a series of equivalent transformations. As a result, the jump conditions of the function $m^p(y;\lambda)$ satisfy the jump matrices in $(\mathrm{\ref{jump matrices of mp}})$, and the jump contour $\Sigma^{(P)}$ is depicted in Figure \ref{SigmaP}. For convenience, it is enough to use the jump conditions as described in Figure \ref{SigmaP}.
	\end{remark}
	\begin{figure}[htp]
		\centering
		\begin{tikzpicture}[>=latex]
			\draw[<->,very thick,rotate=30] (-2,0)node[below]{${4}$} to (2,0)node[above]{${1}$};
			\draw[very thick,rotate=30] (-4,0) to (4,0);
			\draw[<->,very thick,rotate=30] (-1,-1.732)node[right=2mm]{${5}$} to (1,1.732)node[left=2mm]{${2}$};
			\draw[very thick,rotate=30] (-2,-1.732*2) to (2,1.732*2);
			\draw[<->,very thick,rotate=30] (-1,1.732)node[left]{${3}$} to (1,-1.732)node[right]{${6}$};
			\draw[very thick,rotate=30] (-2,1.732*2) to (2,-1.732*2);
			\draw (1.732/2,1/2) arc[start angle=30, end angle=90, radius=1cm];
			\node at(0.8,1.1) {$\frac{\pi}{3}$};
		\end{tikzpicture}
		\caption{The jump contour $\Sigma^{(P)}=\cup_{i=1}^{6}\Sigma_i^{(P)}$, where the angle between rays is $\frac{\pi}{3}$.}
		\label{SigmaP}
	\end{figure}
	\begin{remark}
		Notice that the jump matrices in $(\mathrm{\ref{jump matrices of mp})}$ and the equality $v_1^pv_2^p\cdots v_6^p=I$ imply that
		$$
		s_4+s_1+s_2s_3=0,\quad s_3+s_2+s_1^2+s_1s_2s_3=0,
		$$
		which shows that
           \begin{equation*}
			r_2^*(0)+r_1(0)+r_2(0)r_1^*(0)=0.
		\end{equation*}
        
		On the other hand, the jump matrices satisfy the following $\Z_2$ and $\Z_3$ symmetries
		$$
		v^p(y;\lambda)=\mathcal{B}\left((v^p)^*(y;\lambda)\right)^{-1}\mathcal{B}^{-1}=\mathcal{A}v^p(y;\omega \lambda)\mathcal{A}^{-1}
		$$
		which indicates that
		$$
		m^p(y;\lambda)=\mathcal{B}(m^p)^*(y;\lambda)\mathcal{B}^{-1}=\mathcal{A}m^p(y;\omega \lambda)\mathcal{A}^{-1}.
		$$
	\end{remark}
    
	\begin{RHproblem}
		Find a 3$\times$3 matrix-valued meromorphic function $m^p(y;\lambda)$ to satisfy the following  conditions:
		\begin{enumerate}
		    \item[$(\rm i)$] The matrix-valued function $m^p(y;\lambda)$ is analytic for $\lambda \in \mathbb{C} \backslash \Sigma^{(P)}$;
            \item [$(\rm ii)$]  On $\Sigma^{(P)}$, the jump condition $m_{+}^{p}(y;\lambda )=m_{-}^p(y;\lambda) v^p(y; \lambda)$ holds, where the jump matrices $v^p(y; \lambda)=\left\{v_j^p(y; \lambda)\right\}_{j=1}^6$ are defined in $(\mathrm{\ref{jump matrices of mp})}$;
            \item [$(\rm iii)$] The large $\lambda$ asymptotics of the function $m^p(y; \lambda)$ holds, i.e., $m^p(y; \lambda)=I+O(\frac{1}{\lambda})$ as $\lambda \rightarrow \infty$ and $m^p(y; \lambda)$ is bounded as $\lambda\to 0$ with $\lambda \notin \Sigma^{(P)}$.
            \item [$(\rm iv)$] $m^p(y; \lambda)$ follows the $\mathbb{Z}_2$ and $\mathbb{Z}_3$ symmetries
		\end{enumerate}
		\begin{equation}\label{symmetry of mp}
			\mathcal{A }m^p(y; \omega\lambda) \mathcal{A}^{-1}=m^p(y; \lambda)=\mathcal{B}\overline{ m^p(y; \bar \lambda)} \mathcal{B}, \quad \lambda \in \mathbb{C} \backslash \Sigma^{(P)} .
		\end{equation}
	\end{RHproblem}
	\begin{lemma}
As $\lambda\to\infty$, for any positive integer $N\ge1$, the RH problem for function $m^{p}(y;\lambda)$ admits the following expansion
$$
m^{p}(y;\lambda)=I+
\sum_{j=1}^N\frac{m^{p}_j(y)}{\lambda^j}+\mathcal{O}\left(\frac{1}{\lambda^{N+1}}\right),
$$
where the term $m^{p}_1(y)$ is associated with the solution of the fourth Painlev\'{e} equation in (\ref{PIV}) with parameters $\alpha=-\frac{1}{6},\beta=-\frac{2}{3}$.

	\end{lemma}
	\begin{proof}
		Based on the standard Cauchy integral analysis, the expansion of $m^{p}(y;\lambda)$ as $\lambda\to\infty$ is established. Furthermore, it is sufficient to show the relation between $m^{p}_1(y)$ and Painlev\'{e} IV equation. More precisely, substituting the symmetries in (\ref{symmetry of mp}) into the expansion coefficient $m^p_1(y)$, it follows that
		$$
		m^p_1(y)=\omega^2\mathcal{A}m^p_1(y)\mathcal{A}^{-1}=\mathcal{B}\overline{ m^p_1(y)} \mathcal{B},
		$$
		which implies that
\begin{equation*}
			m^p_1(y)=\begin{pmatrix}
				\omega^2 \varphi_3 & \omega^2 \varphi_1 & \omega^2 \varphi_2\\
				\omega \varphi_2 & \omega \varphi_3 & \omega \varphi_1\\
				\varphi_1 & \varphi_2 & \varphi_3\\
			\end{pmatrix},
		\end{equation*}
		with $\varphi_1=\bar{\varphi}_2$ and $\varphi_3=\bar{\varphi}_3$. Now, let $\Phi(y;\lambda)=m^{p}(y;\lambda)e^{3y\lambda J+\frac{27}{4}\lambda^2 J^2}$ and suppose that
		\begin{equation*}
			\begin{aligned}
				&A(y;\lambda)
				=\Phi_{\lambda}(y;\lambda)\Phi(y;\lambda)^{-1},\\
				&U(y;\lambda)
				=\Phi_{y}(y;\lambda)\Phi(y;\lambda)^{-1}.
			\end{aligned}
		\end{equation*}
		Since both $A(y;\lambda)$ and $U(y;\lambda)$ are  holomorphic functions for $\lambda$, it is seen that
		$$
		\begin{aligned}
			A(y;\lambda)&=A_0(y)+A_1(y)\lambda,\\
			U(y;\lambda)&=U_0(y)+U_1(y)\lambda,\\
		\end{aligned}
		$$
		with $A_1=\frac{27}{2}J^2, A_0=3 y J+\frac{27}{2}[m_1^p,J^2]$ and $U_1=3J, U_0=3 [m_1^p,J]$.
		Finally, the compatibility condition $U_{\lambda}-A_y+[U,A]=0$ implies that
		\begin{equation*}
			\varphi_2'(y)-9 \overline{\varphi_2(y)}^2+\frac{2 i y \varphi_2(y)}{\sqrt{3}}=0.
		\end{equation*}
        \par
		Let $\varphi_2=f(y)+i g(y)$, and thus the equations for functions $f(y)$ and $g(y)$ are gotten as follows:
		\begin{equation*}
			\begin{cases}
				\begin{aligned}
					&18 f(y) g(y)+\frac{2 yf(y) }{\sqrt{3}}+g'(y)=0,\\
					&f'(y)-9 f(y)^2+9 g(y)^2-\frac{2 y g(y)}{\sqrt{3}}=0.
				\end{aligned}
			\end{cases}
		\end{equation*}
        \par
		Suppose that $f=-\frac{\sqrt{3} g'(y)}{2 \left(9 \sqrt{3} g(y)+y\right)}$ and $g=-\frac{ P_{\rm IV}(y)+\frac{2 y}{3}}{6 \sqrt{3}}$, it is immediate to find that the function $P_{\rm IV}(y)$ solves the following ordinary differential equation 
    \begin{equation*}
			-\frac{P''(y)}{ P_{\rm IV}(y)}+\frac{P'(y)^2}{2  P_{\rm IV}(y)^2}+4 y  P_{\rm IV}(y)+\frac{3  P_{\rm IV}(y)^2}{2}-\frac{2}{9  P_{\rm IV}(y)^2}+2 y^2=0,
		\end{equation*}
		which coincides with Painlev\'{e} IV equation (\ref{PIV}) with $\alpha=-\frac{1}{6},\beta=-\frac{2}{3}$.
    \end{proof}
    
\noindent{\bf Conflict of interest declaration.} We declare that we have no competing interests.

\subsection*{Acknowledgements}

The authors would like to express their sincere and profound gratitude to Professor Jonatan Lenells for his invaluable guidance, unwavering support, and insightful discussions, which have been instrumental in the progression and successful completion of this work. We also thank Professor Charlier Christophe for his discussions and suggestions on this manuscript. Additionally, we gratefully acknowledge the generous financial support provided by the National Natural Science Foundation of China (Grant Nos. 12371247 and 12431008), Beijing Natural Science Foundation  (Grant Nos. 1262012 and JQ26004) and Key Project of the Natural Science Foundation of Inner Mongolia Autonomous Region Grant No. 2026ZD036. The last author acknowledges the support of the scholarship provided by the China Scholarship Council (CSC) under Grant No. 202406040149 and the GNFM-INDAM group and the research project Mathematical Methods in NonLinear Physics (MMNLP), Gruppo 4-Fisica Teorica of INFN.

\addtocontents{toc}{\protect\setcounter{tocdepth}{-1}}
\bibliographystyle{amsplain}

\end{document}